\documentclass[10pt, a4wide]{article}
\usepackage[margin=1.9cm]{geometry}

\usepackage[english]{babel}
\usepackage{hyperref}
\usepackage{amssymb,amsmath,amsthm}

\usepackage{placeins}
\usepackage{subfigure}
\usepackage{psfrag}
\usepackage{graphicx}
\usepackage{epstopdf}
\usepackage{mathpazo}
\usepackage{todonotes}

\usepackage{dsfont}
\usepackage{enumitem}
\usepackage{microtype}
\DisableLigatures[f]{encoding = *}

\usepackage[ampersand]{easylist}
\ListProperties(Progressive*=3ex)

\renewcommand{\b}{\beta}
\renewcommand{\d}{\delta}
\newcommand{\e}{\varepsilon}
\newcommand{\h}{\eta}
\renewcommand{\o}{\omega}
\renewcommand{\t}{\tau}
\renewcommand{\l}{\lambda}
\newcommand{\s}{\sigma}

\newcommand{\G}{\Gamma}
\newcommand{\tG}{\widetilde{\Gamma}}
\renewcommand{\L}{\Lambda}
\newcommand{\ev}{e}
\newcommand{\ov}{o}
\newcommand{\cA}{\mathcal{A}}
\newcommand{\cR}{R_{A}(x)}
\newcommand{\cC}{\mathcal{C}}
\newcommand{\cI}{\mathcal{I}}
\newcommand{\cX}{\mathcal{X}}
\newcommand{\cF}{\mathcal{F}}
\newcommand{\cZ}{\mathcal{Z}}
\newcommand{\cM}{\mathcal{M}}
\renewcommand{\ss}{\cX^s}
\newcommand{\ms}{\cX^m}

\newcommand{\N}{\mathbb N}
\newcommand{\R}{\mathbb R}

\newcommand{\binf}{\b \to \infty}
\newcommand{\limb}{\lim_{\binf}}

\newcommand{\ed}{\,{\buildrel d \over =}\,}
\newcommand{\cd}{\xrightarrow{d}}
\newcommand{\xt}{X_t}
\newcommand{\xtn}{\{\xt \}_{t \in \N}}

\newcommand{\rmexp}{\mathrm{Exp}}
\newcommand{\pa}{\partial}

\newcommand{\teo}{\tau^{\mathbf{e}}_{\mathbf{o}}}
\newcommand{\toe}{\tau^{\mathbf{o}}_{\mathbf{e}}}
\newcommand{\ee}{\mathbf{e}}
\newcommand{\oo}{\mathbf{o}}
\newcommand{\E}{\mathbb E}
\newcommand{\pr}[1]{\mathbb P_{\b} \Big ( #1 \Big )}
\newcommand{\TM}{\Theta_{\mathrm{max}}(x,A)}
\newcommand{\Tm}{\Theta_{\mathrm{min}}(x,A)}
\newcommand{\Dm}{\Psi_{\mathrm{min}}(x,A)}
\newcommand{\DM}{\Psi_{\mathrm{max}}(x,A)}
\newcommand{\tha}{\tau^x_{A}}
\newcommand{\st}{~\mathbin{\lvert}~}

\newcommand{\Tha}{\mathrm{T}_{A}(x)}
\newcommand{\FT}{\mathfrak{T}_{A}(x)}
\newcommand{\Ome}{\Omega_{x,A}}
\newcommand{\Vtj}{\Omega^{\mathrm{vtj}}_{x,A}}
\newcommand{\Opt}{\Omega^{\mathrm{opt}}_{x,A}}
\newcommand{\ca}{C_{A}(x)}
\newcommand{\cp}{C^+_{A}(x)}
\newcommand{\xtbb}{\{X^\b_t \}_{t \in \N}}

\newtheorem{thm}{Theorem}[section]
\newtheorem{cor}[thm]{Corollary}
\newtheorem{lem}[thm]{Lemma}
\newtheorem{prop}[thm]{Proposition}

\usepackage[affil-it]{authblk}
\title{Hitting times asymptotics for hard-core interactions on grids} 
\date{\today}
\author[1]{F.R.~Nardi}
\author[1]{A.~Zocca\thanks{Corresponding author. Email address: \texttt{a.zocca@tue.nl}}}
\author[1,2]{S.C.~Borst}
\affil[1]{Department of Mathematics and Computer Science, Eindhoven University of Technology, The Netherlands}
\affil[2]{Department of Mathematics of Networks and Communications, Alcatel-Lucent Bell Labs, NJ}

\begin{document}

\maketitle

\begin{abstract}
We consider the hard-core model with Metropolis transition probabilities on finite grid graphs and investigate the asymptotic behavior of the first hitting time between its two maximum-occupancy configurations in the low-temperature regime. In particular, we show how the order-of-magnitude of this first hitting time depends on the grid sizes and on the boundary conditions by means of a novel combinatorial method. Our analysis also proves the asymptotic exponentiality of the scaled hitting time and yields the mixing time of the process in the low-temperature limit as side-result. In order to derive these results, we extended the model-independent framework in~\cite{MNOS04} for first hitting times to allow for a more general initial state and target subset.
\end{abstract}

\textbf{Keywords:} hard-core model; hitting times; Metropolis Markov chains; finite grid graphs; mixing times; low temperature.

%%%%%%%%%%%%%%%%%%%%%%%%%%%%%%%%%%%%%%%%%%%%%%%%%%%%%%%%%%%%%%%%%%%%%%%%%%%%%%%%%%%%%%%%%%%%%%%%%%%%%%%%%%%%%%%%%%%%%%%%%%%%%%%%%%%%%%%%%%%%%%%%%%%%%%%%%%%%%%%%%%%%%%%%%%%%%%%%%%%%%%%%%%%%%%%%%%%%%%%%%%%%%%%%%%%%%%%%%%%%%%%%
%%%%%%%%%%%%%%%%%%%%%%%%%%%%%%%%%%%%%%%%%%%%%%%%%%%%%%%%%%%%%%%%%%%%%%%%%%%%%%%%%%%%%%%%%%%%%%%%%%%%%%%%%%%%%NEW%SECTION%%%%%%%%%%%%%%%%%%%%%%%%%%%%%%%%%%%%%%%%%%%%%%%%%%%%%%%%%%%%%%%%%%%%%%%%%%%%%%%%%%%%%%%%%%%%%%%%%%%%%%%%
%%%%%%%%%%%%%%%%%%%%%%%%%%%%%%%%%%%%%%%%%%%%%%%%%%%%%%%%%%%%%%%%%%%%%%%%%%%%%%%%%%%%%%%%%%%%%%%%%%%%%%%%%%%%%%%%%%%%%%%%%%%%%%%%%%%%%%%%%%%%%%%%%%%%%%%%%%%%%%%%%%%%%%%%%%%%%%%%%%%%%%%%%%%%%%%%%%%%%%%%%%%%%%%%%%%%%%%%%%%%%%%%

\section{Introduction}
\label{sec1}

{\bf Hard-core lattice gas model.} In this paper we consider a stochastic model where particles in a finite volume dynamically interact subject to hard-core constraints and study the first hitting times between admissible configurations of this model. This model was introduced in the chemistry and physics literature under the name ``hard-core lattice gas model'' to describe the behavior of a gas whose particles have non-negligible radii and cannot overlap~\cite{GF65,vdBS94}. We describe the spatial structure in terms of a finite undirected graph $\L$ of $N$ vertices, which represents all the possible sites where particles can reside. The hard-core constraints are represented by edges connecting the pairs of sites that cannot be occupied simultaneously. We say that a particle configuration on $\L$ is \textit{admissible} if it does not violate the hard-core constraints, i.e.~if it corresponds to an independent set of the graph $\L$.
The appearance and disappearance of particles on $\L$ is modeled by means of a single-site update Markov chain $\xtn$ with Metropolis transition probabilities, parametrized by the fugacity $\l \geq 1$. At every step a site $v$ of $\L$ is selected uniformly at random; if it is occupied, the particle is removed with probability $1/\l$; if instead the selected site $v$ is vacant, then a particle is created with probability $1$ if and only if all the neighboring sites at edge-distance one from $v$ are also vacant. Denote by $\mathcal I(\L)$ the collection of independent sets of $\L$. The Markov chain $\xtn$ is ergodic and reversible with respect to the \textit{hard-core measure with fugacity $\l$ on $\mathcal I(\L)$}, which is defined as
\begin{equation}
\label{eq:hcgb}
\mu_\l(I) := \frac{\l^{|I|}}{ Z_{\l}(\L)}, \quad I \in \mathcal{I}(\L),
\end{equation}
where $Z_{\l}(\L)$ is the appropriate normalizing constant (also called \textit{partition function}). The fugacity $\l$ is related to the \textit{inverse temperature} $\b$ of the gas by the logarithmic relationship $\log \l = \b$. 

We focus on the study of the hard-core model in the \textit{low-temperature} regime where $\l \to \infty$ (or equivalently $\binf$), so that the hard-core measure $\mu_\l$ favors maximum-occupancy configurations. In particular, we are interested in how long it takes the Markov chain $\xtn$ to ``switch'' between these maximum-occupancy configurations. Given a target subset of admissible configurations $A \subset \mathcal{I}(\L)$ and an initial configuration $x \not\in A$, this work mainly focuses on the study of the \textit{first hitting time} $\tha$ of the subset $A$ for the Markov chain $\xtn$ with initial state $x$ at time $t=0$.

\vspace{.2cm}
\noindent{\bf Two more application areas}. The hard-core lattice gas model is thus a canonical model of a gas whose particles have a non-negligible size, and the asymptotic hitting times studied in this paper provide insight into the rigid behavior at low temperatures. Apart from applications in statistical physics, our study of the hitting times is of interest for other areas as well.
The hard-core model is also intensively studied in the area of operations research in the context of communication networks \cite{K91}. In that case, the graph $\L$ represents a communication network where calls arrive at the vertices according to independent Poisson streams. The durations of the calls are assumed to be independent and exponentially distributed. If upon arrival of a call at a vertex $i$, this vertex and all its neighbors are idle, the call is activated and vertex $i$ will be busy for the duration of the call.
If instead upon arrival of the call, vertex $i$ or at least one of its neighbors is busy, the call is lost, hence rendering hard-core interaction. In recent years, extensions of this communication network model received widespread attention, because of the emergence of wireless networks. A pivotal algorithm termed CSMA~\cite{WK05} which is implemented for distributed resource sharing in wireless networks can be described in terms of a continuous-time version of the Markov chain studied in this paper. Wireless devices form a topology and the hard-core constraints represent the conflicts between simultaneous transmissions due to interference~\cite{WK05}. In this context $\L$ is therefore called \textit{interference graph} or \textit{conflict graph}. The transmission of a data packet is attempted independently by every device after a random back-off time with exponential rate $\l$, and, if successful, lasts for an exponentially distributed time with mean $1$. Hence, the regime $\l \to \infty$ describes the scenario where the competition for access to the medium becomes fiercer. The asymptotic behavior of the first hitting times between maximum-occupancy configurations provides fundamental insights into the average packet transmission delay and the temporal starvation which may affect some devices of the network, see~\cite{ZBvLN13}. 

A third area in which our results find application is discrete mathematics, and in particular for algorithms designed to find independent sets in graphs. The Markov chain $\xtn$ can be regarded as a Monte Carlo algorithm to approximate the partition function $Z_{\l}(\L)$ or to sample efficiently according to the hard-core measure $\mu_\l$ for $\l$ large. A crucial quantity to study is then the \textit{mixing time} of such Markov chains, which quantifies how long it takes the empirical distribution of the process to get close to the stationary distribution $\mu_\l$. Several papers have already investigated the mixing time of the hard-core model with Glauber dynamics on various graphs~\cite{BCFKTVV99,G08,Galvin2006a,R06}. By understanding the asymptotic behavior of the hitting times between maximum-occupancy configurations on $\L$ as $\l \to \infty$, we can derive results for the mixing time of the Metropolis hard-core dynamics on $\L$, which in general is smaller than
for the usual Glauber dynamics, as illustrated in~\cite{LEYY12}.

\vspace{.2cm}
\noindent{\bf Results for general graphs}. 
The Metropolis dynamics in which we are interested for the hard-core model can be put, after the identification $e^\b=\l$, in the framework of reversible \textit{Freidlin-Wentzel Markov chains} with Metropolis transition probabilities (see Section~\ref{sec2} for precise definitions). Hitting times for Freidlin-Wentzel Markov chains are central in the mathematical study of metastability. In the literature, several different approaches have been introduced to study the time it takes for a particle system to reach a stable state starting from a metastable configuration. Two approaches have been independently developed based on large deviations techniques: the \textit{pathwise approach}, first introduced in~\cite{CGOV84} and then developed in~\cite{OS95,OS96,OV05}, and the approach in~\cite{C91,C92,C99,CC97}. Other approaches to metastability are the \textit{potential theoretic approach}~\cite{BEGK02,BEGK02b} and, more recently introduced, the \textit{martingale approach}~\cite{BL10,BL11}, see~\cite{CNS14a} for a more detailed review.

In the present paper, we follow the pathwise approach, which has already been used to study many finite-volume models in a low-temperature regime, see~\cite{CN03,CN13,CNS08,HNOS03,HOS00,KO93,NOS05,NS91}, where the state space is seen as an \textit{energy landscape} and the paths which the Markov chain will most likely follow are those with a minimum energy barrier. In~\cite{OS95,OS96,OV05} the authors derive general results for first hitting times for the transition from metastable to stable states, the critical configurations (or bottlenecks) visited during this transition and the tube of typical paths. In~\cite{MNOS04} the results on hitting times are obtained with minimal model-dependent knowledge, i.e.~find all the metastable states and the minimal energy barrier which separates them from the stable states. We extend the existing framework~\cite{MNOS04} in order to obtain asymptotic results for the hitting time $\tha$ for \textit{any} starting state $x$, not necessarily metastable, and \textit{any} target subset $A$, not necessarily the set of stable configurations. In particular, we identify the two crucial exponents $\G_-(x,A)$ and $\G_+(x,A)$ that appear in the upper and lower bounds in probability for $\tha$ in the low-temperature regime. These two exponents might be hard to derive for a given model and, in general, they are not equal. However, we derive a sufficient condition that guarantees that they coincide and also yields the order-of-magnitude of the first moment of $\tha$ on a logarithmic scale. Furthermore, we give another slightly stronger condition under which the hitting time $\tha$ normalized by its mean converges in distribution to an exponential random variable.

\vspace{.2cm}
\noindent{\bf Results for rectangular grid graphs}. We apply these model-independent results to the hard-core model on rectangular grid graphs to understand the asymptotic behavior of the hitting time $\teo$, where $\ee$ and $\oo$ are the two configurations with maximum occupancy, where the particles are arranged in a checkerboard fashion on even and odd sites. Using a novel powerful combinatorial method, we identify the minimum energy barrier between $\ee$ and $\oo$ and prove absence of deep cycles for this model, which allows us to decouple the asymptotics for the hitting time $\teo$ and the study of the critical configurations. In this way, we then obtain sharp bounds in probability for $\teo$, since the two exponents coincide, and find the order-of-magnitude of $\E \teo$ on a logarithmic scale, which depends both on the grid dimensions and on the chosen boundary conditions. In addition, our analysis of the energy landscape shows that the scaled hitting time $\teo / \E \teo$ is exponentially distributed in the low-temperature regime and yields the order-of-magnitude of the mixing time of the Markov chain $\xtn$.

By way of contrast, we also briefly look at the hard-core model on complete $K$-partite graphs, which was already studied in continuous time in~\cite{ZBvL12}. While less relevant from a physical standpoint, the corresponding energy landscape is simpler than that for grid graphs and allows for explicit calculations for the hitting times between any pair of configurations. In particular, we show that whenever our two conditions are not satisfied, $\G_-(x,A) \neq \G_+(x,A)$ and the scaled hitting time is not necessarily exponentially distributed.

\section{Overview and main results}
\label{sec2}
In this section we introduce the general framework of Metropolis Markov chains and show how the dynamical hard-core model fits in it. We then present our two main results for the hitting time $\teo$ for the hard-core model on grid graphs and outline our proof method.

\subsection{Metropolis Markov chains}
\label{sub21}
Let $\cX$ be a finite state space and let $H: \cX \to \R$ be the \textit{Hamiltonian}, i.e.~a non-constant \textit{energy function}. We consider the family of Markov chains $\xtbb$ on $\cX$ with Metropolis transition probabilities $P_\b$ indexed by a positive parameter $\b$
\begin{equation}
\label{eq:mtp}
P_\b(x,y):=
\begin{cases}
q(x,y) e^{-\b [H(y)-H(x)]^+}, & \text{ if } x \neq y,\\
1-\sum_{z \neq x } P_\b(x,z), & \text{ if } x=y,
\end{cases}
\end{equation}
where $q: \cX \times \cX \to [0,1]$ is a matrix that does not depend on $\b$. The matrix $q$ is the \textit{connectivity function} and we assume it to be
\begin{itemize}
\item[$\bullet$] Stochastic, i.e.~$\sum_{y \in \cX} q(x,y) = 1$ for every $x \in \cX$;
\item[$\bullet$] Symmetric, i.e.~$q(x,y)=q(y,x)$ for every $x,y \in \cX$;
\item[$\bullet$] Irreducible, i.e.~for any $x,y \in \cX$, $x \neq y$, there exists a finite sequence $\o$ of states $\o_1,\dots,\o_n \in \cX$ such that $\o_1=x$, $\o_n=y$ and $q(\o_i,\o_{i+1})>0$, for $i=1,\dots, n-1$. We will refer to such a sequence as a \textit{path} from $x$ to $y$ and we will denote it by $\o: x \to y$. 
\end{itemize}
%The matrix $q$ is doubly stochastic, since $q$ is symmetric and stochastic, and thus its stationary distribution is uniform over $\cX$. 
We call the triplet $(\cX, H, q)$ an \textit{energy landscape}. The Markov chain $\xtbb$ is reversible with respect to the \textit{Gibbs measure}
\begin{equation}
\label{eq:gibbs}
\mu_\b(x):=\frac{e^{-\b H(x)}}{\sum_{y \in \cX} e^{-\b H(y)}}.
\end{equation}
Furthermore, it is well-known (see for example~\cite[Proposition 1.1]{C99}) that the Markov chain $\xtbb$ is aperiodic and irreducible on $\cX$. Hence $\xtbb$ is ergodic on $\cX$ with stationary distribution $\mu_\b$. 

For a nonempty subset $A \subset \cX$ and a state $x \in \cX$, we denote by $\tha$ the \textit{first hitting time} of the subset $A$ for the Markov chain $\xtbb$ with initial state $x$ at time $t=0$, i.e.
\[
\tha:=\inf \{ t >0 : X^{\b}_t \in A \st X_0^\b=x\}.
\]

Denote by $\ss$ the set of \textit{stable states} of the energy landscape $(\cX,H,q)$, that is the set of global minima of $H$ on $\cX$, and by $\ms$ the set of \textit{metastable states}, which are the local minima of $H$ in $\cX \setminus \ss$ with maximum stability level (see Section~\ref{sec3} for definition). The first hitting time $\tha$ is often called \textit{tunneling time} when $x$ is a stable state and the target set is some $A \subseteq \ss \setminus \{x\}$, or \textit{transition time from metastable to stable} when $x \in \ms$ and $A = \ss$.

\subsection{The hard-core model}
\label{sub22}
The hard-core model on a finite undirected graph $\L$ of $N$ vertices evolving according to the dynamics described in Section~\ref{sec1} can be put in the framework of Metropolis Markov chains. Indeed, we associate a variable $\s(v)\in \{0,1\}$ with each site $v \in \L$, indicating the absence ($0$) or the presence ($1$) of a particle in that site. Then the hard-core dynamics correspond to the Metropolis Markov chain determined by the energy landscape $(\cX,H,q)$ where
\begin{itemize}
\item[$\bullet$] The state space $\cX \subset \{0,1\}^{\L}$ is the set of \textit{admissible configurations} on $\Lambda$, i.e.~the configurations $\s \in \{0,1\}^\L$ such that $\s(v)\s(w)=0$ for every pair of neighboring sites $v,w$ in $\L$;
\item[$\bullet$] The energy of a configuration $\s \in \cX$ is $H(\s) := - \sum_{v \in \L} \s(v)$;
\item[$\bullet$] The connectivity function $q$ allows only for single-site updates (possibly void), i.e. for any $\s,\s'\in \cX$,
\[
q(\s,\s'):= \begin{cases}
\frac{1}{N}, & \text{if } |\{v \in \L \st \s(v) \neq \s'(v)\}|=1,\\
0, & \text{if } |\{v \in \L \st \s(v) \neq \s'(v)\}| > 1, \\
1 - \sum_{\h \neq \s} q(\s,\h), & \text{if } \s=\s'.
\end{cases}
\]
\end{itemize}
For $\l=e^\b$ the hard-core measure~\eqref{eq:hcgb} on $\L$ is precisely the Gibbs measure~\eqref{eq:gibbs} associated with the energy landscape $(\cX,H,q)$.

Our main focus in the present paper concerns the dynamics of the hard-core model on finite two-dimensional rectangular lattices, to which we will simply refer to as \textit{grid graphs}. More precisely, given two integers $K,L \geq 2$, we will take $\L$ to be a $K \times L$ grid graph with three possible boundary conditions: Toroidal (periodic), cylindrical (semiperiodic) and open. We denote them respectively by $T_{K,L}$, $C_{K,L}$ and $G_{K,L}$. Figure~\ref{fig:threegrids} shows an example of the three possible types of boundary conditions. 

\begin{figure}[!h]
\centering
\subfigure[Open grid $G_{9,7}$]{\includegraphics[scale=0.53]{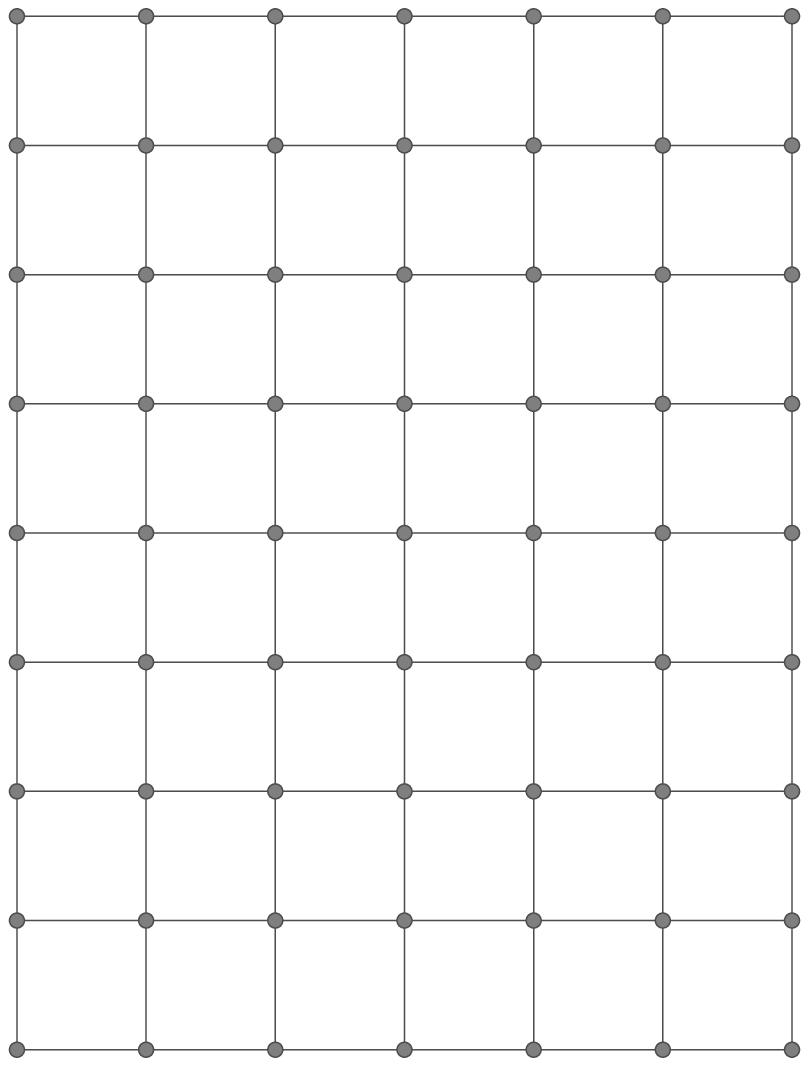}}
\hspace{0.9cm}
\subfigure[Cylindrical grid $C_{8,6}$]{\includegraphics[angle=90,scale=0.64]{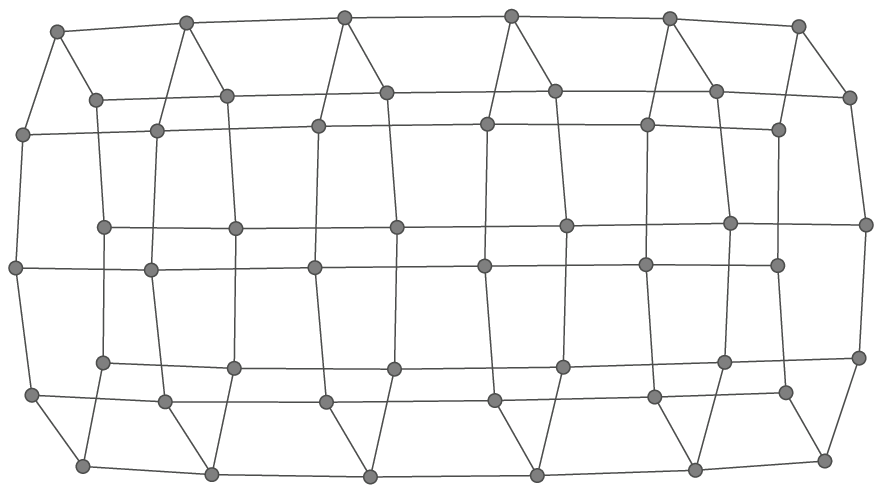}}
\hspace{0.9cm}
\subfigure[Toroidal grid $T_{8,12}$]{\includegraphics[scale=0.35]{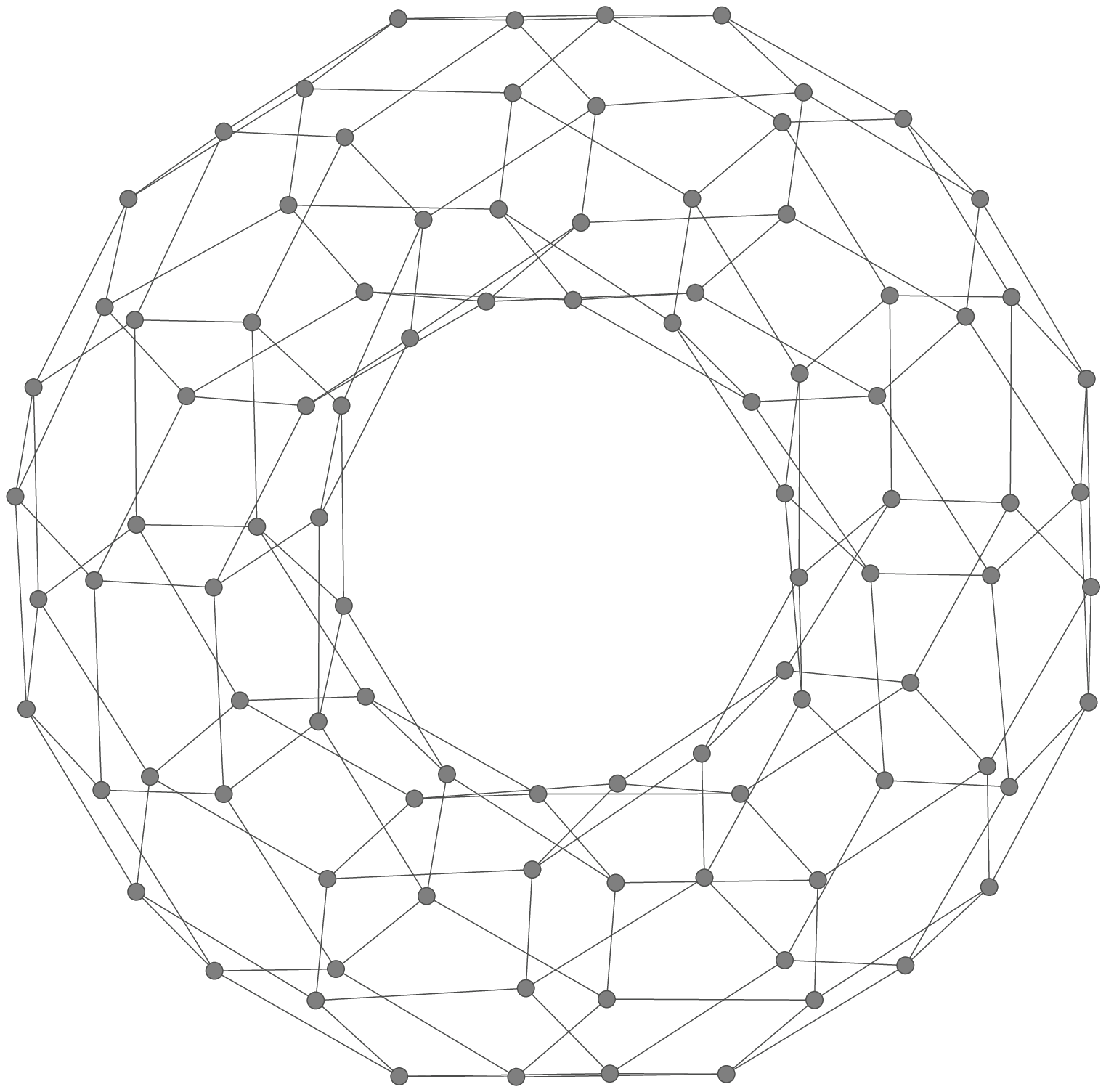}}
\caption{Examples of grid graphs with different boundary conditions}
\label{fig:threegrids}
\end{figure}

There are in total $N = KL$ sites in $\L$. Every site $v \in \L$ is described by its coordinates $(v_1, v_2)$, and since $\L$ is finite, we assume without loss of generality that the leftmost (respectively bottommost) site of $\L$ has the horizontal (respectively vertical) coordinate equal to zero. A site is called \textit{even} (\textit{odd}) if the sum of its two coordinates is even (odd, respectively) and we denote by $V_e$ and $V_o$ the collection of even sites and that of odd sites of $\L$, respectively.

The open grid $G_{K,L}$ is naturally a bipartite graph: All the neighbors in $\L$ of an even site are odd sites and vice versa. In contrast, the cylindrical and toroidal grids may not be bipartite, so that we further assume that $K$ is an even integer for the cylindrical grid $C_{K,L}$ and that both $K$ and $L$ are even integers for the toroidal grid $T_{K,L}$. Since the bipartite structure is crucial for our methodology, we will tacitly work under these assumptions for the cylindrical and toroidal grids in the rest of the paper. As a consequence, $T_{K,L}$ and $C_{K,L}$ are \textit{balanced} bipartite graphs, i.e.~$|V_e| =|V_o|$. The open grid $G_{K,L}$ has $|V_e| = \lceil KL/2 \rceil$ even sites and $|V_o| = \lfloor KL/2 \rfloor$ odd sites, hence it is a balanced bipartite graphs if and only if the product $K L $ is even. We denote by $\ee$ ($\oo$ respectively) the configuration with a particle at each site in $V_e$ ($V_o$ respectively). More precisely,
\[
\ee(v) = \begin{cases}
1 & \text{ if } v \in V_e,\\
0 & \text{ if } v \in V_o,
\end{cases}
\quad \text{ and } \quad 
\oo(v) = \begin{cases}
0 & \text{ if } v \in V_e,\\
1 & \text{ if } v \in V_o.
\end{cases}
\]
Note that $\ee$ and $\oo$ are admissible configurations for any choice of boundary conditions, and that $H(\ee)=-|V_e|=-\lceil KL/2 \rceil$ and $H(\oo)=-|V_o|=-\lfloor KL/2 \rfloor$. In the special case where $\L=G_{K,L}$ with $KL \equiv 1 \pmod 2$, $H(\ee) < H(\oo)$ and, as we will show in Section~\ref{sec5}, $\ss = \{ \ee\}$ and $\ms =\{ \oo\}$. In all the other cases, we have $H(\ee) = H(\oo)$ and $\ss = \{ \ee, \oo\}$; see Section~\ref{sec5} for details.

\subsection{Main results and proof outline}
\label{sub23}
Our first main result describes the asymptotic behavior of the \textit{tunneling time} $\teo$ for any rectangular grid $\L$ in the low-temperature regime $\binf$. In particular, we prove the existence and find the value of an exponent $\G(\L)>0$ that gives an asymptotic control in probability of $\teo$ on a logarithmic scale as $\binf$ and characterizes the asymptotic order-of-magnitude of the mean tunneling time $\E \teo$. We further show that the tunneling time $\teo$ normalized by its mean converges in distribution to an exponential unit mean random variable. 

\begin{thm}[Asymptotic behavior of the tunneling time $\teo$]\label{thm:teo}
Consider the Metropolis Markov chain $\xtbb$ corresponding to hard-core dynamics on a $K \times L$ grid $\L$ as described in Subsection~\ref{sub22}. There exists a constant $\G(\L) >0$ such that
\begin{itemize}[align=left]
\item[{\rm (i)}] For every $\e>0, \quad \limb \pr{ e^{\b (\G(\L)-\e)} < \teo < e^{\b (\G(\L)+\e)}} =1,$
\item[{\rm (ii)}] $\displaystyle \limb \frac{1}{\b} \log \E \teo = \G(\L),$
\item[{\rm (iii)}] $\displaystyle \frac{\teo}{\E \teo} \cd \rmexp(1), \quad \mathrm{as} \, \binf.$
\end{itemize}
In the special case where $\L=G_{K,L}$ with $K L \equiv 1 \pmod 2$, {\rm (i)}, {\rm (ii)}, and {\rm (iii)} hold also for the first hitting time $\toe$, but replacing $\G(\L)$ by $\G(\L)-1$.
\end{thm}

Theorem~\ref{thm:teo} relies on the analysis of the hard-core energy landscape for grid graphs and novel results for hitting times in the general Metropolis Markov chains context. We first explain these new model-independent results and, afterwards, we give details about the properties we proved for the energy landscape of the hard-core model.

The framework~\cite{MNOS04} focuses on the most classical metastability problem, which is the characterization of the transition time $\tau^{\h}_{\ss}$ between a metastable state $\h \in \ms$ and the set of stable states $\ss$. However, the starting configuration for the hitting times we are interested in, is not always a metastable state and the target set is not always $\ss$. In fact, the classical results can be applied for the hard-core model on grids for the hitting time $\toe$ only in the case of an $K \times L$ grid with open boundary conditions and odd side lengths, i.e.~$KL \equiv 1 \pmod 2$. Many other interesting hitting times are not covered by the literature, including:
\begin{itemize}
\item[$\bullet$] The hitting time $\teo$ when $\L$ is a $K \times L$ grid with open boundary conditions and odd side lengths, i.e.~$KL \equiv 1 \pmod 2$, which is a transition from the unique stable state $\ee$ to the metastable state $\oo$;
\item[$\bullet$] The hitting times $\teo \ed \toe$ when $\L$ is an $K \times L$ grid with $KL \equiv 0 \pmod 2$ for \textit{any} boundary conditions, since the configurations $\ee$ and $\oo$ are both stable states;
\vspace{0.1cm}
\item[$\bullet$] The hitting time between any pair of local minima when $\L$ is a complete $K$-partite graph.
\end{itemize}
We therefore generalize the classical pathwise approach~\cite{MNOS04} to study the first hitting time $\tha$ for a Metropolis Markov chain for \textit{any pair of starting state $x$ and target subset $A$}. The interest of extending these results to the tunneling time between two stable states was already mentioned in~\cite{MNOS04,OV05}, but our framework is even more general and we could study $\tha$ for any pair $(x,A)$, e.g.~the transition between a stable state and a metastable one.

Our analysis relies on the classical notion of \textit{cycle}, which is a maximal connected subset of states lying below a given energy level. The exit time from a cycle in the low-temperature regime is well-known in the literature~\cite{C99,CC97,CNS14a,OS95,OV05} and is characterized by the \textit{depth} of the cycle, which is the minimum energy barrier that separates the bottom of the cycle from its external boundary. The usual strategy presented in the literature to study the first hitting time from $x \in \ms$ to $A = \ss$ is to look at the decomposition into maximal cycles of the relevant part of the energy landscape, i.e.~$\cX \setminus \ss$. The first model-dependent property one has to prove is that the starting state $x$ is metastable, which guarantees that there are no cycles in $\cX \setminus \ss$ deeper than the maximal cycle containing the starting state $x$, denoted by $\ca$. In this scenario, the time spent in maximal cycles different from $\ca$, and hence the time it takes to reach $\ss$ from the boundary of $\ca$, is comparable to or negligible with respect to the exit time from $\ca$, making the exit time from $\ca$ and the first hitting time $\tha$ of the same order.

In contrast, for a general starting state $x$ and target subset $A$ \textit{all} the maximal cycles of $\cX \setminus A$ can potentially have a non-negligible impact on the transition from $x$ to $A$ in the low-temperature regime. By analyzing these maximal cycles and the possible \textit{cycle-paths}, we can establish bounds in probability of the hitting time $\tha$ on a logarithmic scale, i.e.~obtain a pair of exponents $\G_-(x,A),\G_+(x,A)$ such that for every $\e>0$
\[
\limb \pr{e^{\b(\G_-(x,A)-\e)} \leq \tha \leq e^{\b(\G_+(x,A)+\e)}} =1.
\]
The sharpness of the exponents $\G_-(x,A)$ and $\G_+(x,A)$ crucially depends on how precisely one can determine which maximal cycles are likely to be visited and which ones are not, see Section~\ref{sec3} for further details. Furthermore, we give a sufficient condition (see Assumption A in Section~\ref{sec3}), which is the \textit{absence of deep typical cycles}, which guarantees that $\G_-(x,A)=\G=\G_+(x,A)$, proving that the random variables $\frac{1}{\b} \log \tha$ converge in probability to $\G$ as $\binf$, and that $\limb \frac{1}{\b} \log \E \tha = \G$. In many cases of interest, one could show that Assumption A holds for the pair $(x,A)$ without detailed knowledge of the typical paths from $x$ to $A$. Indeed, by proving that the model exhibits \textit{absence of deep cycles} (see Proposition~\ref{prop:assa}), similarly to~\cite{MNOS04}, also in our framework the study of the hitting time $\tha$ is decoupled from an exact control of the typical paths from $x$ to $A$. More precisely, one can obtain asymptotic results for the hitting time $\tha$ in probability, in expectation and in distribution without the detailed knowledge of the critical configuration or of the tube of typical paths. Proving the absence of deep cycles when $x \in \ms$ and $A = \ss$ corresponds precisely to identifying the set of metastable states $\ms$, while, when $x \in \ss$ and $A = \ss \setminus \{x\}$, it is enough to show that the energy barrier that separates any state from a state with lower energy is not bigger than the energy barrier separating any two stable states.

Moreover, we give another sufficient condition (see Assumption B in Section~\ref{sec3}), called \textit{``worst initial state''} assumption, to show that the hitting time $\tha$ normalized by its mean converges in distribution to an exponential unit mean random variable. However, checking Assumption B for a specific model can be very involved, and hence we provide a stronger condition (see Proposition~\ref{prop:assb}), which includes the case of the tunneling time between stable states and the classical transition time from a metastable to a stable state. The hard-core model on complete $K$-partite graphs is used as an example to illustrate scenarios where Assumption A or B is violated, $\G_-(x,A) \neq \G_+(x,A)$ and the asymptotic result for $\E \tha$ of the first moment and the asymptotic exponentiality of $\tha/ \E \tha$ do not hold.

In the case of the hard-core model on a rectangular grid $\L$, we develop a powerful combinatorial approach which shows the absence of deep cycles (Assumption A) for this model, concluding the proof of Theorem~\ref{thm:teo}. Furthermore, it yields the value of the energy barrier $\G(\L)$ between $\ee$ and $\oo$, which turns out to depend both on the grid size and on the chosen boundary conditions. This is illustrated by the next theorem, which is our second main result.

\begin{thm}[The exponent $\G(\L)$ for rectangular grids]\label{thm:gamma}
Let $\L$ a $K\times L$ rectangular grid. Then the energy barrier $\G(\L)$ between $\ee$ and $\oo$ appearing in Theorem~\ref{thm:teo} takes the values 
\[
\G(\L)=
\begin{cases}
\min \{ K, L\} +1 & \text{ if } \L = T_{K,L},\\
\min \{ \lceil K/2 \rceil, \lceil L/2 \rceil \} +1 & \text{ if } \L = G_{K,L},\\
\min \{ K/2, L\} +1 & \text{ if } \L = C_{K,L}.
\end{cases}
\]
\end{thm}

The crucial idea behind the proof of Theorem~\ref{thm:gamma} is that along the transition from $\ee$ to $\oo$, there must be a critical configuration where for the first time an entire row or an entire column coincides with the target configuration $\oo$. In such a critical configuration particles reside both in even and odd sites and, due to the hard-core constraints, an interface of empty sites should separate particles with different parities. By quantifying the ``inefficiency'' of this critical configuration we get the minimum energy barrier that has to be overcome for the transition from $\ee$ to $\oo$ to occur. The proof is then concluded by exhibiting a path that achieves this minimum energy and by exploiting the absence of other deep cycles in the energy landscape.

Lastly, we show that by understanding the global structure of an energy landscape $(\cX, H, q)$ and the maximum depths of its cycles, we can also derive results for the mixing time of the corresponding Metropolis Markov chains $\xtbb$, as illustrated in Subsection~\ref{sub38}. In particular, our results show that in the special case of an energy landscape with multiple stable states and without other deep cycles, the hitting time between any two stable states and the mixing time of the chain are of the same order-of-magnitude in the low-temperature regime. This is the case also for the Metropolis hard-core dynamics on grids, see Theorem~\ref{thm:mixgap} in Section~\ref{sec5}.

The rest of the paper is structured as follows. Section~\ref{sec3} is devoted to the model-independent results valid for a general Metropolis Markov chain, which extend the classical framework~\cite{MNOS04}. The proofs of these results are rather technical and therefore deferred to Section~\ref{sec4}. In Section~\ref{sec5} we develop our combinatorial approach to analyze the energy landscapes corresponding to the hard-core model on grids. We finally present in Section~\ref{sec6} our conclusions and indicate future research directions.

%%%%%%%%%%%%%%%%%%%%%%%%%%%%%%%%%%%%%%%%%%%%%%%%%%%%%%%%%%%%%%%%%%%%%%%%%%%%%%%%%%%%%%%%%%%%%%%%%%%%%%%%%%%%%%%%%%%%%%%%%%%%%%%%%%%%%%%%%%%%%%%%%%%%%%%%%%%%%%%%%%%%%%%%%%%%%%%%%%%%%%%%%%%%%%%%%%%%%%%%%%%%%%%%%%%%%%%%%%%%%%%%
%%%%%%%%%%%%%%%%%%%%%%%%%%%%%%%%%%%%%%%%%%%%%%%%%%%%%%%%%%%%%%%%%%%%%%%%%%%%%%%%%%%%%%%%%%%%%%%%%%%%%%%%%%%%%NEW%SECTION%%%%%%%%%%%%%%%%%%%%%%%%%%%%%%%%%%%%%%%%%%%%%%%%%%%%%%%%%%%%%%%%%%%%%%%%%%%%%%%%%%%%%%%%%%%%%%%%%%%%%%%%
%%%%%%%%%%%%%%%%%%%%%%%%%%%%%%%%%%%%%%%%%%%%%%%%%%%%%%%%%%%%%%%%%%%%%%%%%%%%%%%%%%%%%%%%%%%%%%%%%%%%%%%%%%%%%%%%%%%%%%%%%%%%%%%%%%%%%%%%%%%%%%%%%%%%%%%%%%%%%%%%%%%%%%%%%%%%%%%%%%%%%%%%%%%%%%%%%%%%%%%%%%%%%%%%%%%%%%%%%%%%%%%%

\section{Asymptotic behavior of hitting times for Metropolis Markov chains}
\label{sec3}
In this section we present model-independent results valid for any Markov chains with Metropolis transition probabilities~\eqref{eq:mtp} defined in Subsection~\ref{sub21}. In Subsection~\ref{sub31} we introduce the classical notion of a \textit{cycle}. 
If the considered model allows only for a very rough energy landscape analysis, well-known results for cycles are shown to readily yield upper and lower bounds in probability for the hitting time $\tha$: indeed, one can use the depth of the initial cycle $\ca$ as $\G_-(x,A)$ (see Propositions~\ref{prop:plow}) and the maximum depth of a cycle in the partition of $\cX \setminus A$ as $\G_+(x,A)$ (see Proposition~\ref{prop:pup}). 
If one has a good handle on the model-specific \textit{optimal paths} from $x$ to $A$, i.e.~those paths along which the maximum energy is precisely the min-max energy barrier between $x$ and $A$, sharper exponents can be obtained, as illustrated in Proposition~\ref{prop:plowup1}, by focusing on the \textit{relevant cycle}, where the process $\xtbb$ started in $x$ spends most of its time before hitting the subset $A$. 
We even further sharpen these bounds in probability for the hitting time $\tha$ with Proposition~\ref{prop:plowup2} by studying the \textit{tube of typical paths from $x$ to $A$} or \textit{standard cascade}, a task that in general requires a very detailed but local analysis of the energy landscape. 
To complete the study of the hitting time in the regime $\binf$, we prove in Subsection~\ref{sub35} the convergence of the first moment of the hitting time $\tha$ on a logarithmic scale under suitable assumptions (see Theorem~\ref{thm:l1}) and give in Subsection~\ref{sub36} sufficient conditions for the scaled hitting time $\tha / \E \tha$ to converge in distribution as $\binf$ to an exponential unit mean random variable, see Theorem~\ref{thm:ae}. Furthermore, we illustrate in detail two special cases which fall within our framework, namely the classical transition from a metastable state to a stable state and the tunneling between two stable states, which is the relevant one for the model considered in this paper. In Subsection~\ref{sub37} we briefly present the hard-core model on a complete $K$-partite graph, which is an example of a model where the asymptotic exponentiality of the scaled hitting times does not always hold. Lastly, in Subsection~\ref{sub38} we present some results for the mixing time and the spectral gap of Metropolis Markov chains and show how they are linked with the critical depths of the energy landscape.

In the rest of this section and in Section~\ref{sec4}, $\xtn$ will denote a general Metropolis Markov chain with energy landscape $(\cX, H, q)$ and inverse temperature $\b$, as defined in Subsection~\ref{sub21}.

\subsection{Cycles: Definitions and classical results}
\label{sub31}
We recall here the definition of cycle and present some well-known properties.

Recall that a path $\o: x \to y$ has been defined in Subsection~\ref{sub21} as a finite sequence of states $\o_1,\dots,\o_n \in \cX$ such that $\o_1=x$, $\o_n=y$ and $q(\o_i,\o_{i+1})>0$, for $i=1,\dots, n-1$. Given a path $\o=(\o_1,\dots,\o_n)$ in $\cX$, we denote by $|\o|:=n$ its \textit{length} and define its \textit{height} or \textit{elevation} by
\begin{equation}
\label{eq:defheight}
\Phi_\o:=\max_{i=1,\dots,|\o|} H(\o_i).
\end{equation}
A subset $A \subset \cX$ with at least two elements is \textit{connected} if for all $x,y \in A$ there exists a path $\o: x \to y$, such that $\o_i \in A$ for every $i=1,\dots,|\o|$.
Given a nonempty subset $A \subset \cX$ and $x \not \in A$, we define $\Ome$ as the collection of all paths $\o: x \to y$ for some $y \in A$ that do not visit $A$ before hitting $y$, i.e.
%such that $\o_{i} \not\in A$ for every $i < |\o|$
\begin{equation}
\label{eq:defOme}
\Ome:=\{\o: x \to y \st y \in A, \, \, \o_{i} \not\in A \, \, \forall \, i < |\o|\}.
\end{equation}
We remark that only the endpoint of each path in $\Ome$ belongs to $A$. The \textit{communication energy} between a pair $x,y\in\cX$ is the minimum value that has to be reached by the energy in every path $\o: x \to y$, i.e.
\begin{equation}
\label{eq:ch}
\Phi(x,y) := \min_{\o : x\to y} \Phi_\o. %\max_{z\in\o} H(z). 
\end{equation}
%If there exists no path from $x$ to $y$ in $\cX$, we set by convention $\Phi(x,y):=\infty$.
Given two nonempty disjoint subsets $A,B \subset \cX$, we define the communication energy between $A$ and $B$ by
\begin{equation}
\label{eq:chAB}
\Phi(A,B) := \min_{x \in A, y\in B} \Phi(x,y).
\end{equation}
Given a nonempty set $A \subset\cX$, we define its \textit{external boundary} by
\[
\pa A := \{y \notin A \st \exists \, x\in A ~:~ q(x,y)>0 \}.
\]
For a nonempty set $A \subset\cX$ we define its \textit{bottom} $\cF(A)$ as the set of all minima of the energy function $H(\cdot)$ on $A$, i.e.
\[
\cF(A) := \{y\in A : H(y)=\min_{x \in A} H(x)\}.
\]
Let $\ss:=\cF(\cX)$ be the set of \textit{stable states}, i.e.~the set of states with minimum energy. Since $\cX$ is finite, the set $\ss$ is always nonempty. Define the \textit{stability level} $\mathcal V_x$ of a state $x \in \cX$ by
\begin{equation}
\label{eq:sl}
\mathcal V_x := \Phi(x,\cI_{x}) - H(x),
\end{equation}
where $\cI_{x}:=\{z \in \cX : H(z)<H(x)\}$ is the set of states with energy lower than $x$. We set $\mathcal V_x:=\infty$ if $\cI_x$ is empty, i.e.~when $x$ is a stable state. The set of \textit{metastable states} $\ms$ is defined as 
\begin{equation}
\label{eq:ms}
\ms:=\{x\in\cX : \mathcal V_x = \max_{z \in \cX \setminus \ss } \mathcal V_z\}.
\end{equation}
We call a nonempty subset $C \subset \cX$ a \textit{cycle} if it is either a singleton or it is a connected set such that
\begin{equation}
\label{eq:defcycle}
\max_{x \in C} H(x)< H(\cF(\pa C)).
\end{equation}
A cycle $C$ for which condition~\eqref{eq:defcycle} holds is called \textit{non-trivial cycle}. If $C$ is a non-trivial cycle, we define its \textit{depth} as
\begin{equation}
\label{eq:defdepth}
\G(C):=H(\cF(\pa C)) - H(\cF(C)).
\end{equation}
Any singleton $C=\{x\}$ for which condition~\eqref{eq:defcycle} does not hold is called \textit{trivial cycle}. We set the depth of a trivial cycle $C$ to be equal to zero, i.e.~$\G(C)=0$. Given a cycle $C$, we will refer to the set $\cF(\pa C)$ of minima on its boundary as its \textit{principal boundary}. Note that 
\[
\Phi(C,\cX \setminus C)=
\begin{cases}
H(x) & \text{ if } C=\{x\} \text{ is a trivial cycle,}\\
H(\cF(\pa C)) & \text{ if } C \text{ is a non-trivial cycle.}
\end{cases}
\]
In this way, we have the following alternative expression for the depth of a cycle $C$, which has the advantage of being valid also for trivial cycles:
\begin{equation}
\label{eq:Gequiv}
\G(C)= \Phi(C,\cX \setminus C) - H(\cF(C)).
\end{equation}
The next lemma gives an equivalent characterization of a cycle.
\begin{lem}
\label{lem:cycleequiv}
A nonempty subset $C \subset \cX$ is a cycle if and only if it is either a singleton or it is connected and satisfies
\[
\max_{x,y \in C} \Phi(x,y)< \Phi(C,\cX\setminus C).
\]
\end{lem}
The proof easily follows from definitions~\eqref{eq:ch}, \eqref{eq:chAB} and~\eqref{eq:defcycle} and the fact that if $C$ is not a singleton and is connected, then
\begin{equation}
\label{eq:height}
\max_{x,y \in C} \Phi(x,y) = \max_{x \in C} H(x).
\end{equation}

We remark that the equivalent characterization of a cycle given in Lemma~\ref{lem:cycleequiv} is the ``correct definition'' of a cycle in the case where the transition probabilities are not necessarily Metropolis but satisfy the more general \textit{Friedlin-Wentzell condition}
\begin{equation}
\label{eq:fw}
\limb -\frac{1}{\b} \log P_{\b}(x,y) = \Delta(x,y) \quad \forall\, x,y \in \cX,
\end{equation}
where $\Delta(x,y)$ is an appropriate \textit{rate function} $\Delta: \cX^2 \to \R^+ \cup \{\infty\}$. The Metropolis transition probabilities correspond to the case (see~\cite{CNS14b} for more details) where 
\[
\Delta(x,y)=
\begin{cases}
[H(y) - H(x)]^+ & \text{ if } q(x,y)>0,\\
\infty & \text{ otherwise.}
\end{cases}
\]

The next theorem collects well-known results for the asymptotic behavior of the exit time from a cycle as $\b$ becomes large, where the depth $\G(C)$ of the cycle plays a crucial role. 

\begin{thm}[Properties of the exit time from a cycle]\label{thm:exitcycle}%~\cite{MNOS04,OS95,OV05}
Consider a non-trivial cycle $C \subset \cX$.
\begin{enumerate}
\item[{\rm(i)}] For any $x \in C$ and for any $\e>0$, there exists $k_1>0$ such that for all $\b$ sufficiently large
\[
\pr{\t^{x}_{\pa C} < e^{\b (\G(C) - \e)} }\leq e^{-k_1 \b}.
\]
\item[{\rm(ii)}] For any $x \in C$ and for any $\e>0$, there exists $k_2>0$ such that for all $\b$ sufficiently large
\[
\pr{\t^{x}_{\pa C} > e^{\b (\G(C) + \e)} }\leq e^{-e^{k_2 \b}}.
\]
\item[{\rm(iii)}] For any $x,y \in C$, there exists $k_3>0$ such that for all $\b$ sufficiently large
\[
\pr{\t^{x}_{y} > \t^{x}_{\pa C} }\leq e^{-k_3 \b}.
\]
\item[{\rm(iv)}] There exists $k_4>0$ such that for all $\b$ sufficiently large
\[
\sup_{x \in C} \pr{X_{\t^{x}_{\pa C}} \not\in \cF(\pa C)} \leq e^{- k_4 \b}.
\]
\item[{\rm(v)}] For any $x \in C$, $\e>0$ and $\e'>0$, for all $\b$ sufficiently large
\[
\pr{\t^{x}_{\pa C} < e^{\b (\G(C) + \e)}, \, X_{\t^{x}_{\pa C}} \in \cF(\pa C) } \geq e^{- \e' \b}.
\]
\item[{\rm(vi)}] For any $x \in C$, any $\e>0$ and all $\b$ sufficiently large
\[
e^{\b(\G(C) - \e)} < \E \t^{x}_{\pa C} < e^{\b (\G(C) + \e)}.
\]
\end{enumerate}
\end{thm}
The first three properties can be found in~\cite[Theorem 6.23]{OV05}, the fourth one is~\cite[Corollary 6.25]{OV05} and the fifth one in~\cite[Theorem 2.17]{MNOS04}. The sixth property is given in~\cite[Proposition 3.9]{OS95} and implies that
\begin{equation}\label{eq:meanexittime}
\limb \frac{1}{\b} \log \E \t^{x}_{\pa C} = \G(C).
\end{equation}
The third property states that, given that $C$ is a cycle, for any starting state $x \in C$, the Markov chain $\xtn$ visits any state $y \in C$ before exiting from $C$ with a probability exponentially close to one. This is a crucial property of the cycles and in fact can be given as alternative definition, see for instance~\cite{C99,CC97}. The equivalence of the two definitions has been proved in~\cite{CNS14b} in greater generality for a Markov chain satisfying the Friedlin-Wentzell condition~\eqref{eq:fw}. Leveraging this fact, many properties and results from~\cite{C99} will be used or cited.\\

We denote by $\cC(\cX)$ the set of cycles of $\cX$. The next lemma, see~\cite[Proposition 6.8]{OV05}, implies that the set $\cC(\cX)$ has a tree structure with respect to the inclusion relation, where $\cX$ is the root and the singletons are the leafs. 

\begin{lem}[Cycle tree structure]\label{lem:tree}
Two cycles $C, C' \in \cC(\cX)$ are either disjoint or comparable for the inclusion relation, i.e.~$C \subseteq C'$ or $C' \subseteq C$. 
\end{lem}

Lemma~\ref{lem:tree} also implies that the set of cycles to which a state $x \in \cX$ belongs is totally ordered by inclusion. Furthermore, we remark that if two cycles $C,C' \in \cC(\cX)$ are such that $C \subseteq C'$, then $\G(C) \leq \G(C')$; this latter inequality is strict if and only if the inclusion is strict.

%%%%%%%%%%%%%%%%%%%%%%%%%%%%%%%%%%%%%%%%%%%%%%%%%%%%%%%%%%%%%%%%%%%%%%%%%%%%%%%%%%%%%%%%%%%%%%%%%%%%%%%%%%%%%%%%%%%%%%%%%%%%%%%%%%%%%%%%%%%%%%%%%%%%%%%%%%%%%%%%%%%%%%%%%%%%%%%%%%%%%%%%%%%%%%%%%%%%%%%%%%%%%%%%%%%%%%%%%%%%%%%%

\subsection{Classical bounds in probability for hitting time~\texorpdfstring{$\tha$}{}}
\label{sub32}

In this subsection we start investigating the first hitting time $\tha$. Thus, we will tacitly assume that the \textit{target set} $A$ is a nonempty subset of $\cX$ and the \textit{initial state} $x$ belongs to $\cX \setminus A$. Moreover, without loss of generality, we will henceforth assume that
\begin{equation}
\label{eq:wloga}
A= \{ y \in \cX \st \forall \, \o: x \to y \quad \o \cap A \neq \emptyset \},
\end{equation}
which means that we add to the original target subset $A$ all the states in $\cX$ that cannot be reached from $x$ without visiting the subset $A$. Note that this assumption does not change the distribution of the first hitting time $\tha$, since the states which we may have added in this way could not have been visited without hitting the original subset $A$ first.

Given a nonempty subset $A \subset \cX$ and $x \in \cX$, we define the \textit{initial cycle} $\ca$ by
\begin{equation}
\label{eq:defca}
\ca:=\{x\} \cup \{z \in \cX : \Phi(x,z) < \Phi(x,A)\}.
\end{equation}
If $x \in A$, then $\ca = \{x\}$ and thus is a trivial cycle. If $x \not\in A$, the subset $\ca$ is either a trivial cycle (when $\Phi(x,A) = H(x)$) or a non-trivial cycle containing $x$, if $\Phi(x,A) > H(x)$. 
%In both cases, we have 
%\[
%\Phi(\ca,\cX \setminus \ca) = \Phi(x,A).
%\]
In any case, if $x \not\in A$, then $\ca \cap A = \emptyset$. For every $x\in \cX$, we denote by $\G(x,A)$ the depth of the initial cycle $\ca$, i.e.
\[
\G(x,A):=\G(\ca).
\]
Clearly if $\ca$ is trivial (and in particular when $x \in A$), then $\G(x,A)= 0$. Note that by definition the quantity $\G(x,A)$ is always non-negative, and in general
\[
\G(x,A) = \Phi(x,A) - H(\cF(\ca)) \geq \Phi(x,A)-H(x),
\]
with equality if and only if $x \in \cF(\ca)$.\\ 

If $x \not\in A$, then the initial cycle $\ca$ is, by construction, the maximal cycle (in the sense of inclusion) that contains the state $x$ and has an empty intersection with $A$. Therefore any path $\o: x \to A$ has at some point to exit from $\ca$, by overcoming an energy barrier not smaller than its depth $\G(x,A)$. The next proposition gives a probabilistic bound for the hitting time $ \tha$ by looking precisely at this \textit{initial ascent} up until the boundary of $\ca$.
\begin{prop}[Initial-ascent bound]\label{prop:plow}
Consider a nonempty subset $A \subset \cX$ and $x \not\in A$. For any $\e>0$ there exists $\kappa >0$ such that for $\b$ sufficiently large
\begin{equation}\label{eq:plow}
\pr{ \tha < e^{\b (\G(x,A)-\e)}} < e^{-\kappa \b}.
\end{equation}
\end{prop}
The proof is essentially adopted from~\cite{OV05} and follows easily from Theorem~\ref{thm:exitcycle}(i), since by definition of $\ca$, we have that $\tha \geq_\textrm{st} \t^x_{\pa \ca}$.\\

Before stating an upper bound for the tail probability of the hitting time $\tha$, we need some further definitions. Given a nonempty subset $B \subset \cX$, we denote by $\cM(B)$ the \textit{collection of maximal cycles that partitions $B$}, i.e.
\begin{equation}
\label{eq:defcM}
\cM(B):= \{ C \in \cC(\cX) ~:~ C \subseteq B, \, C \text{ maximal} \}.
\end{equation}
Lemma~\ref{lem:tree} implies that every nonempty subset $B \subset \cX$ has a partition into maximal cycles and hence guarantees that $\cM(B)$ is well defined. Note that if $C \in \cC(\cX)$ is itself a cycle, then $\cM(C)=\{C\}$. The importance of the notion of initial cycle besides Proposition~\ref{prop:plow} is partially explained by the next lemma.
\begin{lem}{\rm ~\cite[Lemma 2.26]{MNOS04}}\label{lem:partition}
Given a nonempty subset $A \subset \cX$, the collection $\{\ca\}_{x \in \cX \setminus A}$ of initial cycles is the partition into maximal cycles of $\cX \setminus A$, i.e.
\[
\cM(\cX \setminus A) = \{\ca\}_{x \in \cX \setminus A}.
\]
\end{lem}
We can extend the notion of depth to subsets $B \subsetneq \cX$ which are not necessarily cycles by using the partition of $B$ into maximal cycles. More precisely, we define the \textit{maximum depth $\tG(B)$} of a nonempty subset $B \subsetneq \cX$ as the maximum depth of a cycle contained in $B$, i.e.
\begin{equation}
\label{eq:deftG}
\tG(B) := \max_{C \in \cM(B)} \G(C).
\end{equation}
Trivially $\tG(C) = \G(C)$ if $C \in \cC(\cX)$. The next lemma gives two equivalent characterizations of the maximum depth $\tG(B)$ of a nonempty subset $B \subsetneq \cX$.
\begin{lem}[Equivalent characterizations of the maximum depth]\label{lem:GAequiv1}
Given a nonempty subset $B \subsetneq \cX$,
\begin{equation}\label{eq:GAequiv1}
\tG(B) = \max_{x \in B} \G(x,\cX \setminus B)= \max_{x \in B} \Big \{ \min_{y \in \cX \setminus B}  \Phi(x,y) - H(x) \Big \}.
\end{equation}
\end{lem}
In view of Lemma~\ref{lem:GAequiv1}, $\tG(B)$ is the maximum initial energy barrier that the process started inside $B$ possibly has to overcome to exit from $B$. As illustrated by the next proposition, one can get a (super-)exponentially small upper bound for the tail probability of the hitting time $\tha$, by looking at the maximum depth $\tG(\cX \setminus A)$ of the complementary set $\cX \setminus A$, where the process resides before hitting the target subset $A$.

\begin{prop}[Deepest-cycle bound]\label{prop:pup}~{\rm \cite[Proposition 4.19]{C99}}
Consider a nonempty subset $A \subsetneq \cX$ and $x \not\in A$. For any $\e>0$ there exists $\kappa' >0$ such that for $\b$ sufficiently large
\begin{equation}\label{eq:pup}
\pr{ \tha > e^{\b (\tG(\cX \setminus A)+\e)}} < e^{-e^{\kappa' \b}}.
\end{equation}
\end{prop}

By definition we have $\G(x,A) \leq \tG(\cX \setminus A)$, but in general $\G(x,A) \neq \tG(\cX \setminus A)$ and neither bound presented in this subsection is actually tight, so we will proceed to establish sharper but more involved bounds in the next subsection.

%%%%%%%%%%%%%%%%%%%%%%%%%%%%%%%%%%%%%%%%%%%%%%%%%%%%%%%%%%%%%%%%%%%%%%%%%%%%%%%%%%%%%%%%%%%%%%%%%%%%%%%%%%%%%%%%%%%%%%%%%%%%%%%%%%%%%%%%%%%%%%%%%%%%%%%%%%%%%%%%%%%%%%%%%%%%%%%%%%%%%%%%%%%%%%%%%%%%%%%%%%%%%%%%%%%%%%%%%%%%%%%%

\subsection{Optimal paths and refined bounds in probability for hitting time~\texorpdfstring{$\tha$}{}}
\label{sub33}
The quantity $\G(x,A)$ appearing in Proposition~\ref{prop:plow} only accounts for the energy barrier that has to be overcome starting from $x$, but there is such an energy barrier for every state $z \not\in A$ and it may well be that to reach $A$ it is inevitable to visit a state $z$ with $\G(z,A) > \G(x,A)$. Similarly, also the exponent $\tG(\cX \setminus A)$ appearing in Proposition~\ref{prop:pup} may not be sharp in general. For instance, the maximum depth $\tG(\cX \setminus A)$ could be determined by a deep cycle $C$ in $\cX \setminus A$ that cannot be visited before hitting $A$ or that is visited with a vanishing probability as $\binf$. In this subsection, we refine the bounds given in Propositions~\ref{prop:plow} and~\ref{prop:pup} by using the notion of \textit{optimal path} and identifying the subset of the state space $\cX$ in which these optimal paths lie.

Given a nonempty subset $A \subset \cX$ and $x \not\in A$, define the \textit{set of optimal paths} $\Opt$ as the collection of all paths $\o \in \Ome$ along which the maximum energy $\Phi_\o$ is equal to the communication height between $x$ and $A$, i.e.
\begin{equation}\label{eq:defopt}
\Opt := \{ \o \in \Ome ~:~ \Phi_\o = \Phi(x,A) \}.
\end{equation}
Define the \textit{relevant cycle} $\cp$ as the minimal cycle in $\cC(\cX)$ such that $\ca \subsetneq \cp$, i.e.
\begin{equation}
\label{def:cp}
\cp:= \min\{ C \in \cC(\cX) \st \ca \subsetneq C \}. 
\end{equation}
The cycle $\cp$ is well defined, since the cycles in $\cC(\cX)$ that contain $x$ are totally ordered by inclusion, as remarked after Lemma~\ref{lem:tree}. By construction, $\cp \cap A \neq \emptyset$ and thus $\cp$ contains at least two states, so it has to be a non-trivial cycle. The minimality of $\cp$ with respect to the inclusion gives that 
\[
\max_{z \in \cp} H(z)=\Phi(x,A),
\]
and then, by using Lemma~\ref{lem:cycleequiv}, one obtains
\begin{equation}\label{eq:c+phi}
\Phi(x,A) < H(\cF(\pa \cp)).
\end{equation}

The choice of the name \textit{relevant cycle} for $\cp$ comes from the fact that all paths the Markov chain will follow to go from $x$ to $A$ will almost surely not exit from $\cp$ in the limit $\binf$. Indeed, for the relevant cycle $\cp$ Theorem~\ref{thm:exitcycle}(iii) reads
\begin{equation}
\label{eq:noexit}
\limb \pr{\tha < \tau^{x}_{\pa \cp}} =1.
\end{equation}
The next lemma, which is proved in Section~\ref{sec4}, states that an optimal path from $x$ to $A$ is precisely a path from $x$ to $A$ that does not exit from $\cp$. 
\begin{lem}[Optimal path characterization]\label{lem:equivopt}
Consider a nonempty subset $A \subset \cX$ and $x \not\in A$. Then
\[
\o \in \Opt \quad \Longleftrightarrow \quad \o \in \Ome \quad \mathrm{and} \quad \o \subseteq \cp.\]
\end{lem}
Lemma~\ref{lem:equivopt} implies that the relevant cycle $\cp$ can be equivalently defined as
\begin{equation}
\label{eq:cpequiv}
\cp= \Big \{ y \in \cX \st \Phi(x,y) \leq \Phi(x,A) \Big \} = \Big \{ y \in \cX ~:~ \Phi(x,y) < \Phi(x,A) + \d_0/2 \Big \},
\end{equation}
where $\d_0$ is the minimum energy gap between an optimal and a non-optimal path from $x$ to $A$, i.e.
\[
\d_0=\d_0(x,A):= \min_{\o \in \Ome \setminus \Omega^{\mathrm{opt}}_{x,A}} \Phi_\o - \Phi(x,A).
\]
In view of Lemma~\ref{lem:equivopt} and~\eqref{eq:noexit}, the Markov chain started in $x$ follows in the limit $\binf$ almost surely an optimal path in $\Opt$ to hit $A$. It is then natural to define the following quantities for a nonempty subset $A \subset \cX$ and $x \not\in A$: 
\begin{equation}
\label{eq:defDm}
\Dm:=\min_{\o \in \Opt} \max_{z \in \o} \G(z,A),
\end{equation}
and 
\begin{equation}
\label{eq:defDM}
\DM:=\max_{\o \in \Opt} \max_{z \in \o} \G(z,A).
\end{equation}
Definition~\eqref{eq:defDm} implies that every optimal path $\o \in \Opt$ has to enter at some point a cycle in $\cM(\cX \setminus A)$ of depth at least $\Dm$, while definition~\eqref{eq:defDM} means that every cycle visited by any optimal path $\o \in \Opt$ has depth less than or equal to $\DM$.

An equivalent characterization for the energy barrier $\DM$ can be given, but we first need one further definition. Define $\cR$ as the subset of states which belong to at least one optimal path in $\Opt$, i.e.
\begin{equation}
\label{eq:defcR}
\cR:=\{ y \in \cX \st \exists \, \o \in \Opt ~:~ y \in \o\}.
\end{equation}
Note that $A \cap \cR \neq \emptyset$, since the endpoint of each path in $\Ome$ belongs to $A$, by definition~\eqref{eq:defOme}. In view of Lemma~\ref{lem:equivopt}, $\cR \subseteq \cp$. We remark that this latter inclusion could be strict, since in general $\cR \neq \cp$. Indeed, there could exist a state $y \in \cp$ such that all paths $\o: x \to y$ that do not exit from $\cp$ always visit the target set $A$ before reaching $y$, and thus they do not belong to $\Opt$ (see definitions~\eqref{eq:defOme} and~\eqref{eq:defopt}), see Figure~\ref{fig:cR}.
\vspace{-0.2cm}
\begin{figure}[!ht]
\centering
\subfigure[The subset $\cR$ (in light gray)]{
\includegraphics[scale=0.9]{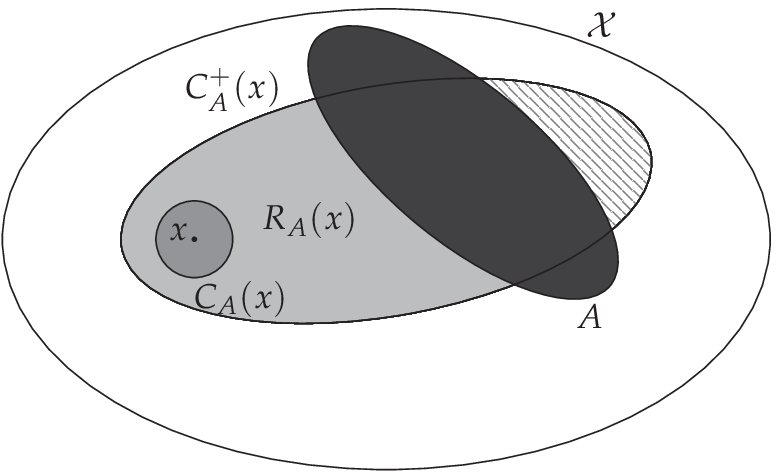}
}
\hspace{0.2cm}
\subfigure[The partition into maximal cycles of $\cR$, including the initial cycle $\ca$ (in dark gray)]{
\includegraphics[scale=0.9]{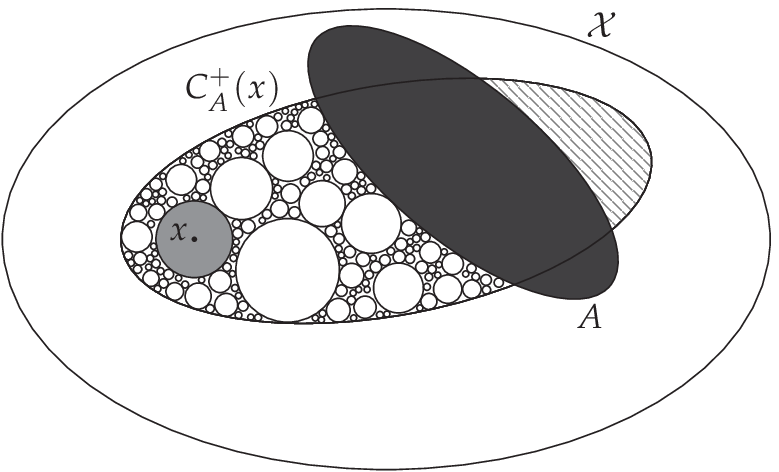}
}
\caption{Example of an energy landscape $\cX$ with highlighted the subset $A$ (in black), the relevant cycle $\cp$ and the subset $\cp \setminus (\cR \cup A)$ (with diagonal mesh)}
\label{fig:cR}
\end{figure}
\FloatBarrier
The next lemma characterizes the quantity $\DM$ as the maximum depth of the subset $\cR \setminus A$ (see definition~\eqref{eq:deftG}).
\begin{lem}[Equivalent characterization of $\DM$]\label{lem:equivDM}
\begin{equation}
\DM= \tG(\cR \setminus A).
\end{equation}
\end{lem}

Using the two quantities $\Dm$ and $\DM$, we can better control in probability the hitting time $\tha$, as stated in the next proposition, which is proved in Section~\ref{sec4}.

\begin{prop}[Optimal paths depth bounds]\label{prop:plowup1}
Consider a nonempty subset $A \subset \cX$ and $x \in \cX \setminus A$. For any $\e>0$ there exists $\kappa>0$ such that for $\b$ sufficiently large
\begin{equation}
\label{eq:plow1}
\pr{\tha < e^{\b (\Dm-\e)}} < e^{- \kappa \b},
\end{equation}
and
\begin{equation}
\label{eq:pup1}
\pr{\tha > e^{\b (\DM+\e)}} < e^{- \kappa \b}.
\end{equation}
\end{prop}
This proposition is in fact a sharper result than Propositions~\ref{prop:plow} and~\ref{prop:pup}, since
\begin{equation}
\label{eq:phiineq}
\G(x,A) \leq \Dm \leq \DM \leq \tG(\cX \setminus A).
\end{equation}
Indeed, since the starting state $x$ trivially belongs to every path in $\Opt$, we have that $\G(x,A) \leq \max_{z \in \o} \G(z,A)$ for every $\o \in \Opt$ and thus $\G(x,A) \leq \Dm$. Furthermore, since by definition $\cp \setminus A \subseteq \cX \setminus A$, Lemma~\ref{lem:equivDM} yields that $\DM \leq \tG(\cX \setminus A)$.\\

It follows from~\eqref{eq:phiineq} that, if $\G(x,A) = \tG(\cX \setminus A)$, then $\Dm =\DM$. However, in general, the exponents $\Dm$ and $\DM$ are not equal and may not be sharp either, as illustrated in the fictitious energy landscape in Figure~\ref{fig:notsharp}.
\vspace{-0.3cm}
\begin{figure}[!ht]
\centering
\subfigure[Energy profile of the energy landscape with the initial cycle $\ca$ (in grey) and the relevant cycle $\cp$ (below the dashed black line)]{\vspace{-1cm} \includegraphics[scale=0.87]{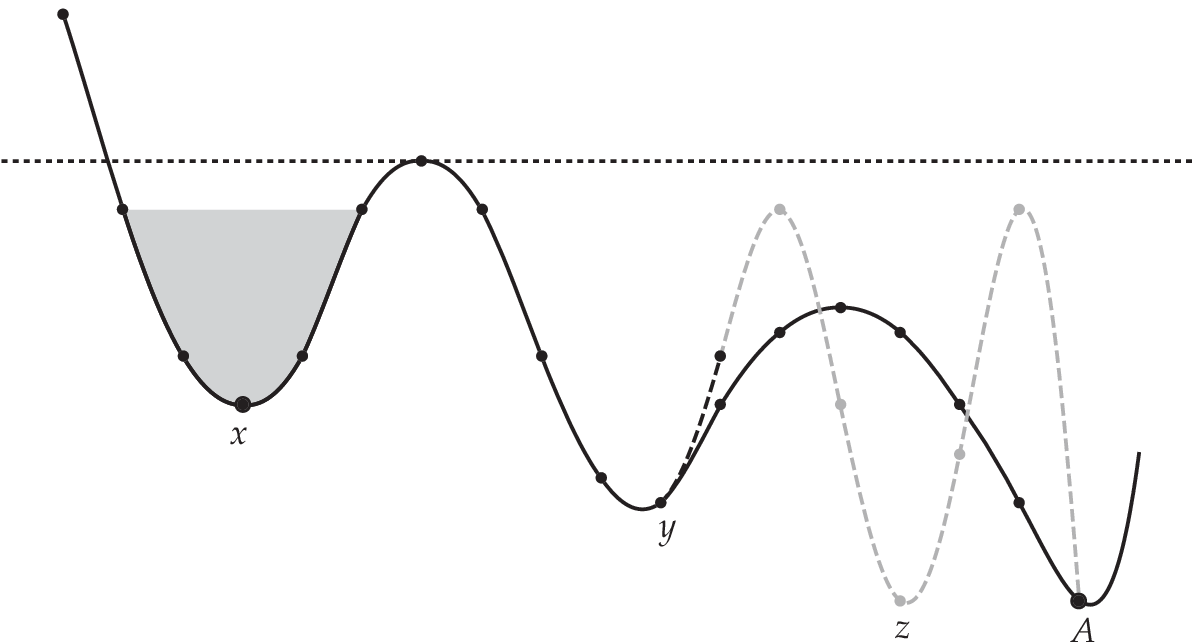}}
\\
\subfigure[Partition into maximal cycles of $\cX\setminus A$ for the same energy landscape]{\hspace{-0.5cm} \includegraphics[scale=0.86]{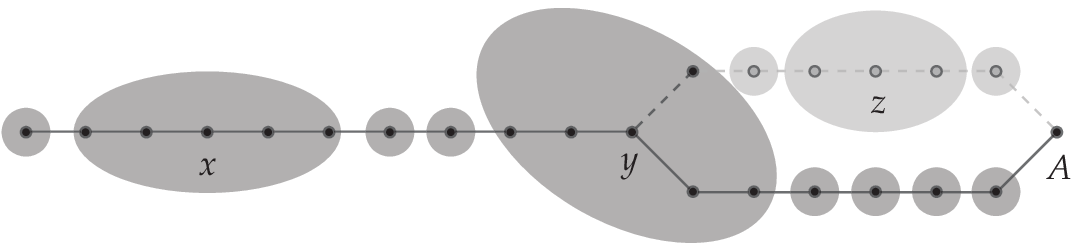}}
\caption{An example energy landscape for which $\DM$ is not sharp}
\label{fig:notsharp}
\end{figure}
\FloatBarrier 
In this example, there are two paths to go from $x$ to $A$: The path $\o$ which goes from $x$ to $y$ and then follows the solid path until $A$, and the path $\o'$, which goes from $x$ to $y$ and then follows the dashed path through $z$ and eventually hitting $A$. Note that $\Phi_{\o} = \Phi_{\o'} = \Phi(x,A)$, so both $\o$ and $\o'$ are optimal paths from $x$ to $A$. By inspection, we get that $\DM=\G(z,A)$. However, the path $\o'$ does not exit the cycle $\cC_{A}(y)$ passing by its principal boundary and, in view of Theorem~\ref{thm:exitcycle}(iv), it becomes less likely than the other path as $\binf$. In fact, the transition from $x$ to $A$ is likely to occur on a smaller time-scale than suggested by the upper bounds in Proposition~\ref{prop:plowup1} and in particular the exponent $\DM$ is not sharp in this example.

In the next subsection, we will show that a more precise control in probability of the hitting time $\tha$ is possible, at the expense of a more involved analysis of the energy landscape.

\subsection{Sharp bounds for hitting time~\texorpdfstring{$\tha$}{} using typical paths}
\label{sub34}
As illustrated at the end of the previous subsection, the exponents $\Dm$ and $\DM$ which control in probability the hitting time $\tha$ may not be sharp in general. In this subsection we obtain exponents that are potentially sharper than $\Dm$ and $\DM$ by looking in more detail at the cycle decomposition of $\cp \setminus A$ and by identifying inside it the \textit{tube of typical paths from $x$ to $A$}. In particular, we focus on how the process moves from two maximal cycles in the partition of $\cp \setminus A$ and determine which of these transitions between maximal cycles are the most likely ones. 

Some further definitions are needed. We introduce the notions of \textit{cycle-path} and a way of mapping every path $\o \in \Ome$ into a cycle-path $\cC_\o$. Recall that for a nonempty subset $A \subset \cX$, $\pa A$ is its external boundary and $\cF(A)$ is its bottom, i.e.~the set of the minima of the energy function $H$ in $A$. A \textit{cycle-path} is a finite sequence $(C_1,\dots, C_m)$ of (trivial and non-trivial) cycles $C_1,\dots, C_m \in \cC(\cX)$ such that 
\[
C_{i} \cap C_{i+1} = \emptyset \quad \text{ and } \quad \pa C_i \cap C_{i+1} \neq \emptyset, \quad \text{ for every } i=1,\dots,m-1.
\]
It can be easily proved that, in a cycle-path $(C_1,\dots, C_m)$, if $C_i$ is a non-trivial cycle for some $i=1,\dots,m$, then its predecessor $C_{i-1}$ and successor $C_{i+1}$ (if any) are trivial cycles, see~\cite[Lemma 2.5]{CNS14a}. We can consider the collection $\mathcal{P}_{x,A}$ of cycle-paths that lead from $x$ to $A$ and consist of maximal cycles in $\cX \setminus A$ only, namely
\begin{equation}
\label{eq:defcyclepath}
\mathcal{P}_{x,A} := \{\text{cycle-path } (C_1,\dots, C_m) \st \, C_1,\dots,C_m \in \cM(\cX \setminus \cA), \, x \in C_1, \, \pa C_m \cap A \neq \emptyset \}.
\end{equation}

We define a mapping $\Ome \to \mathcal{P}_{x,A}$ by assigning to a path $\o=(\o_1,\dots, \o_n) \in \Ome$ the cycle-path $\cC_\o=(C_1,\dots,C_{m(\o)}) \in \mathcal{P}_{x,A}$ as follows. Set $t_0 = 1$, $C_1 = \ca$, and then define recursively
\[
t_i := \min\{ k > t_{i-1} \st \o_{k} \not \in C_i \} \quad \text{ and } \quad C_{i+1}:=C_{A}(\o_{t_i}).
\]
The path $\o$ is a finite sequence and $\o_n \in A$, so there exists an index $m(\o) \in \N$ such that $\o_{t_{m(\o)}}=\o_n \in A$ and there the procedure stops. The way the sequence $(C_1,\dots,C_{m(\o)})$ is constructed shows that it is indeed a cycle-path. Moreover, by using the notion of initial cycle $C_A(\cdot)$ to define $C_1,\dots,C_{m(\o)}$, they are automatically maximal cycles in $\cM(\cX\setminus A)$. Lastly, the fact that $\o \in \Ome$ implies that $x \in C_1$ and that $\pa C_{m(\o)} \cap A \neq \emptyset$, hence $\cC_\o \in \mathcal{P}_{x,A}$ and the mapping is well-defined. We remark that this mapping is not injective, since two different paths in $\Ome$ can be mapped into the same cycle-path in $\mathcal{P}_{x,A}$. In fact, a single cycle-path groups together all the paths that visit the same cycles (the same number of times and in the same order). Cycle-paths are the correct mesoscopic objects to investigate while studying the transition $x \to A$: Indeed one neglects in this way the microscopic dynamics of the process and focuses only on the relevant mesoscopic transitions from one maximal cycle to another.

Furthermore, we note that for a given path $\o \in \Ome$, the maximum energy barrier along $\o$ is the maximum depth in its corresponding cycle-path $\cC_\o$, i.e.
\[
\max_{z \in \o} \G(z,A) = \max_{C \in \cC_\o} \G(C).
\]
We say that a cycle-path $(C_1,\dots, C_m)$ is \textit{connected via typical jumps to} $A$ or simply {\rm vtj}\textit{-connected} to $A$ if 
\[
\cF(\pa C_i) \cap C_{i+1} \neq \emptyset, \quad \text{ for every } i=1,\dots,m-1, \quad \text{ and } \quad \cF(\pa C_m) \cap A \neq \emptyset.
\]
The next lemma, presented in~\cite{CNS14b}, guarantees that there always exists a cycle-path from the initial cycle $C_A(x)$ that is vtj-connected to $A$ for any nonempty target subset $A \subset \cX$ and $x \not\in A$.

\begin{lem}{\rm \cite[Proposition 3.22]{CNS14b}}\label{lem:exitvtjcp}
For any nonempty subset $A \subset \cX$ and $x \not\in A$, there exists a cycle-path $\cC^*=(C_1,\dots,C_{m^*})$ {\rm vtj}-connected to $A$ with $x \in C_1$ and $C_1, \dots, C_m^* \subset \cX \setminus A$.
\end{lem}
By inspecting the proof of~\cite[Proposition 3.22]{CNS14b}, one notices that the given cycle-path $\cC^*=(C_1,\dots,C_{m^*})$ consists only of maximal cycles in $\cX \setminus A$, i.e.~$C_1,\dots, C_{m^*} \in \cM(\cX \setminus A)$, and in particular $C_1 = C_A(x)$. Hence $\cC^* \in \mathcal{P}_{x,A}$ and therefore the collection $\mathcal{P}_{x,A}$ is not empty.

We define $\o \in \Ome$ to be a \textit{typical path from $x$ to $A$} if its corresponding cycle-path $\cC_\o$ is {\rm vtj}-connected to $A$, and we denote by $\Vtj$ the collection of all typical paths from $x$ to $A$, i.e.
\[
\Vtj:=\{\o \in \Ome \st \o \, \textrm{ is typical}\}.
\]
The existence of a {\rm vtj}-connected cycle-path $\cC^*=(C_1,\dots,C_{m^*}) \in \mathcal{P}_{x,A}$ guarantees that 
\[
\Vtj \neq \emptyset.
\]
Indeed, take $y_0=x$, $y_i \in \cF(\pa C_i) \cap C_{i+1}$, $i=1,\dots, m^*-1$ and $y_{m^*} \in \cF(C_{m^*}) \cap A$ and consider a path $\o^*$ that visits precisely the saddles $y_0, \dots, y_{m^*}$ in this order and stays in cycle $C_i$ between the visit to $y_{i-1}$ and $y_i$. Then $\o^*$ is a typical path from $x$ to $A$. The next lemma gives an equivalent characterization for a typical path from $x$ to $A$.
\begin{lem}[Equivalent characterization of a typical path]\label{lem:equivvtj}
Consider a nonempty subset $A \subset \cX$ and $x \not\in A$. Then
\[
\o \in \Vtj \quad \Longleftrightarrow \quad \o \in \Ome \quad \mathrm{and} \quad \Phi(\o_{i+1},A) \leq \Phi(\o_i,A) \quad \forall \, i=1,\dots, |\o|.
\]
\end{lem}
The proof of this result is presented in Section~\ref{sec4}. In particular, Lemma~\ref{lem:equivvtj} shows that every typical path from $x$ to $A$ is an optimal path from $x$ to $A$, i.e.
\begin{equation}\label{eq:vtjopt}
\Vtj \subseteq \Opt,
\end{equation}
since if $\o \in \Vtj$, then $\Phi(\o_i,A) \leq \Phi(\o_1,A)=\Phi(x,A)$ for every $i=2,\dots,|\o|$ and thus $\Phi_\o = \Phi(x,A)$.

Let $\Tha$ be the \textit{tube of typical paths from $x$ to $A$}, which is defined as
\begin{equation}
\label{eq:defTha}
\Tha:=\{y \in \cX \st \exists \, \o \in \Vtj ~:~ y \in \o\}.
\end{equation}
In other words, $\Tha$ is the subset of states $y \in \cX$ that can be reached from $x$ by means of a typical path which does not enter $A$ before visiting $y$. The endpoint of every path in $\Vtj$ belongs to $A$, thus $\Tha \cap A \neq \emptyset$. Since by~\eqref{eq:vtjopt} every typical path is an optimal path, it follows from definitions~\eqref{eq:defcR} and~\eqref{eq:defTha} that 
\[
\Tha \subseteq \cR.
\]
The tube of typical paths can be visualized as the \textit{standard cascade} emerging from state $x$ and reaching eventually $A$, in the sense that it is the part of the energy landscape that would be wet if a water source is placed at $x$ and the water would ``find its way'' until the sink, that is subset $A$. This standard cascade possibly consists of basins/lakes (non-trivial cycles), saddle points (trivial cycles) and waterfalls (trivial cycles). From definition~\eqref{eq:defTha}, it follows that if $z \in \Tha$, then 
\begin{equation}
\label{eq:pzph}
\mathrm{T}_{A}(z) \subseteq \Tha.
\end{equation}

\begin{figure}[!ht]
\centering
\includegraphics[scale=1]{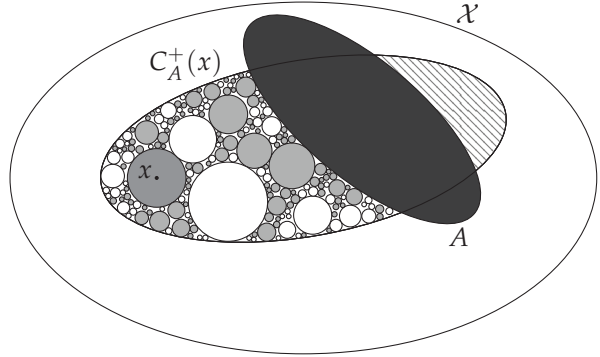}
\caption{Example of an energy landscape with the tube of typical $\Tha$ highlighted in gray}
\label{fig:Tha}
\end{figure}

Denote by $\FT$ the collection of cycles $C \in \cM(\cX \setminus A)$ for which there exists a {\rm vtj}-connected cycle-path $C_1,\dots,C_n \subset \cX \setminus A$ with $C_1=\ca$ and $C_n=C$, i.e.
\[
\FT:=\{C \in \cM(\cX \setminus A) \st \exists \, C_1,\dots,C_n \text{ {\rm vtj}-connected cycle-path with } C_1=\ca \text{ and } C_n = C\}.
\]
Note that the cycles in $\FT$ form the partition into maximal cycles of $\Tha \setminus A$, i.e.
\[
\FT = \cM(\Tha \setminus A),
\]
and that, by construction, there exists $C \in \FT$ such that $\cF(\pa C) \cap A \neq \emptyset$. The boundary of $\Tha$ consists of either states in $A$ or of states which belong to the non-principal part of the boundary of a cycle $C \in \FT$, that is $\pa C \setminus \cF(\pa C)$. In other words,
\begin{equation}
\label{eq:boundaryTha}
\pa \Tha \setminus A = \bigcup_{C \in \FT} (\pa C \setminus \cF(\pa C)).
\end{equation}
The typical paths in $\Vtj$ are the only ones with non-vanishing probability of being visited by the Markov chain $\xtn$ started in $x$ before hitting $A$ in the limit $\binf$, as illustrated by the next lemma which is proved in Section~\ref{sec4}.
\begin{lem}[Exit from the typical tube $\Tha$]\label{lem:exittube}
Consider a nonempty subset $A \subset \cX$ and $x \not\in A$. Then there exists $\kappa>0$ such that for $\b$ sufficiently large
\[
\pr{\t^x_{\pa \Tha} \leq \tha}\leq e^{- \kappa \b}.
\]
\end{lem}
Given a nonempty subset $A \subset \cX$ and $x \not\in A$, define the following quantities: 
\begin{equation}
\label{eq:defTm}
\Tm:=\min_{\o \in \Vtj} \max_{z \in \o} \G(z,A),
\end{equation}
and
\begin{equation}
\label{eq:defTM}
\TM:=\max_{\o \in \Vtj} \max_{z \in \o} \G(z,A).
\end{equation}
In other words, definition~\eqref{eq:defTm} means that every typical path $\o \in \Vtj$ has to enter at some point a cycle of depth at least $\Tm$. On the other hand, definition~\eqref{eq:defDM} implies that all cycles visited by any typical path $\o \in \Vtj$ have depth less than or equal to $\TM$. Hence, $\TM$ can equivalently be characterized as the maximum depth (see definition~\eqref{eq:deftG}) of the tube $\Tha$ of typical paths from $x$ to $A$, as stated by the next lemma.
\begin{lem}[Equivalent characterization of $\TM$]\label{lem:equivTM}
\begin{equation}
\label{eq:TMequiv}
\TM = \tG(\Tha \setminus A) = \max_{C \in \FT} \G(C).
\end{equation}
\end{lem}
Since by~\eqref{eq:vtjopt} every typical path from $x$ to $A$ is an optimal path from $x$ to $A$, definitions~\eqref{eq:defDm},~\eqref{eq:defDM},~\eqref{eq:defTm} and~\eqref{eq:defTM} imply that
\begin{equation}
\label{eq:nestedineq}
\Dm \leq \Tm \leq \TM \leq \DM.
\end{equation}
We now have all the ingredients needed to formulate the first refined result for the hitting time $\tha$, which is proved in Section~\ref{sec4}. The next proposition yields a control in probability for the hitting time $ \tha$ by looking at the \textit{shallowest-typical gorge} inside $\Tha$ that the process has to overcome to reach $A$ and at the \textit{deepest-typical gorge} inside $\Tha$ where the process has a non-vanishing probability to be trapped before hitting $A$.

\begin{prop}[Typical-cycles bounds]\label{prop:plowup2}
Consider a nonempty subset $A \subset \cX$ and $x \not\in A$. For any $\e>0$ there exists $\kappa >0$ such that for $\b$ sufficiently large
\begin{equation}\label{eq:plow2}
\pr{ \tha < e^{\b (\Tm-\e)}} < e^{-\kappa \b},
\end{equation}
and
\begin{equation}\label{eq:pup2}
\pr{ \tha > e^{\b (\TM+\e)}} < e^{-\kappa \b}.
\end{equation}
\end{prop}
The proof, which is a refinement of that of Proposition~\ref{prop:plowup1}, is presented in Section~\ref{sec4}. 

In general, the exponents $\Tm$ and $\TM$ may not be equal, as illustrated in the fictitious energy landscape in Figure~\ref{fig:notequal}. 
\begin{figure}[ht!]
\centering
\subfigure[Energy profile of the energy landscape with the initial cycle $\ca$ (in grey) and the relevant cycle $\cp$ (below the dashed black line)]{\includegraphics[scale=0.9]{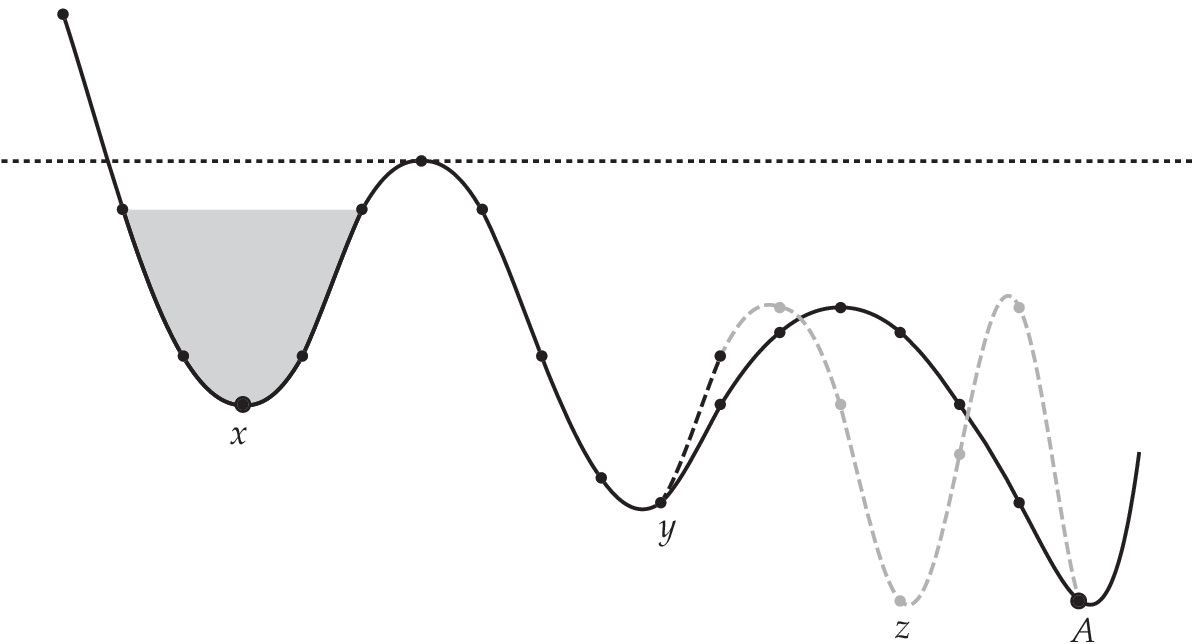}}
\\
\subfigure[Partition into maximal cycles of $\cX\setminus A$ for the same energy landscape]{\hspace{-0.5cm} \includegraphics[scale=0.89]{Dcycles.eps}}
\caption{An example energy landscape for which $\Tm < \TM$}
\label{fig:notequal}
\end{figure}
\FloatBarrier
Also in this example, there are two paths to go from $x$ to $A$: The path $\o$ which goes from $x$ to $y$ and then follows the solid path until $A$, and the path $\o'$, which goes from $x$ to $y$ and then follows the dashed path through $z$ and eventually hitting $A$. Both paths $\o$ and $\o'$ always move from a cycle to the next one visiting the principal boundary, hence they are both typical paths from $x$ to $A$. By inspection, we get that $\TM=\G(z,A)$, since the typical path $\o'$ visits the cycle $\cC_A(z)$. Using the path $\o$ we deduce that $\Tm = \G(y,A)$ and therefore $\Tm < \TM$.

If the two exponents $\Tm$ and $\TM$ coincide, then Proposition~\ref{prop:plowup2} gives a sharp control in probability for the hitting time $\tha$, as stated in the next corollary.
\begin{cor}\label{cor:cp}
Consider a nonempty subset $A \subset \cX$ and $x \not\in A$. Assume that 
\begin{equation}\label{eq:cpa}
\Tm = \Theta(x,A) =  \TM.
\end{equation}
Then, for any $\e>0$
\begin{equation}\label{eq:cp}
\limb \pr{ e^{\b (\Theta(x,A)-\e)} < \tha < e^{\b (\Theta(x,A)+\e)}} =1.
\end{equation}
\end{cor}

There are many examples of models and pairs $(x,A)$ for which $\Tm=\TM$. The most classical ones are the models that exhibit a \textit{metastable behavior}: If one takes $x \in \ms$ and $A = \ss$, then it follows that $\Tm = \mathcal V_x = \TM$ (recall the definition~\eqref{eq:sl} of stability level) and Corollary~\ref{cor:cp} holds, see also~\cite[Theorem 4.1]{MNOS04}.

%%%%%%%%%%%%%%%%%%%%%%%%%%%%%%%%%%%%%%%%%%%%%%%%%%%%%%%%%%%%%%%%%%%%%%%%%%%%%%%%%%%%%%%%%%%%%%%%%%%%%%%%%%%%%%%%%%%%%%%%%%%%%%%%%%%%%%%%%%%%%%%%%%%%%%%%%%%%%%%%%%%%%%%%%%%%%%%%%%%%%%%%%%%%%%%%%%%%%%%%%%%%%%%%%%%%%%%%%%%%%%%%

\subsection{First moment convergence}
\label{sub35}

We now turn our attention to the asymptotic behavior of the mean hitting time $\E \tha $ as $\binf$. In particular, we will show that it scales (almost) exponentially in $\b$ and we will identify the corresponding exponent. There may be some sub-exponential pre-factors, but, without further assumptions, one can only hope to get results on a logarithmic scale, due to the potential complexity of the energy landscape. We remark that a precise knowledge of the tube of typical paths is not always necessary to derive the asymptotic order of magnitude of the mean hitting time $\E \tha$, as illustrated by Proposition~\ref{prop:assa}.

To prove the convergence of the quantity $\frac{1}{\b} \log \E \tha $, we need the following assumption.\\

\noindent \textbf{Assumption A (Absence of deep typical cycles)} \textit{Given the energy landscape $(X,H,q)$, we assume 
\begin{itemize}[align=left]
\item[(A1)] $\displaystyle \Tm=\Theta(x,A)=\TM$, and
\vspace{0.1cm}
\item[(A2)] $\displaystyle \Theta_{\mathrm{max}} (z,A) \leq \Theta(x,A)$ for every $z \in \cX \setminus A$.
\end{itemize}}
Condition (A1) says that every path $\o: x \to A$ visits one of the deepest typical cycles of the tube $\Tha$. Condition (A2) guarantees that by starting in another state $z \neq x$, the deepest typical cycle the process can enter is not deeper than those in $\Tha$. Checking the validity of Assumption A can be very difficult in general, but we give a sufficient condition in Proposition~\ref{prop:assa} which is satisfied in many models of interest, including the hard-core model on rectangular lattices presented in Subsection~\ref{sub22}, which will be revisited in Section~\ref{sec5}. We further remark that (A1) is precisely the assumption of Corollary~\ref{cor:cp}. Therefore, in the scenarios where Assumption A holds, we also have the asymptotic result~\eqref{eq:cp} in probability for the hitting time $\tha$.

The next theorem says that if Assumption A is satisfied, then the asymptotic order-of-magnitude of the mean hitting time $\E \tha $ as $\binf$ is $\Theta(x,A)$. 

\begin{thm}[First moment convergence]\label{thm:l1}
If Assumption A is satisfied, then
\[
\limb \frac{1}{\b} \log \E \tha = \Theta(x,A).
\]
\end{thm}

\noindent In many models of interest, calculating $\tG(\cX \setminus A)$ is easier than calculating $\Tm$ or $\TM$. Indeed, even if $\tG(\cX \setminus A)$ is a quantity that still requires a global analysis of the energy landscape, one needs to compute just the communication height $\Phi(z,A)$ between any state $z \in \cX \setminus A$ and the target set $A$, without requiring a full understanding of the complex cycle structure of the energy landscape. Besides this fact, the main motivation to look at the quantity $\tG(\cX \setminus A)$ is that it allows to give a sufficient condition for Assumption A, as illustrated in the following proposition.

\begin{prop}[Absence of deep cycles]\label{prop:assa}
If
\begin{equation}\label{eq:suffcondA}
\Phi(x,A) - H(x) = \tG(\cX \setminus A),
\end{equation}
then Assumption A holds.
\end{prop}
\begin{proof}
From the inequality 
\[
\Phi(x,A) -H(x) \leq \Tm \leq \TM \leq \tG(\cX \setminus A),
\]
we deduce that $\Tm = \TM$ and (A1) is proved. Moreover, by definition of $\tG(\cX \setminus A)$, we have $\Theta_{\mathrm{max}}(z,A) \leq \tG(\cX \setminus A)$ for every $z \in \cX \setminus A$. This inequality, together with the fact that $\TM = \tG(\cX \setminus A)$, proves that (A2) also holds and thus assumption A is satisfied. \qed
\end{proof}
\noindent We now present two interesting scenarios for which~\eqref{eq:suffcondA} holds.
\subsubsection*{Example 1 (metastability scenario)}
Suppose that
\[
x \in \ms \quad \text{ and } \quad A = \ss.
\]
In this first scenario, $\tha$ is the classical transition time between a metastable state and a stable state, a widely studied object in the statistical mechanics literature (see, e.g.~\cite{MNOS04}). Assumption A is satisfied in this case by applying Proposition~\ref{prop:assa}, since condition~\eqref{eq:suffcondA} holds: The equality $\Phi(x,\ss) -H(x) = \tG(\cX \setminus \ss)$ follows from the assumption $x \in \ms$, which means that there are no cycles in $\cX \setminus \ss$ that are deeper than $\cC_{\ss}(x)$.

\subsubsection*{Example 2 (tunneling scenario)}
Suppose that $x \in \ss$, $A = \ss \setminus \{x\}$ and
\begin{equation}
\label{eq:adw}
\Phi(z,A) -H(z) \leq \Phi(x,A)-H(x) \qquad \forall \, z \in \cX \setminus \{ x \}.
\end{equation}
In the second scenario, the hitting time $\tha$ is the \textit{tunneling time} between any pair of stable states. Assumption~\eqref{eq:adw} says that every cycle in the energy landscape which does not contain a stable state has depth strictly smaller than the cycle $\ca$ and we generally refer to this property as \textit{absence of deep cycles}. This condition immediately implies that~\eqref{eq:suffcondA} holds, i.e.~$\tG(\cX \setminus A)= \Phi(x,A)-H(x)$, and hence in this scenario assumption A holds, thanks to Proposition~\ref{prop:assa}.

The hard-core model on grids introduced in Subsection~\ref{sub22} falls precisely in this second scenario and, by proving the validity of Assumption A, we will get both the probability bounds~\eqref{eq:cp} and the first-moment convergence for the tunneling time $\teo$.

%%%%%%%%%%%%%%%%%%%%%%%%%%%%%%%%%%%%%%%%%%%%%%%%%%%%%%%%%%%%%%%%%%%%%%%%%%%%%%%%%%%%%%%%%%%%%%%%%%%%%%%%%%%%%%%%%%%%%%%%%%%%%%%%%%%%%%%%%%%%%%%%%%%%%%%%%%%%%%%%%%%%%%%%%%%%%%%%%%%%%%%%%%%%%%%%%%%%%%%%%%%%%%%%%%%%%%%%%%%%%%%%

\subsection{Asymptotic exponentiality}
\label{sub36}
We now present a sufficient condition for the scaled random variable $\tha / \E \tha$ to converge in distribution to an exponential unit mean random variable as $\binf$. Define
\begin{equation}
\label{eq:theta}
\Theta_*(x,A):=\limb \frac{1}{\b} \log \E \tha.
\end{equation}
If Assumption A holds, then we know that $\Theta(x,A)=\Theta_*(x,A)$, but the result presented in this section does not require the exact knowledge of $\Theta_*(x,A)$. We prove asymptotic exponentiality of the scaled hitting time under the following assumption.\\

\noindent \textbf{Assumption B (``Worst initial state'')} \textit{Given an energy landscape $(X,H,q)$, we assume that 
\begin{equation}
\label{eq:assb}
\Theta_*(x,A) > \tG(\cX \setminus (A \cup \{x\})).
\end{equation}
}

This assumption guarantees that the following ``recurrence'' result holds: From any state $z \in \cX$ the Markov chain reaches the set $A \cup \{x\}$ on a time scale strictly smaller than that at which the transition $x \to A$ occurs. Indeed, Proposition~\ref{prop:pup} gives that for any $\e >0$
\[
\limb \sup_{z \in \cX} \pr{\t^z_{\{x\} \cup A} > e^{\b(\tG(\cX \setminus (A \cup \{x\}))+\e)}} =0.
\]
We can informally say that Assumption B requires $x$ to be the ``worst initial state'' for the Markov chain when the target subset is $A$.

Proposition~\ref{prop:assb} gives a sufficient condition for Assumption B to hold, which is satisfied in many models of interest, in particular in the hard-core model on grid graphs described in Subsection~\ref{sub22}.

\begin{thm}[Asymptotic exponentiality]\label{thm:ae}
If Assumption B is satisfied, then 
\begin{equation}\label{eq:ae}
\frac{\tha}{ \E \tha} \cd \rmexp(1), \quad \binf.
\end{equation}
More precisely, there exist two functions $k_1(\b)$ and $k_2(\b)$ with $\limb k_1(\b)=0$ and $\limb k_2(\b)=0$ such that for any $s>0$
\[
\Big | \pr{\frac{ \tha }{\E \tha} > s} - e^{-s} \Big | \leq k_1(\b) e^{-(1-k_2(\b))s}.
\]
\end{thm}
The proof, presented in Section~\ref{sec4}, readily follows from the consequences of Assumption B discussed above and by applying~\cite[Theorem 2.3]{FMNS14}, 

We now present a condition which guarantees that Assumption B holds and show that it holds in two scenarios similar to those described in the previous subsection.
\begin{prop}[``The initial cycle $\ca$ is the unique deepest cycle'']\label{prop:assb}
If 
\begin{equation}
\label{eq:suffcondB}
\G(x,A) > \tG(\cX \setminus (A \cup \{x\})),
\end{equation}
then Assumption B is satisfied.
\end{prop}
The proof of this proposition is immediate from~\eqref{eq:phiineq} and~\eqref{eq:nestedineq}. We remark that if condition~\eqref{eq:suffcondB} holds, then the initial cycle $\ca$ is the unique deepest cycle in $\cX \setminus A$. Condition~\eqref{eq:suffcondB} is stronger than~\eqref{eq:assb}, but often easier to check, since one does not need to compute the exact value of $\Theta_*(x,A)$, but only the depth $\G(x,A)$ of the initial cycle $\ca$. We now present two scenarios of interest.

%%%%%%%%%%%%%%%%%%%%%%%%%%%%%%%%%%%%%%%%%%%%%%%%%%%%%%%%%%%%%%%%%%%%%%%%%%%%%%%%%%%%%%%%%%%%%%%%%%%%%%%%%%%%%%%%%%%%%%%%%%%%%%%%%%%%%%%%%%%%%%%%%%%%%%%%%%%%%%%%%%%%%%%%%%%%%%%%%%%%%%%%%%%%%%%%%%%%%%%%%%%%%%%%%%%%%%%%%%%%%%%%

\subsubsection*{Example 3 (unique metastable state scenario)}
Suppose that 
\[
\ms = \{ z \}, \quad A = \ss, \quad \text{ and } \quad x \in \cC_{A}(z).
\]
We remark that this scenario is a special case of the metastable scenario presented in Example 1 in Subsection~\ref{sub35}. This scenario was already mentioned in~\cite{MNOS04}, in the discussion following Theorem 4.15, but we briefly discuss here how to prove asymptotic exponentiality within our framework. Indeed, we have that 
\[
\G(x,\ss) = \G(\cC_{\ss}(z)) = \tG(\cX \setminus \ss),
\]
thanks to the fact that $z$ is the configuration in $\cX \setminus \ss$ with the maximum stability level, which means that $\cC_{\ss}(z)$ is the deepest cycle in $\cX \setminus \ss$. Moreover, the fact that $z$ is the unique metastable state, implies that 
\[
\tG(\cX \setminus \ss) > \tG(\cX \setminus (\ss \cup \{z\})),
\]
since every configuration in $\cX \setminus (\ss \cup \{z\})$ has stability level strictly smaller than $\mathcal V_z$.

\subsubsection*{Example 4 (two stable states scenario)}
Suppose that
\[
\ss = \{ s_1,s_2\}, \quad A = \{s_2\}, \quad x \in \cC_{A}(s_1) \quad \text{ and } \quad \tG(\cX \setminus \{s_1,s_2\}) < \Phi(s_1,s_2) - H(s_1).
\]
This scenario is a special case of the tunneling scenario presented in Example 2 in Subsection~\ref{sub35}. In this case condition~\eqref{eq:suffcondB} is obviously satisfied. In particular, it shows that the scaled tunneling time $\t^{s_1}_{s_2}$ between two stable states in $\cX$ is asymptotically exponential whenever $\ss = \{s_1,s_2\}$ and the condition $\tG(\cX \setminus \{s_1,s_2\}) < \Phi(s_1,s_2) - H(s_1)$ is satisfied.

In Section~\ref{sec5} we will show that for the hard-core model on grids Assumption B holds, being precisely in this scenario, and obtain in this way the asymptotic exponentiality of the tunneling time between the two unique stable states.

\subsection{An example of non-exponentiality}
\label{sub37}
Assumption B is a rather strong assumption. In fact, for many models and for most of choices of $x$ and $A$, the scaled hitting time $\tha / \E \tha$ does not have an exponential distribution in the limit $\binf$. Moreover, we do not claim that Assumption B is necessary to have asymptotically exponentiality of the scaled hitting time $\tha / \E \tha$. However, we will now show that for the hard-core model on $K$-partite graphs Assumption B does not hold and that the model exhibits non-exponentially distributed scaled hitting times.

Take $\L$ to be a complete $K$-partite graph. This means that the sites in~$\L$ can be partitioned into $K$~disjoint sets $\mathcal B_1,\dots, \mathcal B_K$ called \textit{components}, such that two sites are connected by an edge if and only if they belong to different components, see Figure~\ref{fig:kpar}.

\begin{figure}[!ht]
\centering
\subfigure[$K$-partite graph $\L$ with $K=5$]{\label{fig:kpar} \includegraphics[scale=0.22]{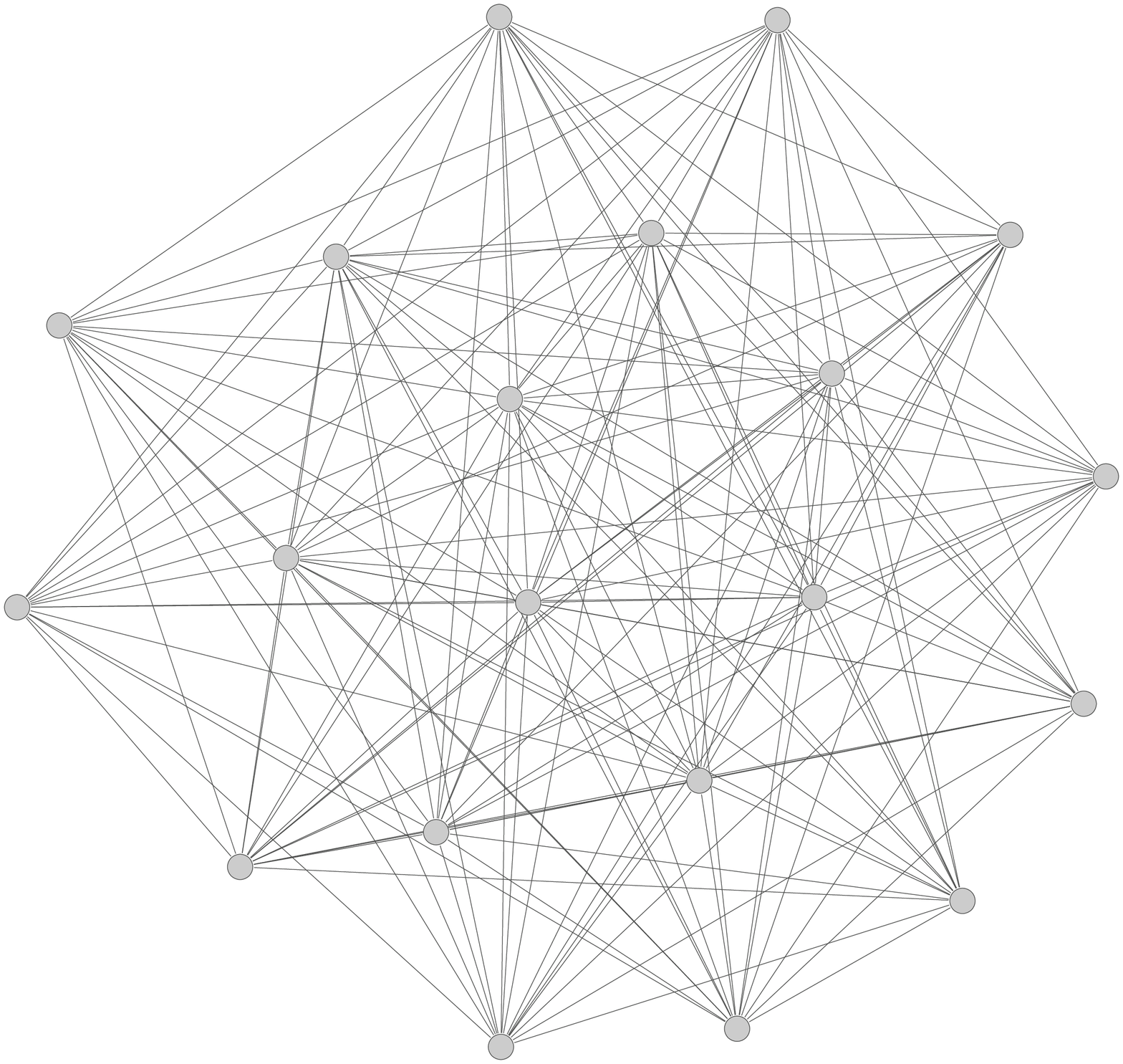}}
\hspace{0.5cm}
\subfigure[State space $\cX$ corresponding to the graph $\L$]{\label{fig:cXkpar} \includegraphics[scale=0.36]{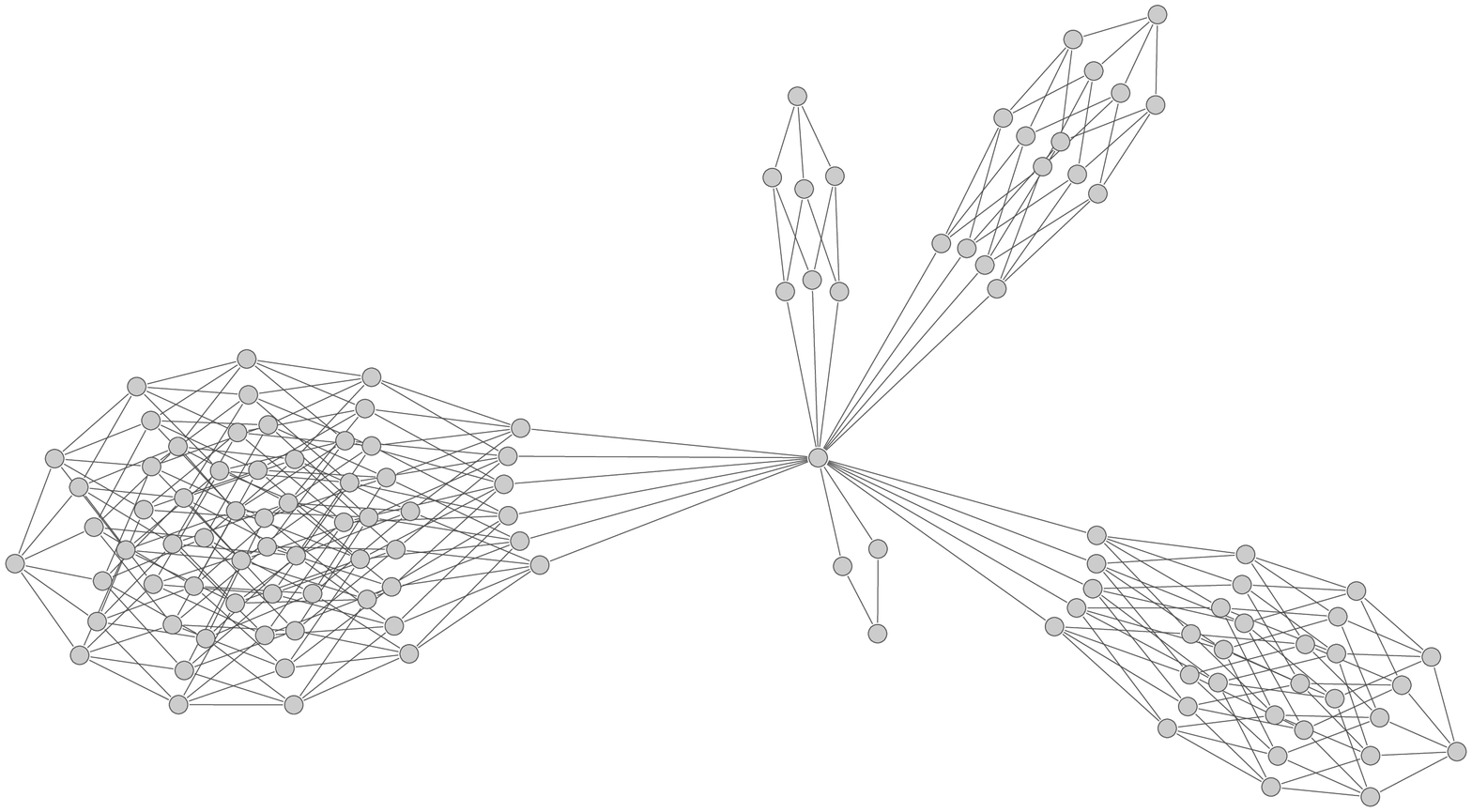}}
\caption{Exampe of a $K$-partite graph $\L$ and of the resulting energy landscape for the hard-core model on $\L$}
\end{figure}

This choice for $\L$ results in a simpler state space $\cX$, for which a detailed analysis is possible. Moreover, for the same model the asymptotic behavior of the first hitting times between maximal-occupancy configurations is already well understood, see~\cite{ZBvL12}. Before stating the results, we need some further definitions. Let $L_k$ be the size of the $k$-th component $\mathcal B_k$, for $k=1,\dots,K$. Clearly the total number of sites in $\L$ is $N=\sum_{k=1}^K L_k$. Define $L_{\text{max}}:=\max_{k=1,\dots,K} L_k$. For $k=1,\dots, K$ define the configuration $\s_k \in \cX$ as
\[
\s_k(v)=
\begin{cases}
1 & \text{ if } v \in \mathcal B_k,\\
0 & \text{ otherwise.}
\end{cases}
\]
The configurations $\{\s_1,\dots,\s_K \}$ are all the local minima of the energy function $H$ on the state space $\cX$. Moreover $\s_k$ is a stable state if and only if $L_k = L_\text{max}$. In addition, denote by $\mathbf{0}$ the configuration in $\cX$ where all the sites are empty, i.e.~the configuration such that $\mathbf{0}(v)=0$ for every $v \in \L$. Given $k_1,k_2 \in \{1,\dots, K\}$, $k_1 \neq k_2$, we take $\s_{k_1}$ and $\s_{k_2}$ as starting and target configurations, respectively. Define $L_*=L_*(k_2):= \max_{k\neq k_2} L_k$ and let $K_*=K_*(k_2):=\{ k \neq k_2 : L_k = L_*\}$ be the set of indices of the components of size $L_*$ different from $k_2$. 

In~\cite{ZBvL12} the same model has been considered, but in continuous time; the results therein (Theorems~IV.1 and IV.2) can be translated to discrete time as follows. Given two functions $f(\b)$ and $g(\b)$, we write $f \sim g$ as $\binf$ when $\limb f(\b)/g(\b) =1$.
\begin{prop}[First moment convergence of the hitting time $\t^{\s_{k_1}}_{\s_{k_2}}$]\label{prop:l1star} For any $k_1,k_2 \in \{1,\dots,K\}$ with $k_1 \neq k_2$, the first hitting time $\t^{\s_{k_1}}_{\s_{k_2}}$ satisfies
\[
\E  \t^{\s_{k_1}}_{\s_{k_2}} \sim N \left ( \frac{\mathds{1}_{\{k_1 \in K_*\}} }{L_*}+ \frac{|K_*|}{L_{k_2}} \right ) e^{\b L_*}, \quad \binf.
\]
In particular,
\[
\limb \frac{1}{\b} \log \E  \t^{\s_{k_1}}_{\s_{k_2}} = L_*.
\]
\end{prop}

\begin{prop}[Asymptotic distribution of the hitting time $\t^{\s_{k_1}}_{\s_{k_2}}$]\label{prop:aestar} Take $k_1,k_2 \in \{1,\dots,K\}$ such\\that $k_1 \neq k_2$. If $k_1 \in K_*$, then
\[
\frac{\t^{\s_{k_1}}_{\s_{k_2}}}{\E  \t^{\s_{k_1}}_{\s_{k_2}}} \cd \rmexp(1), \quad \binf.
\]
Instead, if $k_1 \not\in K_*$, then
\[
\frac{\t^{\s_{k_1}}_{\s_{k_2}}}{\E  \t^{\s_{k_1}}_{\s_{k_2}}} \cd Z, \quad \binf,
\]
where $ Z \ed \sum_{i=1}^M Y_i$ and $(Y_i)_{i \geq 1}$ are {\rm i.i.d.}~exponential unit mean random variables and $M$ is an independent random variable with geometric distribution $\mathbb P (M=n) =
(1 - p)^n p$ for $n\in \N \cup \{0\}$ with $p=\frac{L_{k_2}}{|K_*|L_*+L_{k_2}}$.
\end{prop}

As illustrated in Figure~\ref{fig:cXkpar}, the energy landscape consists of $K$ cycles, one for each component of $\L$, and one trivial cycle $\{\mathbf{0}\}$ which links all the others. The depth of each of the cycles is equal to the size of the corresponding component of $\L$. All the paths from $\s_{k_1}$ to $\s_{k_2}$ must at some point exit from the cycle corresponding to component $k_1$, at whose bottom lies $\s_{k_1}$. After hitting the configuration $\mathbf{0}$, they can go directly into the target cycle, i.e.~the one at which bottom lies $\s_{k_2}$, or they may fall in one of the other $K-1$ cycles. Formalizing these simple considerations, we can prove the following proposition.

\begin{prop}[Structural properties of the energy landscape]\label{prop:kparss}
For any $k_1,k_2 \in \{1,\dots, K\}$, $k_1 \neq k_2$,
\[
\G(\s_{k_1},\{\s_{k_2}\}) = L_{k_1} = \Psi_{\mathrm{min}}(\s_{k_1},\{\s_{k_2}\}),
\]
and 
\[
\Psi_{\mathrm{max}}(\s_{k_1},\{\s_{k_2}\}) =  L_* =  \tG(\cX \setminus \{\s_{k_2}\}).
\]
\end{prop}
In particular, if $k_1 \not\in K_*(k_2)$, then it follows from Propositions~\ref{prop:l1star} and~\ref{prop:kparss} that 
\[
L_* = \limb \frac{1}{\b} \E \t^{\s_{k_1}}_{\s_{k_2}} = \Theta(\s_{k_1},\{\s_{k_2}\}) \not < \tG(\cX \setminus \{\s_{k_1}, \s_{k_2}\}) = L_*.
\]
Hence Assumption B is not satisfied for the the pair $(\s_{k_1},\{\s_{k_2}\})$. Indeed there exists another configuration $\s_{k'}$, for some $k' \in K_*(k_2)$, $k' \neq k_1$, for which the recurrence probability
\[
\pr{ \tau^{\s_{k' }}_{\{\s_{k_1},\, \s_{k_2}\}} > e^{\b(L_{k_1} + \e)}}
\]
does not vanish as $\binf$, since component $\mathcal B_{k'}$ has size $L_* > L_{k_1}$. As illustrated in Proposition~\ref{prop:aestar}, the scaled hitting time $\t^{\s_{k_1}}_{\s_{k_2}}/\E  \t^{\s_{k_1}}_{\s_{k_2}}$ does not converge in distribution to an exponential random variable with unit mean as $\binf$.

\subsection{Mixing time and spectral gap}
\label{sub38}
In this subsection we focus on the long-run behavior of the Metropolis Markov chain $\xtbb$ and in particular examine the rate of convergence to the stationary distribution. We measure the rate of convergence in terms of the total variation distance and the mixing time, which describes the time required for the distance to stationarity to become small. More precisely, for every $0 < \e < 1$, we define the \textit{mixing time} $t^{\mathrm{mix}}_\b(\epsilon)$ by
\[
t^{\mathrm{mix}}_\b(\epsilon):=\min\{ n \geq 0 \st \max_{x \in \cX} \| P^n_\b(x,\cdot) - \mu_\b(\cdot) \|_{\mathrm{TV}} \leq \e \},
\]
where $\| \nu - \nu' \|_{\mathrm{TV}}:=\frac{1}{2} \sum_{x \in \cX} |\nu(x)-\nu'(x)|$ for any two probability distributions $\nu,\nu'$ on $\cX$. Another classical notion to investigate the speed of convergence of Markov chains is the \textit{spectral gap}, which is defined as 
\[
\rho_\b := 1-a_{\b}^{(2)},
\]
where $1=a_{\b}^{(1)} > a_{\b}^{(2)} \geq \dots \geq a_\b^{(|\cX|)} \geq -1$ are the eigenvalues of the matrix $(P_\b(x,y))_{x,y \in \cX}$. The spectral gap can be equivalently defined using the Dirichlet form associated with the pair $(P_\b, \mu_\b)$, see~\cite[Lemma 13.12]{LPW09}.
The problem of studying the convergence rate towards stationarity for a Friedlin-Wentzell Markov chain has already been studied in~\cite{C99,HS88,Miclo2002,Scoppola1993a}. In particular, in~\cite{C99} the authors characterize the order of magnitude of both its mixing time and spectral gap in terms of certain ``critical depths'' of the energy landscape associated with the Friedlin-Wentzell Markov chain. We summarize the results in the context of Metropolis Markov chains in the next proposition.

\begin{prop}[Mixing time and spectral gap for Metropolis Markov chains]\label{prop:mix}
For any $0 < \epsilon < 1$ and any $s\in \ss$, 
\begin{equation}
\label{eq:mixrho}
\limb \frac{1}{\b} \log t^{\mathrm{mix}}_\b(\epsilon) = \tG( \cX \setminus \{s\}) = \limb -\frac{1}{\b} \log \rho_\b.
\end{equation}
Furthermore, there exist two positive constants $0 < c_1 \leq c_2 < \infty$ independent of $\b$ such that
\begin{equation}
\label{eq:rho}
c_1 e^{-\b \tG( \cX \setminus \{s\})} \leq \rho_\b \leq c_2 e^{-\b \tG( \cX \setminus \{s\})} \qquad \forall \, \b \geq 0.
\end{equation}
\end{prop}

%%%%%%%%%%%%%%%%%%%%%%%%%%%%%%%%%%%%%%%%%%%%%%%%%%%%%%%%%%%%%%%%%%%%%%%%%%%%%%%%%%%%%%%%%%%%%%%%%%%%%%%%%%%%%%%%%%%%%%%%%%%%%%%%%%%%%%%%%%%%%%%%%%%%%%%%%%%%%%%%%%%%%%%%%%%%%%%%%%%%%%%%%%%%%%%%%%%%%%%%%%%%%%%%%%%%%%%%%%%%%%%%
%%%%%%%%%%%%%%%%%%%%%%%%%%%%%%%%%%%%%%%%%%%%%%%%%%%%%%%%%%%%%%%%%%%%%%%%%%%%%%%%%%%%%%%%%%%%%%%%%%%%%%%%%%%%%%%%%%%%%%%%%%%%%%%%%%%%%%%%%%%%%%%%%%%%%%%%%%%%%%%%%%%%%%%%%%%%%%%%%%%%%%%%%%%%%%%%%%%%%%%%%%%%%%%%%%%%%%%%%%%%%%%%
%%%%%%%%%%%%%%%%%%%%%%%%%%%%%%%%%%%%%%%%%%%%%%%%%%%%%%%%%%%%%%%%%%%%%%%%%%%%%%%%%%%%%%%%%%%%%%%%%%%%%%%%%%%%%%%%%%%%%%%%%%%%%%%%%%%%%%%%%%%%%%%%%%%%%%%%%%%%%%%%%%%%%%%%%%%%%%%%%%%%%%%%%%%%%%%%%%%%%%%%%%%%%%%%%%%%%%%%%%%%%%%%

\section{Proof of results for general Metropolis Markov chain}
\label{sec4}
In this section we prove the results presented in Section~\ref{sec3} for a Metropolis Markov chain $\xtbb$ with energy landscape $(\cX,H,q)$ and inverse temperature $\b$. For compactness, we will suppress the implicit dependence on the parameter $\b$ in the notation.

%%%%%%%%%%%%%%%%%%%%%%%%%%%%%%%%%%%%%%%%%%%%%%%%%%%%%%%%%%%%%%%%%%%%%%%%%%%%%%%%%%%%%%%%%%%%%%%%%%%%%%%%%%%%%%%%%%%%%%%%%%%%%%%%%%%%%%%%%%%%%%%%%%%%%%%%%%%%%%%%%%%%%%%%%%%%%%%%%%%%%%%%%%%%%%%%%%%%%%%%%%%%%%%%%%%%%%%%%%%%%%%%

\subsection{Proof of Lemma~\ref{lem:equivopt}}
If $\o \in \Opt$, then trivially $\o \in \Omega_{x,\cA}$. Moreover, we claim that $\o \in \Opt$ implies $\o \subseteq \cp$. Indeed, by definition of an optimal path and inequality~\eqref{eq:c+phi}, it follows that an optimal path cannot exit from $\cp$ since 
\[
\Phi_\o = \Phi(x,A) < H(\cF(\pa \cp)).
\]
The reverse implication follows from the minimality of $\cp$, which guarantees that $\Phi(x,A)=\max_{z \in \cp} H(z)$. \qed

\subsection{Proof of Proposition~\ref{prop:plowup1}}
We first prove the lower bound~\eqref{eq:plow1} and, in the second part of the proof, the upper bound~\eqref{eq:pup1}.\\

Consider the event $\{\tha < e^{\b(\Dm -\e)}\}$ first. There are two possible scenarios: Either the process exits from the cycle $\cp$ before hitting $A$ or not. Hence,
\begin{align}
\label{eq:lowtwoterms}
\pr{\tha &< e^{\b(\Dm -\e)}}  = \nonumber \\
&=\pr{\tha < e^{\b(\Dm -\e)}, \tha < \t^x_{\pa \cp}} +\pr{\t^x_{\pa \cp} \leq \tha < e^{\b(\Dm -\e)}} \nonumber\\
&\leq \pr{\tha < e^{\b(\Dm -\e)}, \tha < \t^x_{\pa \cp}} +\pr{\t^x_{\pa \cp} < e^{\b(\Dm -\e)}}.
\end{align}
The quantity $\pr{\t^x_{\pa \cp} < e^{\b(\Dm -\e)}}$ is exponentially small in $\b$ for $\b$ sufficiently large, thanks to Theorem~\ref{thm:exitcycle}(i) and to the fact that $\Dm < \G(\cp)$. In order to derive an upper bound for the first term in the right-hand side of~\eqref{eq:lowtwoterms}, we introduce the following set
\[
\cZ_{\mathrm{opt}}:=\{ z \in \cR \setminus A \st \G(z,A) \geq \Dm \}.
\]
By definition~\eqref{eq:defDm} of $\Dm$, every optimal path $\o \in \Opt$ must inevitably visit a cycle of depth not smaller than $\Dm$ and therefore it has to enter the subset $\cZ_{\mathrm{opt}}$ before hitting $A$. Hence, for every $z \in \cZ_{\mathrm{opt}}$, conditioning on the event $\{ \tha < \t^x_{\pa \cp}, \, X_{\t^x_{\cZ_{\mathrm{opt}}}} = z\}$, we can write
\[
\tha \ed \t^x_{z} + \t^{z}_{A},
\]
and, in particular, $\tha \geq_\textrm{st} \t^{z}_{A}$. Using this fact, we get that there exists some $k_2>0$ such that for $\b$ sufficiently large 
\begin{align}
\pr{&\tha < e^{\b(\Dm -\e)}, \tha < \t^x_{\pa \cp}} = \nonumber \\
& =  \pr{\tha < \t^x_{\pa \cp} } \pr{\tha < e^{\b (\Dm-\e)} \st \tha < \t^x_{\pa \cp}} \nonumber \\
& \leq \pr{\tha < \t^x_{\pa \cp}} \sum_{z \in \cZ_{\mathrm{opt}}} \pr{ \tha < e^{\b (\Dm-\e)} \st \tha < \t^x_{\pa \cp}, \, X_{\t^x_{\cZ_{\mathrm{opt}}}} = z} \pr{X_{\t^x_{\cZ_{\mathrm{opt}}}} = z} \nonumber\\
& \leq \pr{\tha < \t^x_{\pa \cp}} \sum_{z \in \cZ_{\mathrm{opt}}} \pr{ \t^{z}_{A} < e^{\b (\Dm-\e)}} \pr{X_{\t^x_{\cZ_{\mathrm{opt}}}} = z} \nonumber\\
& \leq \pr{\tha < \t^x_{\pa \cp}} \sum_{z \in \cZ_{\mathrm{opt}}} \pr{ \t^{z}_{A} < e^{\b (\G(z,A)-\e)}} \pr{X_{\t^x_{\cZ_{\mathrm{opt}}}} = z} \nonumber\\
& \leq \pr{\tha < \t^x_{\pa \cp}} \sum_{z \in \cZ_{\mathrm{opt}}} \pr{ \t^{z}_{\pa C_{A}(z)} < e^{\b (\G(z,A)-\e)}} \pr{X_{\t^x_{\cZ_{\mathrm{opt}}}} = z} \nonumber\\
& \leq \pr{\tha < \t^x_{\pa \cp}} \sum_{z \in \cZ_{\mathrm{opt}}} e^{-k_2 \b} \cdot \pr{X_{\t^x_{\cZ_{\mathrm{opt}}}} = z} = \pr{\tha < \t^x_{\pa \cp}} \cdot e^{-k_2 \b} \nonumber\\
& \leq e^{-k_2 \b},
\label{eq:Zopt}
\end{align}
where we used Theorem~\ref{thm:exitcycle}(i) and the facts that $\tau^z_A \geq \tau^z_{\pa C_A(z)}$ and that $\G(C_A(z)) = \G(z,A) \geq \Dm$ for every $z \in \cZ_{\mathrm{opt}}$.\\

For the upper bound, we can argue that
\begin{align*}
\pr{\tha &> e^{\b(\DM +\e)}} =\\
&= \pr{\tha > e^{\b(\DM +\e)}, \tha < \t^x_{\pa \cp}} +\pr{\tha > e^{\b(\DM +\e)}, \t^x_{\pa \cp} \leq \tha }\\
&\leq \pr{\tha > e^{\b(\DM +\e)}, \tha < \t^x_{\pa \cp}} +\pr{\t^x_{\pa \cp} \leq \tha}.
\end{align*}
The second term is exponentially small in $\b$ thanks to Theorem~\ref{thm:exitcycle}(iii) applied to the cycle $\cp$, to which both $x$ and at least one state of $A$ belong. We now turn our attention to the first term.\\
If the Markov chain $\xtn$ hits the target set $A$ before exiting from the cycle $\cp$, then it has been following an optimal path and, in particular, before hitting $A$ it can have visited only states in the set $\cR \setminus A$. Consider a state $z \in \cR \setminus A$. By definition of $\cR$, $z$ can be reached from $x$ by means of an optimal path, i.e.~there exists a path $\o^* : z \to x$ such that $\Phi_{\o^*} \leq \Phi(x,A)$. This fact implies that $\Phi(z,A) \leq \Phi(x,A)$ and thus for every path in $\o \in \Omega^{\mathrm{opt}}_{z,A}$, we can obtain a path that belongs to $\Opt$ by concatenating $\o^*$ and $\o$. Hence,
\begin{equation}
\label{eq:PsizPsix}
\Psi_{\mathrm{max}}(z,A) \leq \DM.
\end{equation}
Lemma~\ref{lem:exitvtjcp} guarantees the existence of a cycle-path $C_1,\dots,C_n$ {\rm vtj}-connected to $A$ such that $z \in C_1$ and $C_1, \dots C_n \in \cM(\cX \setminus A)$. From the fact that this cycle-path is {\rm vtj}-connected and Lemma~\ref{lem:equivvtj}, it follows that $H(\cF(\pa C_i)) \leq \Phi(x,A)$. Definition~\eqref{eq:defDM}, inclusion~\eqref{eq:vtjopt} and inequality~\eqref{eq:PsizPsix} imply that
\[
\G(C_i)\leq \DM, \quad i=1,\dots,n.
\]
For every $i = 2,\dots, n$ take a state $y_{i} \in \cF(\pa C_{i-1}) \cap C_{i}$. Furthermore, take $y_{1} = z$ and $y_{n+1} \in \cF(\pa C_n)\cap A$. Consider the set of paths
\begin{equation*}
\mathcal E_{\e,z,A} := \mathcal E_{\e,z,A} \left (y_1,C_1,y_2,C_2,\dots,y_n,C_n,y_{n+1}\right ) 
\end{equation*}
consisting of the paths constructed by the concatenation of any $n$--tuple of paths $\o^{(1)},\o^{(2)},\dots,\o^{(n)}$ satisfying the following conditions:
\begin{itemize}[align=left]
\item[(1)] The path $\o^{(i)}$ has length $|\o^{(i)}| \leq e^{\b(\DM+\e/4)}$, for any $i=1,\dots,n$;
\item[(2)] The path $\o^{(i)}$ joins $y_{i}$ to $y_{i+1}$, i.e.~$\o^{(i)} \in \Omega_{y_{i},y_{i+1}}$, for any $i=1,\dots,n$;
\item[(3)] All the states $\o^{(i)}_{j}$ belong to $C_{i}$ for any $j=1,\dots,|\o^{(i)}|-1$, for any $i=1,\dots,n$.
\end{itemize}
We stress that the first condition restricts the set $\mathcal E_{\e,z,A}$ to paths that spend less than $e^{\b(\DM+\e/4)}$ time in cycle $C_{i}$, for every $i =1,\dots n$. Note that the length of any path $\o \in \mathcal E_{\e,z,A}$ satisfies the upper bound $|\o| \leq |\cX| e^{\b(\DM+\e/4)}$. Moreover, since the state space $\cX$ is finite, for $\b$ sufficiently large 
\[
|\o| \leq |\cX| e^{\b(\DM+\e/4)} \leq e^{\b(\DM+\e/2)} \quad \forall \, \o \in \mathcal E_{\e,z,A}.
\]
Therefore, for every $z \in \cR \setminus A$
\[
\pr{\t^z_{A} \leq e^{\b(\DM+\e/2)} } \geq \pr{\t^z_{A} \leq e^{\b(\DM+\e/2)}, (X_m)_{m=1}^{\tha} \in \mathcal E_{\e,z,A} } = \pr{(X_m)_{m=1}^{\t^z_{A}} \in \mathcal E_{\e,z,A}}.
\] 
Using the Markov property, we obtain that for any $\e'>0$ and $\b$ sufficiently large
\[
\pr{(X_m)_{m=1}^{\tha} \in \mathcal E_{\e,z,A}} = \prod_{i=1}^{n} \pr{\t^{y_{i}}_{\pa C_{i}} \leq e^{\b(\DM + \e/4)}, X^{y_{i}}_{\t^{y_i}_{\pa C_{i}}} = y_{i+1} } \geq e^{-\b \e' n} \geq e^{-\b  \e'|\cX|},
\]
where the second last inequality follows from Theorem~\ref{thm:exitcycle}(v). Since $e^{-\b  \e'|\cX|}$ does not depend on the initial state $z$,
\[
\inf_{ z \in \cR \setminus A} \pr{\t^z_{A} \leq e^{\b(\DM+\e/2)} } \geq e^{-\b  \e'|\cX|}.
\]
Applying iteratively the Markov property at the times $k e^{\b(\DM + \e/2)}$, with $k=1,\dots, e^{\b \e /2}$, we obtain that
\begin{align*}
\pr{\tha > e^{\b(\DM+\e)},\tha <\t^x_{\pa \cp}}
& \leq  \Big(\sup_{z \in \cR \setminus A}  \pr{\t^z_{A} > e^{\b(\DM+\e/2)} } \Big)^{e^{\b \e/2}} \\
& \leq \left ( 1 - e^{-\b \e' |\cX|} \right )^{e^{\b \e/2}} \leq e^{ - e^{\b (\e/2-\e'|\cX|)}}.
\end{align*}
We remark that we can take the supremum over the states in $\cR \setminus A$, since all the other states in $\cp \setminus \cR$ cannot be reached by means of an optimal path (i.e.~without exiting from $\cp$) before visiting the target subset $A$. Choosing $\e' > 0$ small enough and $\b$ sufficiently large, we get that $e^{ -e^{ \b(\e/2-\e' |\cX|)}} \leq e^{-k\b}$ for any $k>0$. \qed

%%%%%%%%%%%%%%%%%%%%%%%%%%%%%%%%%%%%%%%%%%%%%%%%%%%%%%%%%%%%%%%%%%%%%%%%%%%%%%%%%%%%%%%%%%%%%%%%%%%%%%%%%%%%%%%%%%%%%%%%%%%%%%%%%%%%%%%%%%%%%%%%%%%%%%%%%%%%%%%%%%%%%%%%%%%%%%%%%%%%%%%%%%%%%%%%%%%%%%%%%%%%%%%%%%%%%%%%%%%%%%%%

\subsection{Proof of Lemma~\ref{lem:equivvtj}}
Take a path $\o \in \Ome$ and consider the corresponding cycle-path $\cC_\o=(C_1,\dots,C_{m(\o)})$. We will show that 
\[
\o \not\in \Vtj \quad \Longleftrightarrow \quad \exists \, i ~:~ \Phi(\o_{i+1},A) > \Phi(\o_{i},A).
\]
If $\o \not\in \Vtj$, then the cycle-path $\cC_\o=(C_1,\dots,C_{m(\o)})$ is not {\rm vtj}-connected to $A$, i.e.~there exists an index $1 \leq k \leq m(\o)$ such that $\pa C_{k} \cap C_{k+1} \neq \emptyset$, but $\cF(\pa C_{k}) \cap C_{k+1} = \emptyset$. Take the corresponding index $i$ in the path $\o$ such that $\o_i \in C_k$ and $\o_{i+1} \in \pa C_k \cap C_{k+1}$. Since $\o_{i+1} \not\in \cF(\pa C_k)$,
\[
\Phi(\o_{i+1},A) \geq H(\o_{i+1}) > H(\cF(\pa C_k)) = \Phi(\o_i,A),
\]
where the last equality follows from the fact that $C_k=\cC_A(\o_i)$.

For the converse implication, suppose the path $\o$ contains two consecutive states $\o_i$ and $\o_{i+1}$ such that $\Phi(\o_{i+1},A) > \Phi(\o_{i},A)$. Consider the index $k$ such that $\o_i \in C_k$. The states $\o_i$ and $\o_{i+1}$ must be contained in two different cycles, otherwise if $\cC_A(\o_i) =C_k = \cC_A(\o_{i+1})$, one would have $\Phi(\o_{i+1},A) = \Phi(\o_{i},A)$, which is a contradiction. Assume then that $\o_i \in C_k$ for some $k$, and thus $\o_{i+1} \in \pa C_k \cap C_{k+1}$ by definition of a cycle-path. The state $\o_{i+1}$ cannot belong to $\cF(\pa C_k)$, since otherwise
\[
\Phi(\o_{i+1},A) \leq \Phi(\cF(\pa C_k),A) = H(\cF(\pa C_k))=\Phi(\o_i,A),
\]
where we have used the fact that $C_k = C_A(\o_i)$. Hence $\o_{i+1} \not\in \cF(\pa C_k)$ and thus $\o \not\in \Vtj$. \qed

%%%%%%%%%%%%%%%%%%%%%%%%%%%%%%%%%%%%%%%%%%%%%%%%%%%%%%%%%%%%%%%%%%%%%%%%%%%%%%%%%%%%%%%%%%%%%%%%%%%%%%%%%%%%%%%%%%%%%%%%%%%%%%%%%%%%%%%%%%%%%%%%%%%%%%%%%%%%%%%%%%%%%%%%%%%%%%%%%%%%%%%%%%%%%%%%%%%%%%%%%%%%%%%%%%%%%%%%%%%%%%%%

\subsection{Proof of Lemma~\ref{lem:exittube}}

In~\eqref{eq:boundaryTha} we have used the fact that the only way to exit from the tube $\Tha$ without having hit the subset $A$ first is to exit from the non-principal boundary of a cycle $C \in \FT$. Therefore
\begin{align*}
\pr{\t^x_{\pa \Tha} < \tha}
& = \sum_{C \in \FT } \pr{\t^x_{\pa \Tha} < \tha, \, X_{\t^x_{\pa \Tha}-1} \in C,  \, X_{\t^x_{\pa \Tha}} \not\in \cF(\pa C)}\\
& = \sum_{C \in \FT } \sum_{z \in C} \pr{\t^x_{\pa \Tha} < \tha, \, X_{\t^x_{\pa \Tha}-1}=z,  \, X_{\t^x_{\pa \Tha}} \not\in \cF(\pa C)}\\
& \leq \sum_{C \in \FT } |C| \sup_{z \in C} \pr{X_{\t^z_{\pa C}} \not\in \cF(\pa C) } \leq \sum_{C \in \FT } |C| e^{-k_C \b} < e^{- \kappa \b},
\end{align*}
for some $\kappa > 0$ and $\b$ sufficiently large. The second last inequality follows from Theorem~\ref{thm:exitcycle}(iv). Thanks to the definition~\eqref{eq:defTha} of the typical tube, $\pr{\t^x_{\pa \Tha} = \tha}=0$, since all the states of the target state $A$ that can be hit starting from $x$ by means of a typical path belong to $\Tha$ and not to $\pa \Tha$. \qed

%%%%%%%%%%%%%%%%%%%%%%%%%%%%%%%%%%%%%%%%%%%%%%%%%%%%%%%%%%%%%%%%%%%%%%%%%%%%%%%%%%%%%%%%%%%%%%%%%%%%%%%%%%%%%%%%%%%%%%%%%%%%%%%%%%%%%%%%%%%%%%%%%%%%%%%%%%%%%%%%%%%%%%%%%%%%%%%%%%%%%%%%%%%%%%%%%%%%%%%%%%%%%%%%%%%%%%%%%%%%%%%%

\subsection{Proof of Proposition~\ref{prop:plowup2}}

As mentioned in Subection~\ref{sub34}, this proposition is a refinement of Proposition~\ref{prop:plowup1}, so instead of giving a full proof, we will just describe the necessary modifications.\\

We first prove~\eqref{eq:plow2}. Consider the event $\{\tha < e^{\b(\Tm -\e)}\}$ first. There are two possible scenarios: Either the process exits the tube $\Tha$ of typical paths before hitting $A$ or it stays in $\Tha$ until it hits $A$. Hence,
\begin{align}
\label{eq:lowTha}
\pr{\tha &< e^{\b(\Tm -\e)}}  \nonumber \\
&= \pr{\tha < e^{\b(\Tm -\e)}, \tha < \t^x_{\pa \Tha}} +\pr{\t^x_{\pa \Tha} \leq \tha < e^{\b(\Tm -\e)}} \nonumber\\
&\leq \pr{\tha < e^{\b(\Tm -\e)}, \tha < \t^x_{\pa \Tha}} +\pr{\t^x_{\pa \Tha} \leq \tha}.
\end{align}
Lemma~\ref{lem:exittube} implies that the second term in the right-hand side of~\eqref{eq:lowTha} is exponentially small in $\b$. In order to derive an upper bound for the first term in~\eqref{eq:lowTha}, we introduce the set
\[
\cZ_{\mathrm{vtj}}:=\{ z \in \Tha \setminus A \st \G(z,A) \geq \Tm \}.
\]
By definition~\eqref{eq:defTm} of $\Tm$, every typical path $\o \in \Vtj$ must inevitably visit a cycle of depth not smaller than $\Tm$ and therefore has to enter the subset $\cZ_{\mathrm{vtj}}$ before hitting $A$. Hence, for every $z \in \cZ_{\mathrm{vtj}}$, conditioning on the event $\{ \tha < \t^x_{\pa \Tha}, \, X_{\t^x_{\cZ_{\mathrm{vtj}}}} = z\}$, we can write
\[
\tha \ed \t^x_{z} + \t^{z}_{A},
\]
and in particular we have that $\tha >_\textrm{st} \t^{z}_{A}$. Using this fact and arguing like in~\eqref{eq:Zopt}, we can prove that there exists $\kappa >0$ such that $\b$ sufficiently large such that
\[
\pr{\tha < e^{\b(\Tm -\e)}, \tha < \t^x_{\pa \Tha}} \leq e^{-\kappa \b}.
\]

We now turn our attention to the proof of the upper bound~\eqref{eq:pup2}. First note that
\begin{align}
\pr{\tha &> e^{\b(\TM+\e)} } = \nonumber\\
& = \pr{\tha > e^{\b(\TM+\e)} , \tha <\t^x_{\pa \Tha}}  +\pr{\tha > e^{\b(\TM+\e)}, \t^x_{\pa \Tha} \leq \tha} \nonumber\\
& \leq \pr{\tha > e^{\b(\TM+\e)} , \tha <\t^x_{\pa \Tha}}  +\pr{\t^x_{\pa \Tha} \leq \tha},
\label{eq:upTha}
\end{align}
where the the latter term is exponentially small in $\b$ for $\b$ sufficiently large, thanks to Lemma~\ref{lem:exittube}. For the first term in~\eqref{eq:upTha}, we refine the argument given in the second part of the proof of Proposition~\ref{prop:plowup1}. Consider a state $z \in \Tha \setminus A$. Since $\mathrm{T}_{A}(z) \subseteq \Tha$, it follows from~\eqref{eq:TMequiv} that
\begin{equation}
\label{eq:TzTx}
\Theta_{\mathrm{max}}(z,A) \leq \TM.
\end{equation}
Thanks to Lemma~\ref{lem:exitvtjcp}, there exists a cycle-path of maximal cycles $C_1,\dots,C_n \subset$ in $\cX \setminus A$ that is {\rm vtj}-connected to $A$ and such that $z\in C_1$. The definition of {\rm vtj}-connected cycle-path, Lemma~\ref{lem:equivTM} and inequality~\eqref{eq:TzTx} imply that
\begin{equation}
\G(C_i)\leq \TM, \quad \forall \, i=1,\dots,n.
\label{eq:GCiTM}
\end{equation}
For each $i =2, \dots, n$, take a state $y_{i} \in \cF(\pa C_{i-1}) \cap C_{i}$. Furthermore, take $y_{1} = z$ and $y_{n+1} \in \cF(\pa C_n)\cap A$. We consider the collection of paths
\[
\mathcal E^*_{\e,z,A} := \mathcal E^*_{\e,z,A} \left (y_1,C_1,y_2,C_2,\dots,y_n,C_n,y_{n+1}\right), 
\]
which consists of all paths obtained by concatenating any $n$--tuple of paths $\o^{(1)},\o^{(2)},\dots,\o^{(n)}$ satisfying the following conditions:
\begin{itemize}[align=left]
\item[(1)] The path $\o^{(i)}$ has length $|\o^{(i)}| \leq e^{\b(\TM+\e/4)}$, for any $i=1,\dots,n$;
\item[(2)] The path $\o^{(i)}$ joins $y_{i}$ to $y_{i+1}$, i.e.~$\o^{(i)} \in \Omega_{y_{i},y_{i+1}}$, for any $i=1,\dots,n$;
\item[(3)] All the states $\o^{(i)}_{j}$ belong to $C_{i}$ for any $j=1,\dots,|\o^{(i)}|-1$, for any $i=1,\dots,n$.
\end{itemize}
This collection is similar to the collection $\mathcal E_{\e,z,A}$ described in the proof of Proposition~\ref{prop:plowup1}, but condition~(1) here is stronger. Using~\eqref{eq:GCiTM} and arguing as in the proof of Proposition~\ref{prop:plowup1}, we obtain that
\[
\pr{\t^z_{A} \leq e^{\b(\TM+\e/2)} } \geq \pr{(X_m)_{m=1}^{\tha} \in \mathcal E^*_{\e,z,A}} \geq e^{-\b  \e'|\cX|}.
\]
Since $e^{-\b  \e'|\cX|}$ does not depend on the initial state $z$, we get for any $\e'>0$ and $\b$ sufficiently large
\[
\inf_{ z \in \Tha} \pr{\t^z_{A} \leq e^{\b(\TM+\e/2)} } \geq e^{-\b  \e'|\cX|},
\]
and thus
\begin{align}
\pr{\tha > e^{\b(\TM+\e)},\tha <\t^x_{\pa \Tha}}
& \leq  \Big(\sup_{z \in \Tha \setminus A}  \pr{\t^z_{A} > e^{\b(\TM+\e/2)} } \Big)^{e^{\b \e/2}} \nonumber \\
& \leq \left ( 1 - e^{-\b \e' |\cX|} \right )^{e^{\b \e/2}} \leq e^{ - e^{\b (\e/2-\e'|\cX|)}}, \label{eq:firstterm} 
\end{align} 
by applying iteratively the Markov property at the times $k e^{\b(\TM + \e/2)}$, with $k=1,\dots, e^{\b \e /2}$. Choosing $\e' > 0$ small enough and $\b$ sufficiently large, we get that the right-hand side of inequality~\eqref{eq:firstterm} is super-exponentially small in $\b$, which completes the proof of the upper bound~\eqref{eq:pup2}. \qed

%%%%%%%%%%%%%%%%%%%%%%%%%%%%%%%%%%%%%%%%%%%%%%%%%%%%%%%%%%%%%%%%%%%%%%%%%%%%%%%%%%%%%%%%%%%%%%%%%%%%%%%%%%%%%%%%%%%%%%%%%%%%%%%%%%%%%%%%%%%%%%%%%%%%%%%%%%%%%%%%%%%%%%%%%%%%%%%%%%%%%%%%%%%%%%%%%%%%%%%%%%%%%%%%%%%%%%%%%%%%%%%%

\subsection{Proof of Theorem~\ref{thm:l1}}

Since Assumption (A1) holds, we set $\Theta(x,A) = \Tm = \TM$. The starting point of the proof is the following technical lemma.

\begin{lem}[Uniform integrability] \label{lem:ui}
If Assumption {\rm (A2)} holds, then for any $\e >0$ the variables $Y^x_{A}(\b) := \tha e^{-\b(\Theta(x,A) +\e)}$ are uniformly integrable, i.e.~there exists $\b_0 >0$ such that for any $\d>0$ there exists $K \in (0,\infty)$ such that for any $\b > \b_0$
\[
\E \left (Y^x_{A}(\b) \mathds{1}_{\{Y^x_{A}(\b) >K\}} \right )< \d.
\]
\end{lem}
\begin{proof}
The proof is similar to that of~\cite[Corollary 3.5]{MNOS04}. It suffices to have exponential control of the tail of the random variable $Y^x_{A}(\b)$ for $\b$ sufficiently large,~i.e.
\[
\pr{Y^x_{A}(\b) > n } = \pr{ \tha e^{-\b(\Theta(x,A) +\e)} > n} \leq a^n,
\]
with $a<1$. Assumption (A2) implies that $\Theta_{\mathrm{max}}(z,A) \leq \Theta(x,A)$ for every $z \in \cX \setminus A$. Then, iteratively using the Markov property gives
\[
\pr{ \tha  > n e^{-\b(\Theta(x,A) +\e)}} 
\leq \Big (\sup_{z \not\in A} \pr{ \t^{z}_{A}  > e^{\b(\Theta(x,A) +\e)}}\Big )^n \leq \Big (\sup_{z \not\in A} \pr{ \t^{z}_{A}  > e^{\b(\Theta_{\mathrm{max}}(z,A) +\e)}}\Big )^n,
\]
and the conclusion follows from Proposition~\ref{prop:plowup2}. \qed
\end{proof}
Proposition~\ref{prop:plowup2} implies that the random variable $Y^x_{A}(\b) := \tha e^{-\b(\Theta(x,A) +\e)}$ converges to $0$ in probability as $\binf$. Lemma~\ref{lem:ui} guarantees that the sequence $(Y^x_{A}(\b))_{\b \geq \b_0}$ is also uniformly integrable and thus $\limb \E |Y^x_{A}(\b)| = 0$. Therefore, for any $\e>0$ we have that for $\b$ sufficiently large $\E \tha < e^{\b(\Theta(x,A) +\e)}$.
As far as the lower bound is concerned, for any $\e>0$ Proposition~\ref{prop:plowup2} and the identity $\Theta(x,A)=\Tm$ yield 
\[
\E \tha > e^{\b(\Theta(x,A) -\e/2)} \pr{ \tha > e^{\b(\Theta(x,A) -\e/2)}} \geq e^{\b(\Theta(x,A) -\e/2)} (1-e^{-\kappa \b}) \geq e^{\b(\Theta(x,A) -\e)}.
\]
Since $\e$ is arbitrary, the conclusion follows. \qed

%%%%%%%%%%%%%%%%%%%%%%%%%%%%%%%%%%%%%%%%%%%%%%%%%%%%%%%%%%%%%%%%%%%%%%%%%%%%%%%%%%%%%%%%%%%%%%%%%%%%%%%%%%%%%%%%%%%%%%%%%%%%%%%%%%%%%%%%%%%%%%%%%%%%%%%%%%%%%%%%%%%%%%%%%%%%%%%%%%%%%%%%%%%%%%%%%%%%%%%%%%%%%%%%%%%%%%%%%%%%%%%%

\subsection{Proof of Theorem~\ref{thm:ae}}
As mentioned before, the strategy is to show that the Markov chain $\xtn$ satisfies the assumptions of \cite[Theorem 2.3]{FMNS14}, which for completeness we reproduce here. For $R > 0$ and $r \in (0,1)$, we say that the pair $(x,A)$ with $A \subset \cX$ satisfies $\textrm{Rec}(R, r)$ if
\[
\sup_{z \in \cX} \pr{ \tau^z_{\{x,A\}}>R } \leq r.
\]
The quantities $R$ and $r$ are called \textit{recurrence time} and \textit{recurrence error}, respectively.

\begin{thm}{\rm \cite[Theorem 2.3]{FMNS14}}
Consider a nonempty subset $A \subset \cX$ and $x \not\in A$ such that $\mathrm{Rec}(R(\b),r(\b))$ holds and 
\begin{itemize}[align=left]
\item[{\rm (i)}] $\limb R(\b) /\E \tha (\b) = 0$,
\item[{\rm (ii)}] $\limb r(\b)=0$. 
\end{itemize}
Then there exist two functions $k_1(\b)$ and $k_2(\b)$ with $\limb k_1(\b)=0$ and $\limb k_2(\b)=0$ such that for any $s>0$
\begin{equation}\label{eq:aerec}
\Big | \pr{\frac{ \tha }{\E \tha} > s} - e^{-s} \Big | \leq k_1(\b) e^{-(1-k_2(\b))s}.
\end{equation}
\end{thm}
Since $\tG(\cX \setminus (A\cup \{x\})) < \Theta(x,A)$ by assumption, we can take $\e>0$ small enough such that $\tG(\cX \setminus (A\cup \{x\})) +\e < \Theta(x,A)$. Proposition~\ref{prop:pup} implies that there exists $\kappa>0$ such that the pair $(x,A)$ satisfies $\textrm{Rec}(e^{\b \tG(\cX \setminus (A\cup \{x\}))+\e)}, e^{-\kappa \b})$ for $\b$ sufficiently large, since
\[
\sup_{z \in \cX} \pr{ \t^z_{\{x,A\}}> e^{\b (\tG(\cX \setminus (A\cup \{x\}))+\e)} } \leq e^{-e^{\kappa \b}}.
\]
Clearly $r(\b)=e^{-e^{\kappa \b}} \to 0$ as $\binf$ and thus assumption (ii) holds. Assumption (i) is also satisfied, since
\[
\limb \frac{1}{\b} \log R(\b) = \tG(\cX \setminus (A\cup \{x\})) +\e < \Theta(x,A)=\limb \frac{1}{\b} \log \E \tha. \eqno \qed
\]

\subsection{Proof of Proposition~\ref{prop:mix}}
The two limits in~\eqref{eq:mixrho} are an almost immediate consequence of~\cite[Theorem 5.1]{C99} and~\cite[Proposition 2.1]{Miclo2002}. Indeed, we just need to show that the critical depths $H_2$ and $H_3$ (see below for their definitions) that appear in these two results are equal to $\tG( \cX \setminus \{s\})$, for any $s \in \ss$. The critical depth $H_2$ is equal to $\tG( \cX \setminus \{s\})$ by definition, see~\cite{C99}. Note that this quantity is well defined, since its value is independent of the choice of $s$, as stated in~\cite[Theorem 5.1]{C99}. This critical depth is also known in the literature as \textit{maximal internal resistance} of the state space $\cX$, see~\cite[Remark 4.4]{MNOS04}.

The definition of the critical depth $H_3$ is more involved and we need some further notation. Consider the two-dimensional Markov chain $(X_t,Y_t)_{t \geq 0}$, where $X_t$ and $Y_t$ are two independent Metropolis Markov chains on the same energy landscape $(\cX,H,q)$ and indexed by the same inverse temperature $\b$. In other words, $(X_t,Y_t)_{t \geq 0}$ is the Markov chain on $\cX \times \cX$ with transition probabilities $P_\b^{\otimes 2}$ given by
\[
P_\b^{\otimes 2}\Big ( (x,y),(w,z) \Big) = P_\b(x,w)P_\b(y,z) \quad \forall \, (x,y),(w,z) \in \cX^2.
\]
The critical depth $H_3$ is then defined as
\[
H_3:= \tG( \cX \times \cX \setminus D),
\]
where $D := \{ (x,x) \st x \in \cX\}$. Consider the \textit{null-cost graph} on the set of stable states, i.e.~the directed graph $(V,E)$ with vertex set $V=\ss$ and edge set 
\[
E=\Big \{ (s,s') \in \ss \times \ss ~\Big | ~ \limb - \frac{1}{\b} \log P_\b(s,s') = 0 \Big \}.
\]
\cite[Theorem 5.1]{C99} guarantees that $H_2 \leq H_3$ and states that if the null-cost graph has an aperiodic component, then $H_2=H_3$. We claim that this condition is always satisfied by a Metropolis Markov chain with energy landscape $(\cX,H,q)$ with a non-constant energy function $H$. It is enough to show that for any such a Markov chain there exists at least one stable state $s \in \ss$ such that 
\[
\limb - \frac{1}{\b} \log P_\b(s,s) = 0.
\]
The subset $\cX \setminus \ss$ is a non-empty set, since $H$ is non-constant. Since $q$ is irreducible, there exists a state $s \in \ss$ and $x \in \cX \setminus \ss$ such that $q(s,x) >0$. Furthermore, we can choose $s \in \ss$ and $x \in \cX \setminus \ss$ such that the difference $H(x)-H(s)$ is minimal. For this stable state $s$, the transition probability towards itself reads
\begin{align*}
P_\b(s,s) &=1 - \sum_{y \neq s} q(s,y)e^{-\b (H(y)-H(s))^+} \\
& = 1 - \sum_{s' \in \ss, \, s'\neq s} q(s,s') - \sum_{y \in \cX \setminus \ss} q(s,y) e^{-\b (H(y)-H(s))^+}\\
& \geq 1 - \sum_{s' \in \ss, \, s'\neq s} q(s,s') -  e^{-\b (H(x)-H(s))^+} \sum_{y \in \cX \setminus \ss} q(s,y) \\
&  \geq 1 - \sum_{s' \in \ss, \, s'\neq s} q(s,s') -  e^{-\b (H(x)-H(s))^+}.
\end{align*}
Since $q$ is a stochastic matrix, it follows that $1- \sum_{s' \in \ss, \, s'\neq s} q(s,s') >0$ independently of $\b$ and thus
\[
\limb - \frac{1}{\b} \log P_\b(s,s) = 0,
\]
since for every $\e>0$ there exists $\b_0$ such that $P_\b(s,s) \geq 1 - \sum_{s' \in \ss, \, s'\neq s} q(s,s') - e^{-\b (H(x)-H(s))^+} > e^{- \b \e}$ for $\b > \b_0$.\\
Finally, the bounds~\eqref{eq:rho} follow immediately from~\cite[Theorem 2.1]{HS88}, since the quantity $m$ which appears there is equal to $\tG( \cX \setminus \{s\})$ thanks to Lemma~\ref{lem:GAequiv1}. \qed

%%%%%%%%%%%%%%%%%%%%%%%%%%%%%%%%%%%%%%%%%%%%%%%%%%%%%%%%%%%%%%%%%%%%%%%%%%%%%%%%%%%%%%%%%%%%%%%%%%%%%%%%%%%%%%%%%%%%%%%%%%%%%%%%%%%%%%%%%%%%%%%%%%%%%%%%%%%%%%%%%%%%%%%%%%%%%%%%%%%%%%%%%%%%%%%%%%%%%%%%%%%%%%%%%%%%%%%%%%%%%%%%
%%%%%%%%%%%%%%%%%%%%%%%%%%%%%%%%%%%%%%%%%%%%%%%%%%%%%%%%%%%%%%%%%%%%%%%%%%%%%%%%%%%%%%%%%%%%%%%%%%%%%%%%%%%%%NEW%SECTION%%%%%%%%%%%%%%%%%%%%%%%%%%%%%%%%%%%%%%%%%%%%%%%%%%%%%%%%%%%%%%%%%%%%%%%%%%%%%%%%%%%%%%%%%%%%%%%%%%%%%%%%
%%%%%%%%%%%%%%%%%%%%%%%%%%%%%%%%%%%%%%%%%%%%%%%%%%%%%%%%%%%%%%%%%%%%%%%%%%%%%%%%%%%%%%%%%%%%%%%%%%%%%%%%%%%%%%%%%%%%%%%%%%%%%%%%%%%%%%%%%%%%%%%%%%%%%%%%%%%%%%%%%%%%%%%%%%%%%%%%%%%%%%%%%%%%%%%%%%%%%%%%%%%%%%%%%%%%%%%%%%%%%%%%

\section{Energy landscape analysis for hard-core grid models}
\label{sec5}
This section is devoted to the analysis of the energy landscapes corresponding to the hard-core dynamics on the three different types of grids presented in Section~\ref{sec2}. Starting from geometrical and combinatorial properties of the admissible configurations, we prove some structural properties of the energy landscapes for $\cX_{T_{K,L}}, \cX_{G_{K,L}}$ and $\cX_{C_{K,L}}$. These results are precisely the model-dependent characteristics that are needed to exploit the general framework developed in Section~\ref{sec3} to obtain the main results for the hard-core model on grids presented in Subsection~\ref{sub23}. These structural properties are stated in the next three theorems and the rest of this section is devoted to their proofs.

\begin{thm}[Structural properties of $\cX_{T_{K,L}}$] \label{thm:toricel}
Let $\L$ be the $K\times L$ toric grid $T_{K,L}$. Then
\begin{itemize}[align=left]
\item[{\rm (a)}] $\tG(\cX \setminus \{\ee,\oo\}) \leq \min\{K,L\}$, 
\item[{\rm (b)}] $\G(\ee,\{\oo\}) = \min\{K,L\}+1 = \tG(\cX \setminus \{\oo\})$.
\end{itemize}
\end{thm}
Theorem~\ref{thm:toricel} implies that conditions~\eqref{eq:suffcondA} and~\eqref{eq:suffcondB} hold for the pair $(\ee,\{\oo\})$ in the energy landscape $(\cX_{T_{K,L}},H,q)$. Hence Assumptions A and B are satisfied and the statements of Theorems~\ref{thm:teo} and~\ref{thm:gamma} for a toroidal grid $T_{K,L}$ follow from Corollary~\ref{cor:cp} and Theorems~\ref{thm:l1} and~\ref{thm:ae}, respectively.

\begin{thm}[Structural properties of $\cX_{G_{K,L}}$] \label{thm:openel}
Let $\L$ be a $K\times L$ open grid $G_{K,L}$. If $KL \equiv 0 \pmod 2$, then
\begin{itemize}[align=left]
\item[{\rm (a)}] $\tG(\cX \setminus \{\ee,\oo\}) \leq \min\{\lceil K/2 \rceil, \lceil L/2 \rceil  \}$, 
\item[{\rm (b)}] $\G(\ee,\{\oo\}) = \min\{\lceil K/2 \rceil, \lceil L/2 \rceil  \}+1 = \tG(\cX \setminus \{\oo\})$.
\end{itemize}
If instead $KL \equiv 1 \pmod 2$, then
\begin{itemize}[align=left]
\item[{\rm (a*)}] $\tG(\cX \setminus \{\ee,\oo\}) < \min\{\lceil K/2 \rceil, \lceil L/2 \rceil \}$, 
\item[{\rm (b*)}] $\G(\ee,\{\oo\}) = \min\{\lceil K/2 \rceil, \lceil L/2 \rceil  \}+1 = \tG(\cX \setminus \{\oo\})$ and $\G(\oo,\{\ee\}) = \min\{\lceil K/2 \rceil, \lceil L/2 \rceil  \}= \tG(\cX \setminus \{\ee\})$.
\end{itemize}
\end{thm}
We remark that in the case $KL \equiv 1 \pmod 2$, inequality (a*) is strict, while inequality (a) is not, and this fact is crucial in order to conclude that $\oo$ is the unique metastable state of the state space $\cX_{G_{K,L}}$ when $KL \equiv 1 \pmod 2$. Using Theorem~\ref{thm:openel}, we can check that the pair $(\ee,\{\oo\})$ satisfies both Assumptions A and B (since both conditions~\eqref{eq:suffcondA} and~\eqref{eq:suffcondB} hold) and thus prove the asymptotic properties for the hitting times $\teo$ and $\toe$ illustrated in the statements of Theorems~\ref{thm:teo} and~\ref{thm:gamma} for an open grid $G_{K,L}$.

\begin{thm}[Structural properties of $\cX_{C_{K,L}}$] \label{thm:cycleel}
Let $\L$ be a $K\times L$ cylindrical grid $C_{K,L}$. Then
\begin{itemize}[align=left]
\item[{\rm (a)}] $\tG(\cX \setminus \{\ee,\oo\}) \leq \min\{K/2,L\}$,
\item[{\rm (b)}] $\G(\ee,\{\oo\}) = \min\{K/2,L\}+1 = \tG(\cX \setminus \{\oo\})$.
\end{itemize}
\end{thm}
Using Theorem~\ref{thm:cycleel}, we can check that Assumptions A and B are satisfied by the pair $(\ee,\{\oo\})$, and then the statements of Theorems~\ref{thm:teo} and~\ref{thm:gamma} for a cylindrical grid $C_{K,L}$ follow from Corollary~\ref{cor:cp} and Theorems~\ref{thm:l1} and~\ref{thm:ae}.

The ideas behind the proofs of these three theorems are similar, but for clarity we present them separately in Subsections~\ref{sub51},~\ref{sub52} and~\ref{sub53}.\\

Denote $\G(\L):= \tG(\cX \setminus \{\ee\})$, where $(\cX,H,q)$ is the energy landscape corresponding to the hard-core model on the grid $\L$. In the case $\L = G_{K,L}$ with $KL \equiv 1 \pmod 2$, Theorem~\ref{thm:openel} gives that $\G(\L)=\min \{ \lceil K/2 \rceil, \lceil L/2 \rceil \}$. In all the other cases by symmetry we have $\tG(\cX \setminus \{\ee\}) = \tG(\cX \setminus \{\oo \})$ and hence, from Theorems~\ref{thm:toricel},~\ref{thm:openel} and~\ref{thm:cycleel} it then follows that
\[
\G(\L)=
\begin{cases}
\min \{ K, L\} +1 & \text{ if } \L = T_{K,L},\\
\min \{ \lceil K/2 \rceil, \lceil L/2 \rceil \} +1 & \text{ if } \L = G_{K,L} \text{ and } KL \equiv 0 \pmod 2,\\
\min \{ \lceil K/2 \rceil, \lceil L/2 \rceil \} & \text{ if } \L = G_{K,L} \text{ and } KL \equiv 1 \pmod 2,\\
\min \{ K/2, L\} +1 & \text{ if } \L = C_{K,L}.
\end{cases}
\]
Besides appearing in the two main theorems (Theorems~\ref{thm:teo} and~\ref{thm:gamma}), the exponent $\G(\L)$ also characterizes the asymptotic order of magnitude of the mixing time $t^{\mathrm{mix}}_\b(\epsilon,\L)$ and of the spectral gap $\rho_\b(\L)$ of the hard-core dynamics $\xtn$ on $\L$ (see Subsection~\ref{sub38}), as illustrated in the next theorem.

\begin{thm}[Mixing time and spectral gap]\label{thm:mixgap}
For any rectangular grid $\L$ and for any $0 < \epsilon < 1$, 
\[
\limb \frac{1}{\b} \log t^{\mathrm{mix}}_\b(\epsilon,\L) = \G(\L) = \limb -\frac{1}{\b} \log \rho_\b(\L).
\]
Furthermore, there exist two constants $c_1,c_2$ independent of $\b$ such that
\[
c_1 e^{-\b \G(\L)} \leq \rho_\b(\L) \leq c_2 e^{-\b \G(\L)}.
\]
\end{thm}
The proof readily follows from the properties of the energy landscapes illustrated in Theorems~\ref{thm:toricel}, \ref{thm:openel} and~\ref{thm:cycleel} and by applying Proposition~\ref{prop:mix}.\\

We next introduce some notation and definitions for grid graphs. Recall that $\L$ is a $K \times L$ grid graph with $K,L \geq 2$ which has $N = KL$ sites in total. We define \textit{energy wastage of a configuration} $\s \in \cX$ on the grid graph $\L$ as the difference between its energy and the energy of the configuration $\ee$, i.e.
\begin{equation}\label{eq:delta}
U(\s) := H(\s) - H(\ee).
\end{equation}
Since $H(\ee) =- \lceil N/2 \rceil$, we have that
\[
U(\s) =  H(\s) + \lceil N/2 \rceil =  \lceil N/2 \rceil - \sum_{v \in \L} \s(v).
\]
Moreover, since $\ee$ is a stable state, $U(\s) \geq 0$. The function $U: \cX \to \R_+ \cup \{0\}$ is usually called \textit{virtual energy} in the literature~\cite{C99,CNS14b} and satisfies the following identity 
\[
U(\s) = - \limb \frac{1}{\b} \log \mu_\b(\s),
\]
where $\mu_\b$ is the Gibbs measure~\eqref{eq:gibbs} of the Markov chain $\xtn$.\\

We denote by $c_j$, $j=0,\dots,L-1$, the $j$-th column of $\L$, i.e.~the collection of sites whose horizontal coordinate is equal to $j$, and by $r_i$, $i=0,\dots,K-1$, the $i$-th row of $\L$, i.e.~the collection of sites whose vertical coordinate is equal to $i$. In addition, define the \textit{$i$-th horizontal stripe}, with $i=1,\dots,\lfloor K/2 \rfloor$, as 
\[
S_i:=r_{2i-2} \cup r_{2i-1},
\]
and the \textit{$j$-th vertical stripe}, with $j=1,\dots,\lfloor L/2 \rfloor$ as
\[
C_i:=c_{2j-2} \cup c_{2j-1}.
\]

An important feature of the energy wastage $U$ for grid graphs, is that it can be seen as the sum of the energy wastages on each row (or on each horizontal stripe). More precisely, let $U_j(\s)$ be the energy wastage of a configuration $\s \in \cX$ in the $i$-th row, i.e.
\begin{equation}\label{eq:deltai}
U_i(\s) := \lceil L/2 \rceil - \sum_{v \in r_i} \s(v).
\end{equation}
Similarly, let $U^S_i(\s)$ be the energy wastage of a configuration $\s \in \cX$ on the $i$-th horizontal stripe, i.e.
\begin{equation}\label{eq:tdeltai}
U^S_i(\s) := L - \sum_{v \in S_i} \s(v) = U_{2i-2}(\s) + U_{2i-1}(\s).
\end{equation}
Then, we can rewrite the energy wastage of a configuration $\s \in \cX$ as
\begin{equation}\label{eq:deltaid}
U(\s) = \sum_{i=1}^K U_i(\s) = \sum_{i=1}^{\lceil K/2 \rceil} U^S_i(\s).
\end{equation}
Given two configurations $\s,\s' \in \cX$ and a subset of sites $W \subset \L$, we write
\[
\s_{|W} = \s'_{|W} \quad \Longleftrightarrow \quad \s(v) = \s'(v) \quad \forall \, v \in W.
\]
We say that a configuration $\s \in \cX$ has a \textit{vertical odd (even) bridge} if there exists a column in which configuration $\s$ perfectly agrees with $\oo$ (respectively $\ee$), i.e.~if there exists an index $0 \leq j \leq L-1$ such that
\[
\s_{|c_j} = \oo_{|c_j} \quad (\text{respectively } \s_{|c_j} = \ee_{|c_j}).
\]
We define \textit{horizontal odd and even bridges} in an analogous way and we say that a configuration $\s \in \cX$ has an \textit{odd (even) cross} if it has both vertical and horizontal odd (even) bridges. 
\begin{figure}[h!]
\centering
\subfigure[Horizontal odd bridge]{\includegraphics[scale=1]{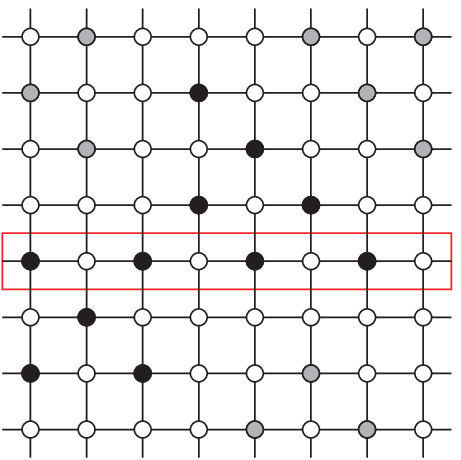}}
\hspace{0.4cm}
\subfigure[Two vertical odd bridges]{\includegraphics[scale=1]{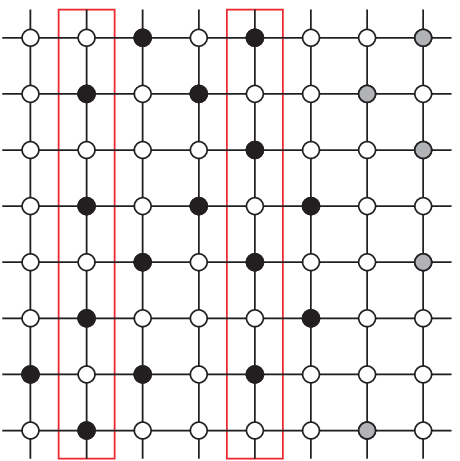}}
\hspace{0.4cm}
\subfigure[Odd cross]{\includegraphics[scale=1]{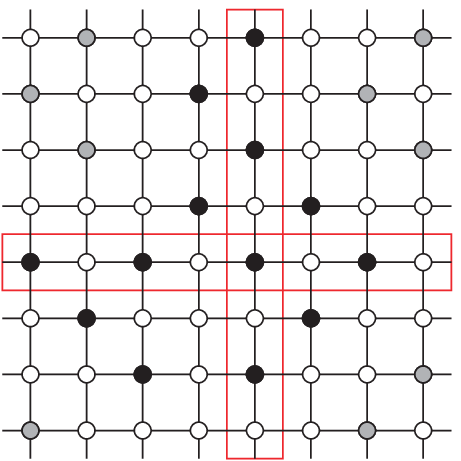}}
\caption{Examples of configurations on the $8 \times 8$ toric grid displaying odd bridges or crosses}
\label{fig:bridgesandcross}
\end{figure}
\FloatBarrier

We remark that the structure of the grid graph $\L$ and the hard-core constraints prohibit the existence of two perpendicular bridges of different parity, e.g.~a vertical odd bridge and a horizontal even bridge. Bridges and crosses are the geometric feature of the configurations which will be crucial in the following subsections to prove Theorems~\ref{thm:toricel},~\ref{thm:openel} and~\ref{thm:cycleel}.

%%%%%%%%%%%%%%%%%%%%%%%%%%%%%%%%%%%%%%%%%%%%%%%%%%%%%%%%%%%%%%%%%%%%%%%%%%%%%%%%%%%%%%%%%%%%%%%%%%%%%%%%%%%%%%%%%%%%%%%%%%%%%%%%%%%%%%%%%%%%%%%%%%%%%%%%%%%%%%%%%%%%%%%%%%%%%%%%%%%%%%%%%%%%%%%%%%%%%%%%%%%%%%%%%%%%%%%%%%%%%%%%

\subsection{Energy landscape analysis for toric grids (Proof of Theorem~\ref{thm:toricel})}
\label{sub51}
This subsection is devoted to the proof of Theorem~\ref{thm:toricel} valid for a toric grid $T_{K,L}$. Without loss of generality, we assume henceforth that $K \leq L$. Recall that by construction of the toric grid $\L$, both $K$ and $L$ are even integers. In the remainder of the section we will write $\cX$ instead of $\cX_{T_{K,L}}$ to keep the notation light. 

We first introduce a \textit{reduction algorithm}, which is used to construct a specific path in $\cX$ from any given state in $\cX \setminus \{ \ee,\oo\}$ to the subset $\{\ee,\oo\}$ and to show that
\begin{equation}
\label{eq:ndwtoric}
\tG(\cX \setminus \{\ee,\oo\}) \leq K,
\end{equation}
which proves Theorem~\ref{thm:toricel}(a). Afterwards, we show in Proposition~\ref{prop:lowertg} that
\[
\Phi(\ee,\oo) -H(\ee) \geq K +1,
\]
by giving lower bounds on the energy wastage along every path $\ee \to \oo$. The reduction algorithm is then used again in Proposition~\ref{prop:refpatht} to build a reference path $\o^*: \ee \to \oo$ which shows that the lower bound is sharp and hence
\[
\Phi(\ee,\oo) -H(\ee) = K +1,
\]
which, together with~\eqref{eq:ndwtoric}, proves Theorem~\ref{thm:toricel}(b).\\

The starting point of the energy landscape analysis is a very simple observation: A configuration in $\cX$ has zero energy wastage in a given row (column) if and only if it has an odd or even horizontal (vertical) bridge. The following lemma formalizes this property. We give the statement and the proof only for rows, since those for columns are analogous.

\begin{lem}[Energy efficient rows are bridges]\label{lem:rw}
For any $\s \in \cX$ and any $i = 0, \dots, K-1$,
\[
U_i(\s)=0 \quad \Longleftrightarrow \quad \s_{|r_i}=\ee_{|r_i} \quad \mathrm{or} \quad \s_{|r_i}=\oo_{|r_i}.
\]
\end{lem}
\begin{proof}
The $i$-th row of the toric grid graph $\L$ is a cycle graph with $L/2$ even sites and $L/2$ odd sites. If $\s_{|r_i}=\ee_{|r_i}$ or $\s_{|r_i}=\oo_{|r_i}$, then trivially there are $L/2$ occupied sites and hence $U_i(\s)=0$. Noticing that the configurations $\ee_{|r_i}$ and $\oo_{|r_i}$ on row $i$ correspond to the only two maximum independent sets of the cycle graph $r_i$ proves the converse implication. \qed
\end{proof}

%%%%%%%%%%%%%%%%%%%%%%%%%%%%%%%%%%%%%%%%%%%%%%%%%%%%%%%%%%%%%%%%%%%%%%%%%%%%%%%%%%%%%%%%%%%%%%%%%%%%%%%%%%%%%%%%%%%%%%%%%%%%%%%%%%%%%%%%%%%%%%%%%%%%%%%%%%%%%%%%%%%%%%%%%%%%%%%%%%%%%%%%%%%%%%%%%%%%%%%%%%%%%%%%%%%%%%%%%%%%%%%%

\subsubsection*{Reduction algorithm for toric grids}
We now describe an iterative procedure which builds a path $\o$ in $\cX$ from an initial configuration $\h$ (with specific properties, see below) to state $\oo$. We call it \textit{reduction algorithm}, because along the path it creates the even clusters are gradually reduced and they eventually disappear, since the final configuration is $\oo$.

The algorithm cannot be initialized in all configurations $\s \in \cX \setminus \{ \oo\}$. Indeed, we require that the initial configuration $\s$ is such that there are no particles in the even sites of the first vertical stripe $C_1$, i.e.
\begin{equation}
\label{eq:racond}
\sum_{v \in C_1 \cap V_e} \s(v) =0.
\end{equation}
This technical assumption is required because the algorithm needs ``some room'' to start working, as will become clear later.

The path $\o$ is the concatenation of $L$ paths $\o^{(1)}, \dots, \o^{(L)}$. Path $\o^{(j)}$ goes from $\s_{j}$ to $\s_{j+1}$, where we set $\s_1=\s$ and we define recursively configuration $\s_{j+1}$ from configuration $\s_{j}$ as follows.
\[
\s_{j+1}(v) := \begin{cases}
\s_j(v) & \text{ if } v \in \L \setminus (c_j \cup c_{j+1}),\\
\oo(v) & \text{ if } v \in c_j,\\
\s_j(v) & \text{ if } v \in c_{j+1} \cap V_o,\\
0 & \text{ if } v \in c_{j+1} \cap V_e.
\end{cases}
\]
Clearly, due to the periodic boundary conditions, the column index should be taken modulo $L$. It can be checked that indeed $\s_{L+1}=\oo$. We now describe in detail how to construct each of the paths $\o^{(j)}$ for $j=1,\dots,L$. We build a path $\o^{(j)}=(\o^{(j)}_{0}, \o^{(j)}_{1}, \dots, \o^{(j)}_{K+1})$ of length $K+1$ (but possibly with void moves), with $\o^{(j)}_{0}=\s_j$ and $\o^{(j)}_{K+1}=\s_{j+1}$. We start from configuration $\o^{(j,0)}=\s_j$ and we repeat iteratively the following procedure for all $i=0,\dots,K-1$:
\begin{itemize}
\item If $i \equiv 0 \pmod 2$, consider the even site $v=(j+1, i+(j+1 \pmod 2))$.
	\begin{itemize}
	\item[-] If $\o^{(j)}_{i}(v)=0$, we set $\o^{(j)}_{i+1}=\o^{(j)}_{i}$ and thus $H(\o^{(j)}_{i+1}) = H(\o^{(j)}_{i})$.
	\item[-] If $\o^{(j)}_{i}(v)=1$, then we remove from configuration $\o^{(j)}_{i}$ the particle in $v$ increasing the energy by $1$ and obtaining in this way configuration $\o^{(j)}_{i+1}$, which is such that $H(\o^{(j)}_{i+1}) = H(\o^{(j)}_{i})+1$.
	\end{itemize}
\item If $i \equiv 1 \pmod 2$, consider the the odd site $v=(j,i-1+(j+1 \pmod 2))$.
	\begin{itemize}
	\item[-] If $\o^{(j)}_{i}(v)=1$, we set $\o^{(j)}_{i+1}=\o^{(j)}_{i}$ and thus $H(\o^{(j)}_{i+1}) = H(\o^{(j)}_{i})$.
	\item[-] If $\o^{(j)}_{i}(v)=0$, then we add to configuration $\o^{(j)}_{i}$ a particle in site $v$ decreasing the energy by $1$. We obtain in this way a configuration $\o^{(j)}_{i+1}$, which is admissible because by construction all the first neighboring sites of $v$ are unoccupied. In particular, the particle at its right (i.e.~that at the site $v+(1,0)$) may have been removed exactly at the previous step. The new configuration has energy $H(\o^{(j)}_{i+1}) = H(\o^{(j)}_{i})-1$.
	\end{itemize}
\end{itemize}
Note that for the last path $\o^{(L)}$ all the moves corresponding to even values of $i$ are void (there are no particles in the even sites of $c_0$). The way the path $\o^{(j)}$ is constructed shows that for every $j=1,\dots,L$,
\[
H(\s_{j+1}) \leq H(\s_j),
\]
since the number of particles added in (the odd sites of) column $c_j$ is greater than or equal to the number of particles removed in (the even sites of) column $c_{j+1}$. Moreover,
\[
\Phi_{\o^{(j)}} \leq H(\s_j) + 1
\]
since along the path $\o^{(j)}$ every particle removal (if any) is always followed by a particle addition. These two properties imply that the path $\o: \s \to \oo$ created by concatenating $\o^{(1)},\dots,\o^{(L)}$ satisfies
\[
\Phi_{\o} \leq H(\s) + 1.
\]

\noindent \textit{Proof of }{\rm Theorem~\ref{thm:toricel}(a).}
It is enough to show that for every $\s \in \cX \setminus \{\ee,\oo\}$
\[
\Phi(\s,\oo) -H(\s) \leq K,
\]
since inequality~\eqref{eq:ndwtoric} then follows the equivalent characterization of $\tG$ given in Lemma~\ref{lem:GAequiv1}. To prove such an inequality, we have to exhibit for every $\s \in \cX \setminus \{\ee,\oo\}$ a path $\o: \s \to \oo$ in $\cX$ such that $\Phi_{\o}=\max_{\h \in \o} H(\h) \leq H(\s)+K$. We construct such a path $\o$ as the concatenation of two shorter paths, $\o^{(1)}$ and $\o^{(2)}$, where $\o^{(1)}: \s \to \s'$, $\o^{(2)}: \s' \to \oo$ and $\s'$ is a suitable configuration which depends on $\s$ (see definition below).

Since $\s \neq \ee$ by assumption, the configuration $\s$ must have a vertical stripe with \textit{strictly} less than $K$ even occupied sites. Without loss of generality (modulo a cyclic rotation of column labels) we can assume that this vertical stripe is the first one, $C_1$, and we define
\begin{equation}
\label{eq:tb}
b := \sum_{v \in C_1 \cap V_e} \s(v) \leq K-1.
\end{equation}
Define $\s'$ as the configuration that differs from $\s$ only in the even sites of the first vertical stripe, i.e.
\[
\s'(v):=
\begin{cases}
\s(v) & \text{ if } v \in \L \setminus (C_1 \cap V_e),\\
0 & \text{ if } v \in C_1 \cap V_e.
\end{cases}
\]
The path $\o^{(1)}=(\o^{(1)}_1, \dots, \o^{(1)}_{b+1})$, with $\o^{(1)}_{1}=\s$ and $\o^{(1)}_{b+1}=\s'$ can be constructed as follows. For $i=1,\dots,b$, in step $i$ we remove from configuration $\o^{(1)}_{i}$ the first particle in $C_1 \cap V_e$ in lexicographic order obtaining in this way configuration $\o^{(1)}_{i+1}$, increasing the energy by $1$. Therefore the configuration $\s'$ is such that $H(\s')-H(\s) = b$ and 
\[
\Phi_{\o^{(1)}} = \max_{\h \in \o^{(1)}} H(\h) \leq H(\s)+b.
\]
The path $\o^{(2)}: \s' \to \oo$ is then constructed by means of the reduction algorithm described earlier, choosing $\s'$ as initial configuration and $\oo$ as target configuration. The reduction algorithm guarantees that
\[
\Phi_{\o^{(2)}}=\max_{\h \in \o^{(2)}} H(\h) \leq H(\s')+1.
\]
The concatenation of the two paths $\o^{(1)}$ and $\o^{(2)}$ gives a path $\o: \s \to \oo$ which satisfies the inequality $\Phi_\o \leq H(\s)+b+1$ and therefore
\[
\Phi(\s,\oo)-H(\s) \leq b+1 \stackrel{\eqref{eq:tb}}{\leq} K. \eqno \qed
\]

\begin{prop}[Lower bound for $\Phi(\ee,\oo)$]\label{prop:lowertg}
\[
\Phi(\ee,\oo) - H(\ee) \geq K+1.
\]
\end{prop}
\begin{proof}
We need to show that in every path $\o: \ee \to \oo$, there is at least one configuration with energy wastage greater than or equal to $K+1$. Take a path $\o =(\o_1,\dots, \o_n) \in \Omega_{\ee,\oo}$. Since $\ee$ has no odd bridge and $\oo$ does, at some point along the path $\o$ there must be a configuration $\o_{m^*}$ which is the first to display an odd bridge, horizontal or vertical, or both simultaneously. In symbols
\[
m^* := \min\{ m \leq n \st \exists \, i ~:~ (\o_m)_{|r_i} = \oo_{|r_i} \quad \mathrm{or} \quad \exists \, j ~:~ (\o_m)_{|c_j} = \oo_{|c_j}\}.
\]
Clearly $m^*>2$. We claim that $U(\o_{m^*-1}) \geq K+1$ or $U(\o_{m^*-2}) \geq L+1$. We distinguish the following three cases:
\begin{itemize}
\item[(a)] $\o_{m^*}$ displays an odd vertical bridge only;
\item[(b)] $\o_{m^*}$ displays an odd horizontal bridge only;
\item[(c)] $\o_{m^*}$ displays an odd cross.
\end{itemize}
These three cases cover all the possibilities, since the addition of a single particle cannot create more than one bridge in each direction.

For case (a), let $v^* \in \L$ be the unique site where configurations $\o_{m^*-1}$ and $\o_{m^*}$ differ and assume that $v^* \in r_{i^*}$ for some $0 \leq i^*\leq K-1$. By construction, $v^*$ must be an odd site and $\o_{m^*-1}(v^*) = 0$ and $\o_{m^*}(v^*)=1$. This means that two neighboring sites at the left and at the right of $v^*$ must be unoccupied for both configurations, in particular $\o_{m^*-1}(v^*+(1,0)) = 0$ and $\o_{m^*-1}(v^*+(-1,0)) = 0$, otherwise the addition of a particle in $v^*$ would not be allowed. Moreover, there must be another odd unoccupied site in the same $r_{i^*}$, otherwise the addition of a particle in $v^*$ would create also an odd horizontal bridge, which we assumed $\o_{m^*}$ does not have. Having at least two unoccupied sites both in $r_{i^*} \cap V_e$ and in $r_{i^*} \cap V_o$ implies that
\begin{equation}\label{eq:ineqrowj}
U_{i^*}(\o_{m^*-1}) \geq 2.
\end{equation}
Moreover we claim that the energy wastage in every row which does not contain site $v^*$ is also greater than or equal to $1$. Indeed configuration $\o_{m^*-1}$ cannot display any odd row (by definition of $m^*$) and neither an even row, since $\o_{m^*-1}(v^*+(1,0)) = 0$ and $\o_{m^*-1}(v^*+(-1,0)) = 0$. Therefore for every $i=0,\dots,K-1$ and $i \neq i^*$ we have $(\o_{m^*})_{|r_i} \neq \oo_{|r_i}, \ee_{|r_i}$ and hence $U_i(\o_{m^*}) \geq 1$ by Lemma~\ref{lem:rw}. This, together with inequality~\eqref{eq:ineqrowj}, implies
\[
U(\o_{m^*-1}) \geq \sum_{i=0}^{K-1} U_i(\o_{m^*})  \geq K+1.
\]

For case (b) we can argue as in case (a), but interchanging the role of rows and columns, and obtain that 
\[
U(\o_{m^*-1}) \geq L+1 \geq K+1.
\]

For case (c), the vertical and horizontal odd bridges that $\o_{m^*}$ has must necessarily meet in the odd site $v^*$. Having an odd cross, $\o_{m^*}$ cannot have any horizontal or vertical even bridge. Consider the previous configuration $\o_{m^*-1}$ along the path $\o$, which can be obtained from $\o_{m^*}$ by removing the particle in $v^*$. From these considerations and from the definition of $m^*$ it follows that $\o_{m^*-1}$ has no vertical bridge (neither odd or even) and thus, by Lemma~\ref{lem:rw}, it has energy wastage at least one in every column, which amounts to
\[
U(\o_{m^*-1}) \geq L.
\]
If there is at least one column in which $\o_{m^*-1}$ has energy wastage strictly greater than one, we get 
\[
U(\o_{m^*-1}) \geq L+1,
\]
and the claim is proved. Consider now the other scenario, in which the configuration $\o_{m^*-1}$ has energy wastage exactly one in every column, which means $U(\o_{m^*-1}) = L$. Consider its predecessor in the path $\o$, namely the configuration $\o_{m^*-2}$. We claim that 
\[
U(\o_{m^*-2}) = L+1.
\]
By construction, configuration $\o_{m^*-2}$ must differ in exactly one site from $\o_{m^*-1}$ and therefore 
\[
U(\o_{m^*-2}) = U(\o_{m^*-1}) \pm1.
\]
Consider the case where $U(\o_{m^*-2}) = U(\o_{m^*-1})-1 = L -1$. In this case the configuration $\o_{m^*-2}$ must have a zero-energy-wastage column and by Lemma~\ref{lem:rw} it would be a vertical bridge. If it was an odd vertical bridge, the definition of $m^*$ would be violated. If it was an even vertical bridge, it would be impossible to obtain the odd horizontal bridge (which $\o_{m^*}$ has) in just two single-site updates, since three is the minimum number of single-site updates needed. Therefore
\[
U(\o_{m^*-2}) = U(\o_{m^*-1}) + 1 = L +1. \eqno \qed
\]
\end{proof}

\begin{prop}[Reference path]\label{prop:refpatht}
There exists a path $\o^*:\ee \to \oo$ in $\cX_{T_{K,L}}$ such that 
\[
\Phi_{\o^*} - H(\ee) = K+1.
\]
\end{prop}
\begin{proof}
We construct such a path $\o^*$ as the concatenation of two shorter paths, $\o^{(1)}$ and $\o^{(2)}$, where $\o^{(1)}: \ee \to \s^*$ and $\o^{(2)}: \s^* \to \oo$, and prove that $\Phi_{\o^{(1)}} = H(\s^*) = H(\s) + K$ and that $\Phi_{\o^{(2)}} = H(\s^*) + 1$ are satisfied, so that $\Phi_{\o^*} = \max_{\h \in \o^*} H(\h) = H(\ee)+ K+1$ as desired. The reason why $\o$ is best described as the concatenation of two shorter paths is the following: The reduction algorithm cannot in general be started directly from $\ee$ and the path $\o^{(1)}$ indeed leads from $\ee$ to $\s^*$, which is a suitable configuration to initialize the reduction algorithm. The configuration $\s^*$ differs from $\ee$ only in the even sites of the first vertical stripe:
\[
\s^*(v):=
\begin{cases}
\ee(v) & \text{ if } v \in \L \setminus C_1,\\
0 & \text{ if } v \in C_1.
\end{cases}
\]
The path $\o^{(1)}=(\o^{(1)}_{1},\dots,\o^{(1)}_{K+1})$, with $\o^{(1)}_{1}=\ee$ and $\o^{(1)}_{K+1}=\s^*$ can be constructed as follows. For $i=1,\dots,K$, at step $i$ we remove from configuration $\o^{(1)}_{i}$ the first particle in $C_1 \cap V_e$ in lexicographic order, increasing the energy by $1$ and obtaining in this way configuration $\o^{(1)}_{i+1}$. Therefore the configuration $\s'$ is such that $H(\s^*)-H(\ee) = K$ and $\Phi_{\o^{(1)}} = H(\ee)+K$. The second path $\o^{(2)}: \s^* \to \oo$ is then constructed by means of the reduction algorithm, which can be used since the configuration $\s^*$ satisfies condition~\eqref{eq:racond} and hence is a suitable initial configuration for the algorithm. The algorithm guarantees that $\Phi_{\o^{(2)}} = H(\s^*) +1$ and thus the conclusion follows. \qed
\end{proof}

%%%%%%%%%%%%%%%%%%%%%%%%%%%%%%%%%%%%%%%%%%%%%%%%%%%%%%%%%%%%%%%%%%%%%%%%%%%%%%%%%%%%%%%%%%%%%%%%%%%%%%%%%%%%%%%%%%%%%%%%%%%%%%%%%%%%%%%%%%%%%%%%%%%%%%%%%%%%%%%%%%%%%%%%%%%%%%%%%%%%%%%%%%%%%%%%%%%%%%%%%%%%%%%%%%%%%%%%%%%%%%%%

\subsection{Energy landscape analysis for open grids (Proof of Theorem~\ref{thm:openel})}
\label{sub52}
We now prove Theorem~\ref{thm:openel} valid for an open grid $G_{K,L}$. Also in this case, we assume without loss of generality that $K \leq L$. Recall that $K$ and $L$ are positive integers, not necessarily even as in the previous subsection. In the remainder of the section we will write $\cX$ instead of $\cX_{G_{K,L}}$.

We first introduce a modification of the previous reduction algorithm tailored for open grids. The scope of this reduction algorithm is twofold. It is used first to build a specific path in $\cX$ from any given state in $\cX \setminus \{ \ee,\oo\}$ to the subset $\{\ee,\oo\}$ and to prove that if $KL \equiv 0 \pmod2$, then
\begin{equation}
\label{eq:ndwopen}
\tG(\cX \setminus \{\ee,\oo\}) \leq \lceil K/2 \rceil,
\end{equation}
which is Theorem~\ref{thm:openel}(a). The same argument also shows that if $KL \equiv 1 \pmod2$, then
\begin{equation}
\label{eq:ndwopenstrict}
\tG(\cX \setminus \{\ee,\oo\}) < \lceil K/2 \rceil,
\end{equation}
and also Theorem~\ref{thm:openel}(a*) is proved. By giving a lower bound on the energy wastage along every path $\ee \to \oo$, we show in Proposition~\ref{prop:lowerog} that
\[
\Phi(\ee,\oo) -H(\ee) \geq \lceil K/2 \rceil +1.
\]
Then, using again the reduction algorithm for open grids, we construct a reference path $\o^*: \ee \to \oo$ which proves that the lower bound above is sharp and hence
\begin{equation}
\label{eq:eeoo}
\Phi(\ee,\oo) -H(\ee) = \lceil K/2 \rceil +1.
\end{equation}
In the special case $KL \equiv 1 \pmod 2$, since $\Phi(\oo,\ee)=\Phi(\ee,\oo)$ and $H(\oo)=H(\ee)+1$, we can easily derive from the last equality that 
\begin{equation}
\label{eq:ooee}
\Phi(\oo,\ee) -H(\oo) = \lceil K/2 \rceil.
\end{equation}
Lastly, we combine inequality~\eqref{eq:ndwopen} and equation~\eqref{eq:eeoo} to obtain
\[
\tG(\cX \setminus \{\oo\}) =  \lceil K/2 \rceil+1,
\]
which concludes the proof of Theorem~\ref{thm:openel}(b). In the special case $KL \equiv 1 \pmod 2$, inequality~\eqref{eq:ndwopenstrict} and equation~\eqref{eq:ooee} prove Theorem~\ref{thm:openel}(b*), since they yield that
\[
\tG(\cX \setminus \{\ee\})= \lceil K/2 \rceil.
\]
We need one additional definition: Say that a configuration in $\cX$ displays an \textit{odd (even) vertical double bridge} if there exists at least one vertical stripe $S_i$ in which configuration $\s$ perfectly agrees with $\oo$ (respectively $\ee$), i.e.~if there exists an index $1 \leq j \leq \lfloor L/2 \rfloor$ such that 
\[
\s_{|C_j} = \oo_{|C_j} \quad (\text{respectively } \s_{|C_j} = \ee_{|C_j}).
\]
An \textit{odd (even) horizontal double bridge} is defined analogously. The two types of double bridges are illustrated in Figure~\ref{fig:doublebridges}.

Observe that an admissible configuration on the open grid has zero energy wastage in a horizontal (vertical) stripe if and only if it has an odd or even horizontal (vertical) bridge in that stripe. The next lemma formalizes this property. We give the statement and the proof only for horizontal stripes, since those for vertical stripes are analogous. In the special case of an open grid where $KL \equiv 1 \pmod 2$, the topmost row and the leftmost column need special treatment, since they do not belong to any stripe. The second part of the following lemma shows that an admissible configuration has zero energy wastage in that row/column if and only if they agree perfectly with $\ee$ therein. Again we will state and prove the result for the topmost row, the result for the leftmost column is analogous.

\begin{figure}[!h]
\centering
\subfigure[Odd horizontal double bridge]{\includegraphics[scale=1]{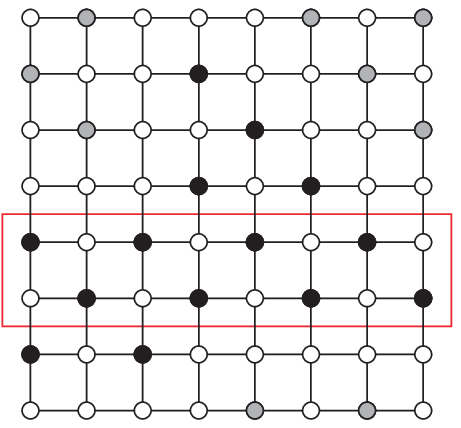}}
\hspace{2cm}
\subfigure[Odd vertical double bridges]{\includegraphics[scale=1]{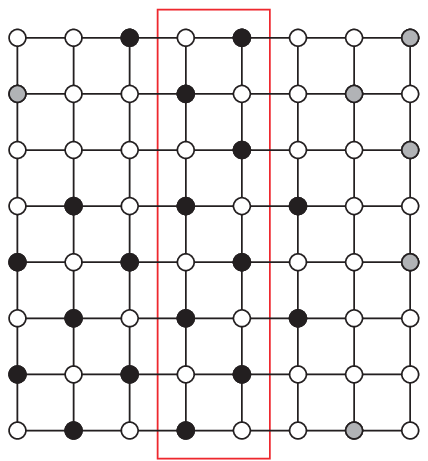}}
\caption{Examples of configurations on the $8 \times 8$ open grid displaying an odd double bridge}
\label{fig:doublebridges}
\end{figure}
\FloatBarrier

\begin{lem}[Energy efficient stripes are double bridges]\label{lem:bw}
Consider a configuration $\s \in \cX$.
\begin{itemize}
\item[(a)] For any $i = 0, \dots, \lfloor K /2 \rfloor -1$, the energy wastage $U^S_i(\s)$ in horizontal stripe $S_i$ satisfies
\[
U^S_i(\s)=0 \quad \Longleftrightarrow \quad \s_{|S_i}=\ee_{|S_i}\quad \mathrm{or} \quad \s_{|S_i}=\oo_{|S_i}.
\]
\item[(b)] If additionally $KL \equiv 1 \pmod 2$, then the energy wastage in the topmost row $U_{K-1}(\s)$ satisfies
\[
U_{K-1}(\s)=0 \quad \Longleftrightarrow \quad \s_{|r_{K-1}}=\ee_{|r_{K-1}}.
\]
\end{itemize}
\end{lem}
\begin{proof}
(a) Consider the rectangular $2 \times L$ grid graph induced by the horizontal stripe $S_i$: It has $L$ even sites and $L$ odd sites. If $\s_{|S_i}=\ee_{|S_i}$ or $\s_{|S_i}=\oo_{|S_i}$, trivially $U^S_i(\s)=0$. Let us prove the converse implication. Denote by $\ev_t$ ($\ev_b$) the number of particles present in even sites in the top (bottom) row of stripe $S_i$. Analogously, define $\ov_t$ ($\ov_b$) as the number of particles present in odd sites in the top (bottom) row of stripe $S_i$. We will show that:
\begin{itemize}
\item[(i)] $U^S_i(\s)=0$ and $\ev_t+\ev_b=0$ $\Longleftrightarrow$ $\ov_t+\ov_b=L$;
\item[(ii)] $U^S_i(\s)=0$ and $\ev_t+\ev_b>0$ $\Longleftrightarrow$ $\ov_t+\ov_b=L$.
\end{itemize}
Cases (i) and (ii $\Leftarrow$) are immediate. Thus we focus on the implication (ii $\Rightarrow$).\\
Note that if $\ev_t+\ev_b \in [1, L -1]$ particles are present in even sites, then they block at least $\ev_t+\ev_b +1$ odd sites, which must then be unoccupied. 
Indeed in the top row each of the $\ev_t$ particles blocks the odd node at its right and in the bottom row each of the $\ev_b$ particles blocks the odd node at its left. In one of the two rows, say the top one, there is at least one even unoccupied site and consider the even site at its right where a particle resides. This particle blocks also the odd site at its left. Hence $\ov_t+\ov_b \leq L-(\ev_t+\ev_b+1)$, which gives $U^S_i(\s) = L - (\ev_t+\ev_b+\ov_t+\ov_b) >0$.\\

(b) The topmost row has $\frac{L+1}{2}$ even sites and $\frac{L-1}{2}$ odd sites. Denote by $\ev$ (respectively $\ov$) the number of particles present in even (respectively odd) sites in row $r_{K-1}$. The energy wastage of $\s$ on this row can be computed as $U_{K-1}(\s)=\frac{L+1}{2}-\ev-\ov$. Trivially, if $\s_{|r_{K-1}}=\ee_{|r_{K-1}}$, then $\ev=\frac{L+1}{2}$ and thus $U_{K-1}(\s)=0$. Let us prove the opposite implication. Assume that $\s_{|r_{K-1}} \neq \ee_{|r_{K-1}}$, i.e.~$\ev<\frac{L+1}{2}$. If $\ev=0$, then $U_{K-1}(\s)\geq 1$, since $\ov \leq \frac{L-1}{2}$. If instead $\ev \in [1,\frac{L+1}{2}-1]$, then each particle residing in an even site blocks the odd site at its left, therefore $\ov \leq \frac{L-1}{2} - \ev$, which implies
\[
U_{K-1}(\s) = \frac{L+1}{2}-\ev-\ov \geq \frac{L+1}{2} -\ev - \left (\frac{L-1}{2} - \ev\right )\geq 1. \eqno \qed
\]
\end{proof}

%%%%%%%%%%%%%%%%%%%%%%%%%%%%%%%%%%%%%%%%%%%%%%%%%%%%%%%%%%%%%%%%%%%%%%%%%%%%%%%%%%%%%%%%%%%%%%%%%%%%%%%%%%%%%%%%%%%%%%%%%%%%%%%%%%%%%%%%%%%%%%%%%%%%%%%%%%%%%%%%%%%%%%%%%%%%%%%%%%%%%%%%%%%%%%%%%%%%%%%%%%%%%%%%%%%%%%%%%%%%%%%%

\subsubsection*{Reduction algorithm for open grids}
We now describe the \textit{reduction algorithm for open grids}, which is a modification of the reduction algorithm for toric grids that builds a path $\o$ in $\cX$ from a given initial configuration $\s$ to \textit{either} $\oo$ or $\ee$. The reduction algorithm for open grids takes two inputs instead of one: The initial configuration $\s$ and the target state which is either $\oo$ or $\ee$. This is the first crucial difference with the corresponding algorithm for toric grid, where the target configuration was always $\oo$. In the following, we first assume that the target state is $\oo$ and illustrate the procedure in this case. The necessary modifications when the target state is $\ee$ are presented later.

The initial configuration $\s$ for the reduction algorithm must be such that there are no particles in the even sites of the first column $c_0$, i.e.
\begin{equation}
\label{eq:ogracond}
\sum_{v \in c_0 \cap V_e} \s(v) =0.
\end{equation}
This condition ensures that the algorithm has enough ``room'' to work properly. Note that condition~\eqref{eq:ogracond} is different from condition~\eqref{eq:racond} for the reduction algorithm for toric grids, which requires instead that the even sites of \textit{both} the first two columns $c_0$ and $c_1$ should be empty.

The path $\o$ is the concatenation of $L$ paths $\o^{(1)}, \dots,\o^{(L)}$. Path $\o^{(j)}$ goes from $\s_{j}$ to $\s_{j+1}$, where we set $\s_1=\h$ and define recursively configuration $\s_{j+1}$ from configuration $\s_{j}$  for every $j=1, \dots, L$ as
\[
\s_{j+1}(v)
= \begin{cases}
\s_j(v) & \text{ if } v \in \L \setminus (c_j \cup c_{j+1}),\\
\oo(v)  & \text{ if } v \in c_j,\\
\s_j(v) & \text{ if } v \in c_{j+1} \cap V_o,\\
0 		& \text{ if } v \in c_{j+1} \cap V_e.
\end{cases}
\]
This procedure guarantees that $\s_{L+1}=\oo$. The path $\o^{(j)}$ for $j=1,\dots,L$ is constructed exactly as the path $\o^{(j)}$ for the reduction algorithm for toric grids. Since their construction is identical, every path $\o^{(j)}$ enjoys the same properties as those of the original reduction algorithm, namely 
\[
H(\s_{j+1}) \leq H(\s_j) \quad \text{ and } \quad \Phi_{\o^{(j)}} \leq H(\s_j) + 1.
\]
This means that the path $\o: \s \to \oo$ created by their concatenation satisfies
\[
\Phi_{\o} \leq H(\s) + 1.
\]

In the scenario where the target state is $\ee$, three modifications are needed. First the initial state $\s$ must be such that there are no particles in the \textit{odd} sites of the first column $c_0$, i.e.
\[
\sum_{v \in c_0 \cap V_e} \s(v) =0.
\]
Secondly, the sequence of intermediate configurations $\s_j$, $j=1,\dots, L$ must be modified as follows: We set $\s_1=\s$ and we define recursively $\s_{j+1}$ from $\s_{j}$ as
\[
\s_{j+1}(v) = \begin{cases}
\s_j(v) & \text{ if } v \in \L \setminus (c_j \cup c_{j+1}),\\
\ee(v) & \text{ if } v \in c_j,\\
\s_j(v) & \text{ if } v \in c_{j+1} \cap V_e,\\
0 & \text{ if } v \in c_{j+1} \cap V_o.
\end{cases}
\]
Lastly, for step $i$ of path $\o^{(j)}$, we need a different offset to select the site $v$, namely $v=(j,i+(j \pmod 2))$ when $i \equiv 0 \pmod 2$ and $v=(j,i-1+(j \pmod 2))$ when $i \equiv 1 \pmod 2$. One can check that the resulting path $\o: \s \to \ee$ satisfies the inequality
\[
\Phi_{\o} \leq H(\s) + 1.
\]

\noindent \textit{Proof of }{\rm Theorem~\ref{thm:openel}(a) and~(b).}
It is enough to prove that for every $\s \in \cX \setminus \{\ee,\oo\}$
\[
\Phi(\s,\{\ee,\oo\}) -H(\s) \leq  \lfloor K/2 \rfloor.
\]
Indeed, this claim, together with the equivalent characterization of $\tG$ given in Lemma~\ref{lem:GAequiv1}, proves simultaneously inequality~\eqref{eq:ndwopen} when $KL \equiv 0 \pmod 2$ and the strict inequality~\eqref{eq:ndwopenstrict} when $KL \equiv 1 \pmod 2$, since in this case $ \lfloor K/2 \rfloor < \lceil K/2 \rceil$. To prove such an inequality, we have to exhibit for every $\s \in \cX \setminus \{\ee,\oo\}$ a path $\o: \s \to \{\ee,\oo\}$ in $\cX$ such that $\Phi_{\o}=\max_{\h \in \o} H(\h) \leq H(\s)+\lfloor K/2 \rfloor$. \\
Let $b$ be the number of particles present in configuration $\s$ in the odd sites of the leftmost column of $\L$, i.e.
\[
b:=\sum_{v \in c_0 \cap V_o} \s(v).
\]
Every column in $\L$ has $\lfloor K/2 \rfloor$ odd sites, and hence $ 0 \leq b \leq \lfloor K/2 \rfloor$. Differently from the proof of Theorem~\ref{thm:toricel}(a), here the value of $b$ determines how the path $\o$ will be constructed. We distinguish two cases: (a) $b = \lfloor K/2 \rfloor$ and (b) $b < \lfloor K/2 \rfloor$.\\

(a) Assume that $b = \lfloor K/2 \rfloor$. In this case, we construct a path $\o: \s \to \oo$ by means of the reduction algorithm for open grids, choosing as initial configuration $\s$ and as target configuration $\oo$. The way this path is built guarantees that $\Phi_{\o} \leq H(\s) + 1$, which implies that
\[
\Phi(\s,\oo)-H(\s) = 1 \leq \lfloor K/2 \rfloor.
\]

(b) Assume that $b < \lfloor K/2 \rfloor $. In this case we create a path $\o: \s \to \ee$ as the concatenation of two shorter paths, $\o^{(1)}$ and $\o^{(2)}$, where $\o^{(1)}: \s \to \s'$, $\o^{(2)}: \s' \to \ee$ and $\s'$ is a suitable configuration which depends on $\s$ (see definition below). The reason why $\o$ is best described as concatenation of two shorter paths is the following: Since $b < \lfloor K/2 \rfloor$, the reduction algorithm can not be started directly from $\s$ and the path $\o^{(1)}$ indeed leads from $\s$ to $\s'$, which is a suitable configuration to initialize the reduction algorithm for open grids. The configuration $\s'$ differs from $\s$ only in the odd sites of the first column, that is
\[
\s'(v):=
\begin{cases}
\s(v) & \text{ if } v \in \L \setminus (c_0 \cap V_o),\\
0 & \text{ if } v \in c_0 \cap V_o.
\end{cases}
\]
The path $\o^{(1)}=(\o^{(1)}_{1},\dots,\o^{(1)}_{b+1})$, with $\o^{(1)}_{1}=\s$ and $\o^{(1)}_{b+1}=\s'$, can be constructed as follows. For $i=1,\dots,b$, at step $i$ we remove from configuration $\o^{(1)}_{i}$ the topmost particle in $c_0 \cap V_o$ increasing the energy by $1$ and obtaining in this way configuration $\o^{(1)}_{i+1}$,. Therefore the configuration $\s'$ is such that $H(\s')-H(\s) = b$ and 
\[
\Phi_{\o^{(1)}} = \max_{\h \in \o^{(1)}} H(\h) \leq H(\s)+b.
\]
The path $\o^{(2)}: \s' \to \ee$ is then constructed by means of the reduction algorithm for open grids described earlier, using $\s'$ as initial configuration and $\ee$ as target configuration. The reduction algorithm guarantees that
\[
\Phi_{\o^{(2)}} = \max_{\h \in \o^{(2)}} H(\h) \leq H(\s') + 1.
\]
The concatenation of the two paths $\o^{(1)}$ and $\o^{(2)}$ gives a path $\o: \s \to \ee$ which satisfies the inequality $\Phi_{\o} \leq H(\s)+ b + 1$ and therefore
\[
\Phi(\s,\ee)-H(\s) = b+1 \leq \lfloor K/2 \rfloor. \eqno \qed
\]

\begin{prop}[Lower bound for $\Phi(\ee,\oo)$]\label{prop:lowerog}
\[
\Phi(\ee,\oo) - H(\ee) \geq \lceil K/2 \rceil +1.
\]
\end{prop}
\begin{proof}
It is enough to show that in every path $\o: \ee \to \oo$ there is at least one configuration with energy wastage greater than or equal to $\lceil K/2 \rceil+1$. Take a path $\o=(\o_1,\dots, \o_n) \in \Omega_{\ee,\oo}$. Since $\ee$ does not have an odd bridge while $\oo$ does, at some point along the path $\o$ there must be a configuration $\o_{m^*}$ which is the first to display an odd bridge, horizontal or vertical, or both simultaneously. In symbols
\[
m^* := \min\{ m \leq n \st \exists \, i ~:~(\o_m)_{|r_i} = \oo_{|r_i} \quad \mathrm{or} \quad \exists \, j ~:~(\o_m)_{|c_j} = \oo_{|c_j}\}.
\]
Clearly $m^*>2$. We claim that $U(\o_{m^*-1}) \geq \lceil K/2 \rceil +1$ or $U(\o_{m^*-2}) \geq \lceil L/2 \rceil+1$. We distinguish the following three cases:
\begin{itemize}
\item[(a)] $\o_{m^*}$ displays an odd vertical bridge only;
\item[(b)] $\o_{m^*}$ displays an odd horizontal bridge only;
\item[(c)] $\o_{m^*}$ displays an odd cross.
\end{itemize}
These three cases cover all possibilities, since the addition of a single particle cannot create more than one bridge in each direction. Let $v^* \in \L$ be the unique site where configuration $\o_{m^*-1}$ and $\o_{m^*}$ differ.\\

For case (a), assume first that $v^*$ belong to the $i^*$-th horizontal stripe, i.e.~$v^* \in S_{i^*}$ for some $0 \leq i^*\leq \lfloor K/2 \rfloor -1$. By construction, $v^*$ must be an odd site and $\o_{m^*-1}(v^*) = 0$ and $\o_{m^*}(v^*)=1$ and thus $U^S_{i^*}(\o_{m^*-1}) \geq 1$. We claim that in fact
\[
U^S_{i^*}(\o_{m^*-1}) \geq 2.
\]
It is enough to show that $U^S_{i^*}(\o_{m^*-1}) \neq 1$. Suppose by contradiction that $U^S_{i^*}(\o_{m^*-1}) =1$, then it must be the case that $U^S_{i^*}(\o_{m^*}) =0$, due the addition of a particle in $v^*$, and by Lemma~\ref{lem:bw} the horizontal stripe $S_{i^*}$ must agree fully with $\oo$ ($\o_{m^*}\neq \ee$, since it has a particle residing in $v^*$ which is an odd site). This fact would imply that $\o_{m^*}$ has an odd horizontal bridge, which contradicts our assumption for case (a).\\

Assume instead that $K$ is odd and that $v^*$ does not belong to any horizontal stripe and belongs instead to the topmost row, i.e.~$v^* \in r_{K-1}$. By construction, $v^*$ must be an odd site and $\o_{m^*-1}(v^*) = 0$ and $\o_{m^*}(v^*)=1$ and thus $U_{K-1}(\o_{m^*-1}) \geq 1$. We claim that in fact 
\[
U_{K-1}(\o_{m^*-1}) \geq 2.
\]
It is enough to show that $U_{K-1}(\o_{m^*-1}) \neq 1$. Suppose by contradiction that $U_{K-1}(\o_{m^*-1}) =1$, then it must be $U_{K-1}(\o_{m^*}) =0$, due to the addition of a particle in $v^*$. By Lemma~\ref{lem:bw} $\o_{m^*}$ must agree fully with $\ee$ on this topmost row, but this cannot be the case since $\o_{m^*}$ has a particle residing in $v^*$ which is an odd site.\\

Moreover, we claim that the energy wastage in every horizontal stripe that does not contain site $v^*$ (and in the topmost row if $KL \equiv 1 \pmod 2$ and $v^* \not\in r_{K-1}$) is also greater than or equal to $1$. Indeed, configuration $\o_{m^*-1}$ cannot display any horizontal odd bridge (by definition of $i^*$) and neither a horizontal even bridge, since $\o_{m^*-1}(v^*+(1,0)) = 0$ and $\o_{m^*-1}(v^*+(-1,0)) = 0$. Therefore for every $i=1,\dots,\lfloor K/2 \rfloor$ such that $v^* \not\in S_j$ we have $(\o_{m^*})_{|S_i} \neq \oo_{|S_i}, \ee_{|S_i}$ and hence, by Lemma~\ref{lem:bw}
\[
U^S_i(\o_{m^*}) \geq 1.
\]
If $K$ is odd, then the topmost row $r_{K-1}$ cannot be a horizontal odd bridge (our assumption would be violated) and neither a horizontal even bridge (it would be impossible to obtain the horizontal odd bridge which $\o_{m^*}$ has in a single step, the minimum number of steps needed is two). Therefore, by Lemma~\ref{lem:bw},
\[
U_{K-1}(\o_{m^*-1}) \geq 1.
\]
There are three possible scenarios:
\begin{itemize}
\item $K$ even: There are $K/2-1$ horizontal stripes with positive energy wastage and $U^S_{i^*}(\o_{m^*-1}) \geq 2$;
\item $K$ odd and $v^* \not\in r_{K-1}$: There are $\lfloor K/2 \rfloor -2$ horizontal stripes with positive energy wastage, $U_{K-1}(\o_{m^*-1}) \geq 1$ and $U^S_{i^*}(\o_{m^*-1}) \geq 2$;
\item $K$ odd and $v^* \in r_{K-1}$: There are $\lfloor K/2 \rfloor -1$ horizontal stripes with positive energy wastage and $U_{K-1}(\o_{m^*-1}) \geq 2$.
\end{itemize}
In all three scenarios, by summing the energy wastage of the horizontal stripes (and possibly that of the topmost row) we obtain
\[
U(\o_{m^*-1}) \geq \lceil K/2 \rceil +1.
\]

For case (b) we can argue in a similar way, but interchanging the roles of rows and columns, and obtain that 
\[
U(\o_{m^*-1}) \geq \lceil L/2 \rceil +1 \geq \lceil K/2 \rceil +1.
\]

For case (c), the vertical and horizontal odd bridges that $\o_{m^*}$ has must necessarily meet in the odd site $v^*$. Having an odd cross, $\o_{m^*}$ cannot display any horizontal or vertical even bridge. Consider the previous configuration $\o_{m^*-1}$ along the path $\o$, which can be obtained from $\o_{m^*}$ by removing the particle in $v^*$. From these considerations and from the definition of $m^*$ it follows that $\o_{m^*-1}$ has no vertical bridge (neither odd or even) and thus, by Lemma~\ref{lem:bw}, it has energy wastage at least one in each of the $\lfloor L/2 \rfloor$ vertical stripes and possibly in the leftmost column, if $L$ is odd. In both cases, we have
\[
U(\o_{m^*-1}) \geq \lceil L/2 \rceil.
\]
If there is at least one column in which $\o_{m^*-1}$ has energy wastage strictly greater than one, then the proof is concluded, since
\[
U(\o_{m^*-1}) \geq \lceil L/2 \rceil+1 \geq \lceil K/2 \rceil+1.
\]
Consider now the other scenario, in which the configuration $\o_{m^*-1}$ has energy wastage exactly one in every vertical stripe (and possibly in the leftmost column, if $L$ is odd), which means $U(\o_{m^*-1}) = \lceil L/2 \rceil$. Consider its predecessor in the path $\o$, namely the configuration $\o_{m^*-2}$. We claim that 
\[
U(\o_{m^*-2}) = \lceil L/2 \rceil +1.
\]
Indeed, by construction, configuration $\o_{m^*-2}$ must differ in exactly one site from $\o_{m^*-1}$ and therefore 
\[
U(\o_{m^*-2})=U(\o_{m^*-1})\pm1.
\]
Consider the case where $U(\o_{m^*-2}) = U(\o_{m^*-1})-1 = \lceil L/2 \rceil -1$. In this case the configuration $\o_{m^*-2}$ must have a zero-energy-wastage vertical stripe and by Lemma~\ref{lem:bw} it would be a vertical double bridge. If it was a vertical odd double bridge, the definition of $m^*$ would be violated. If it was an even vertical double bridge, it would be impossible to obtain the horizontal odd bridge (which $\o_{m^*}$ has) in just two single-site updates, since three is the minimum number of single-site updates needed. Therefore
\[
U(\o_{m^*-2})=U(\o_{m^*-1}) + 1 = \lceil L/2 \rceil +1. \eqno \qed
\]
\end{proof}

\begin{prop}[Reference path]\label{prop:refpatho}
There exists a path $\o^*:\ee \to \oo$ in $\cX_{G_{K,L}}$ such that 
\[
\Phi_{\o^*} - H(\ee) = \lceil K/2 \rceil +1.
%\max_{\s \in \o} U(\s) = \min \{\lceil K/2 \rceil, \lceil L/2 \rceil  \}+1.
\]
\end{prop}
\begin{proof}
We describe just briefly how the reference path $\o^*$ is constructed, since it is very similar to the one given in the proof of Proposition~\ref{prop:refpatht}. Also in this case, the path $\o^*$ is the concatenation of two shorter paths, $\o^{(1)}$ and $\o^{(2)}$, where $\o^{(1)}: \ee \to \s^*$ and $\o^{(2)}: \s^* \to \oo$, where $\s^*$ is the configuration that differs from $\ee$ only in the even sites of the leftmost column:
\[
\s^*(v):=
\begin{cases}
\ee(v) & \text{ if } v \in \L \setminus c_0,\\
0 & \text{ if } v \in c_0.
\end{cases}
\]
The path $\o^{(1)}$ consists of $\lceil K/2 \rceil$ steps, at each of which we remove the first particle in $c_0 \cap V_e$ in lexicographic order from the previous configuration. The last configuration is precisely $\s^*$, which has energy $H(\s^*)=H(\ee)+ \lceil K/2 \rceil$, and, trivially, $\Phi_{\o^{(1)}} = H(\ee)+ \lceil K/2 \rceil$. The second path $\o^{(2)}: \s^* \to \oo$ is then constructed by means of the reduction algorithm, which can be used since configuration $\s^*$ is a suitable initial configuration for it, satisfying condition~\eqref{eq:ogracond}. The algorithm guarantees that $\Phi_{\o^{(2)}} = H(\s^*) +1$ and thus the concatenation of the two paths $\o^{(1)}$ and $\o^{(2)}$ yields a path $\o^*$ with $\Phi_{\o^*} = \max_{\h \in \o} H(\h) = H(\ee) + \lceil K/2 \rceil +1$ as desired. \qed
\end{proof}

\subsection{Energy landscape analysis for cylindrical grids (Proof of Theorem~\ref{thm:cycleel})}
\label{sub53}
In this subsection we briefly describe how to proceed to prove Theorem~\ref{thm:cycleel}. The cylindrical grid $C_{K,L}$ is a hybrid between the toric grid and the open grid, since the columns of $C_{K,L}$ have the same structure as the columns of the toric grid $T_{K,L}$, while the horizontal stripes of $C_{K,L}$ enjoy the same structural properties of those of the open grid $G_{K,L}$. Along the lines of Lemmas~\ref{lem:rw} and~\ref{lem:bw} we can prove that the only columns with zero energy wastage are vertical bridges and the only horizontal stripes with zero energy wastage are horizontal double bridges.

In order to prove that 
\[
\Phi(\ee,\oo)-H(\ee) \geq \min\{K/2, L\} +1,
\]
one can argue in a similar way as was done for the other two types of grids. Also for the cylindrical grid, in any path $\o: \ee \to \oo$ there must be a configuration $\o_{m^*}$ which is the first to display a horizontal odd bridge or a vertical odd bridge or both simultaneously, i.e.
\[
m^* := \min\{ m \leq n \st \exists \, i ~:~ (\o_m)_{|r_i} = \oo_{|r_i} \quad \mathrm{or} \quad \exists \, j ~:~ (\o_m)_{|c_j} = \oo_{|c_j}\}.
\]
One can prove that 
\[
\max\{ U(\o_{m^*-1}),U(\o_{m^*-2})\} \geq \min\{K/2, L\} +1.
\]
We distinguish two cases, depending on whether $K/2 \geq L$ or $K/2 < L$. In these two cases, the proof can be obtained by studying the energy wastage either in the columns or in the horizontal stripes, in the same spirit as for the toric and open grids in Subsections~\ref{sub51} and~\ref{sub52}, respectively. Moreover, depending on whether $K/2 \geq L$ or $K/2 < L$, we can take the reference path $\o^*$ to be the same as in Subsections~\ref{sub51} and~\ref{sub52}, respectively. Lastly, one can show that
\[
\tG(\cX \setminus \{\ee,\oo\}) \leq \min \{ K/2, L\},
\]
by exploiting what has been done in Subsection~\ref{sub51}, if $K/2 \geq L$, and the strategy adopted in Subsection~\ref{sub52}, otherwise.

%%%%%%%%%%%%%%%%%%%%%%%%%%%%%%%%%%%%%%%%%%%%%%%%%%%%%%%%%%%%%%%%%%%%%%%%%%%%%%%%%%%%%%%%%%%%%%%%%%%%%%%%%%%%%%%%%%%%%%%%%%%%%%%%%%%%%%%%%%%%%%%%%%%%%%%%%%%%%%%%%%%%%%%%%%%%%%%%%%%%%%%%%%%%%%%%%%%%%%%%%%%%%%%%%%%%%%%%%%%%%%%%
%%%%%%%%%%%%%%%%%%%%%%%%%%%%%%%%%%%%%%%%%%%%%%%%%%%%%%%%%%%%%%%%%%%%%%%%%%%%%%%%%%%%%%%%%%%%%%%%%%%%%%%%%%%%%NEW%SECTION%%%%%%%%%%%%%%%%%%%%%%%%%%%%%%%%%%%%%%%%%%%%%%%%%%%%%%%%%%%%%%%%%%%%%%%%%%%%%%%%%%%%%%%%%%%%%%%%%%%%%%%%
%%%%%%%%%%%%%%%%%%%%%%%%%%%%%%%%%%%%%%%%%%%%%%%%%%%%%%%%%%%%%%%%%%%%%%%%%%%%%%%%%%%%%%%%%%%%%%%%%%%%%%%%%%%%%%%%%%%%%%%%%%%%%%%%%%%%%%%%%%%%%%%%%%%%%%%%%%%%%%%%%%%%%%%%%%%%%%%%%%%%%%%%%%%%%%%%%%%%%%%%%%%%%%%%%%%%%%%%%%%%%%%%

\section{Conclusions}
\label{sec6}
We have studied the first hitting times between maximum-occupancy configurations and mixing times for the hard-core interaction of particles on grid graphs. In order to do so, we extended the framework~\cite{MNOS04} for reversible Metropolis Markov chains. We expect that similar results for the first hitting time $\tha$ with a general initial state $x$ and target subset $A$ can be proved for irreversible Markov chains that satisfy the Friedlin-Wentzell condition~\eqref{eq:fw}. Furthermore, we developed a novel combinatorial method for grid graphs, valid for various boundary conditions, which shows that the energy landscape corresponding to hard-core dynamics on grid graphs has no deep wells and yields the minimum energy barrier between the two chessboard configurations $\ee$ and $\oo$. We obtained in this way results for the asymptotic behavior of the first hitting time $\teo$ in the low-temperature regime. We expect that our combinatorial approach can be exploited to prove similar results for other graphs which can be embedded in a grid graph (e.g.~triangular or hexagonal lattice) or for the hard-core model where there are two or more types of particles and the hard-core constraints exist only between particles of different type. The study of the critical configurations and of the minimal gates along the transition from $\ee$ to $\oo$ was beyond the scope of this paper and will be the focus of future work. 

\section*{Acknowledgments}
A. Zocca wishes to thank J. Sohier for the helpful discussions. This work was financially supported by The Netherlands Organization for Scientific Research (NWO) through the TOP-GO grant 613.001.012.

\bibliographystyle{plain}
\bibliography{HCG.bbl}
\end{document}